\documentclass[12pt]{amsart}
\usepackage[hmargin=1in,vmargin=1.25in]{geometry}
\usepackage{float}
\usepackage[utf8]{inputenc}
\usepackage{amsmath}
\usepackage{amsthm}
\usepackage{amssymb}
\usepackage{amsfonts}
\usepackage{enumerate}
\usepackage{tikz}
\usepackage{tkz-euclide}
\usepackage{ytableau}
\usepackage{quiver}
\usepackage{longtable}
\usepackage{mathtools}

\newcommand{\spa}{\hspace{.3mm}}
\DeclareMathOperator{\dote}{\ensuremath{\text{\spa-\spa-\spa-}}}

\theoremstyle{plain}
\newtheorem{theorem}{Theorem}[section]
\newtheorem{lemma}[theorem]{Lemma}
\newtheorem{corollary}[theorem]{Corollary}

\newtheorem{conjecture}[theorem]{Conjecture}

\theoremstyle{definition}
\newtheorem{defin}[theorem]{Definition}

\newtheorem{remark}[theorem]{Remark}
\newtheorem{exmp}[theorem]{Example}
\newtheorem{falseconjecture}[theorem]{False Conjecture}

\def\bx{{{\mathbf x}}}

\def\bz{{{\mathbf z}}}
\def\bv{{{\mathbf v}}}
\def\bw{{{\mathbf {w}}}}
\def\bu{{{\mathbf {u}}}}

\def\bm{{{\mathbf {m}}}}
\def\ba{{{\mathbf {a}}}}
\def\bb{{{\mathbf {b}}}}
\def\F{{{\rm F}}}

\def\K{{{\rm K}}}

\def\colword{{{\rm colword}}}

\def\tab{{{ \text {tab} }}}

\ytableausetup{boxsize=1.5em}

\newcommand{\inc}{\text{\rm inc}}
\newcommand{\inv}{\text{\rm inv}}
\newcommand{\SSYT}{\text{\rm SSYT}}
\newcommand{\SYT}{\text{\rm SYT}}
\newcommand{\T}{\text{\rm T}}
\newcommand{\ST}{\text{\rm ST}}
\newcommand{\IYT}{\text{\rm IT}}
\newcommand{\keySSYT}{\text{\rm keyT}}
\newcommand{\keySYT}{\text{\rm keyST}}
\newcommand{\keyIYT}{\text{\rm keyIT}}
\newcommand{\decSSYT}{\text{\rm powT}}
\newcommand{\decSYT}{\text{\rm powST}}
\newcommand{\decIYT}{\text{\rm powIT}}
\newcommand{\posSYT}{\text{\rm HikSYT}}
\newcommand{\decArray}{\text{\rm powArray}}
\newcommand{\injDecArray}{\text{\rm BPA}}
\newcommand{\strongSSYT}{\text{\rm strongT}}
\newcommand{\strongSYT}{\text{\rm strongST}}
\newcommand{\strongIYT}{\text{\rm strongIT}}

\newcommand{\col}{\text{\rm col}}
\newcommand{\row}{\text{\rm row}}
\newcommand{\threeone}{\ensuremath{(\mathbf{3}+\mathbf{1})}}
\newcommand{\twotwo}{\ensuremath{(\mathbf{2} + \mathbf{2})}}

\newcommand{\eval}{\text{\rm Eval}}
\newcommand{\partEval}{\text{\rm partialEval}}
\newcommand{\sort}{\text{\rm sort}}
\newcommand{\tabl}{\text{\rm tab}}

\newcommand{\mult}{\text{\rm mult}}
\newcommand{\peak}{\text{\rm peak}}
\newcommand{\Peak}{\text{\rm Peak}}

\newcommand{\prob}{\text{\rm prob}}

\newcommand{\slot}{\text{\rm slot}}
\newcommand{\ej}{\text{\rm ej}}
\newcommand{\Ej}{\text{\rm Ej}}
\def\by{{{\mathbf {y}}}}
\newcommand{\SK}{\text{ \rm SK}}

\numberwithin{equation}{section}

\title[Lower Bounds for Chromatic Symmetric Functions]{Toward Lower Bounds for\\ Chromatic Symmetric Functions in the elementary basis}

   \author{Isaiah Siegl}

   \address
   {Dept.\ of Mathematics\\
    University of Washington\\
    Seattle, WA}
   \date{\today}
   \email{isaiahsiegl@gmail.com}

\begin{document}

\begin{abstract}
Tatsuyuki Hikita recently proved the Stanley--Stembridge conjecture using probabilistic methods, showing that the chromatic symmetric functions of unit interval graphs are $e$-positive. Finding a combinatorial interpretation for these $e$-coefficients remains a major open problem. One approach is to look for combinatorial interpretations which are subsets of Gasharov's $P$-tableaux. Towards this goal, we introduce sets of \emph{strong} and \emph{powerful} $P$-tableaux, and use them to find combinatorial interpretations for various $e$-coefficients of the chromatic symmetric function $X_{\inc(P)}(\mathbf{x}, q)$. We conjecture that the set of strong $P$-tableaux gives a lower bound for the $e$-coefficients of $X_{\inc(P)}(\mathbf{x}, q)$.
Additionally, we show that strong $P$-tableaux and the Shareshian--Wachs inversion statistic appear naturally in the proof of Hikita's result.
\end{abstract}%

\maketitle
\tableofcontents
\section{Introduction}
In 1995, Stanley \cite{Stanleychromatic} introduced the \emph{chromatic symmetric function} $X_G(\mathbf{x})= \sum_{\kappa} \prod_{v \in V} x_{\kappa(v)}$ of a graph $G = (V,E)$, where the sum is over all proper colorings of $G$. In the years since, the chromatic symmetric function has been a well-studied graph invariant \cite{APdMOZcsfTree, FMWZ, CCcsf, LScsfTree, WYZcsfTree}. The chromatic symmetric function is of particular interest when $G$ is the incomparability graph of a \threeone -free poset. For example, $X_{\inc(P)}(\mathbf{x})$ is Schur-positive when $P$ is \threeone -free, by work of Haiman \cite{Himmanant} and Gasharov \cite{Gasharov}. Stanley observed that in the case of a \threeone -free poset $P$, the coefficients of $X_{\inc(P)}(\mathbf{x}) = \sum_{\lambda \vdash |P|} c_{\lambda}^P e_{\lambda}(\mathbf{x})$ in the basis of elementary symmetric functions are \emph{monomial immanants} of certain matrices, which he previously studied with Stembridge \cite{StanleyStembridge}. Combining work of \cite{Stanleychromatic} and \cite{StanleyStembridge}, Stanley--Stembridge conjectured that the chromatic symmetric functions of incomparability graphs of \threeone -free posets expand positively in the basis of elementary symmetric functions. Guay-Paquet \cite{GPchromatic} reduced the Stanley--Stembridge conjecture to the special case of incomparability graphs of natural unit interval orders. In the case of natural unit interval orders, Shareshian--Wachs \cite{SWchromatic} conjectured and Brosnan--Chow \cite{BrosnanChow} and Guay-Paquet \cite{guaypaquetCohom} proved that the chromatic symmetric function is the Frobenius characteristic of the dot action on the cohomology ring of Hessenberg varieties, giving a representation theoretic realization of Schur-positivity.
Kato \cite{KatoCSF} discovered an alternative geometric interpretation of the chromatic symmetric function, connecting to his previous work on Catalan functions \cite{KatoCatalan}.
Partial results towards the Stanley--Stembridge conjecture have been obtained using a variety of methods \cite{ChoHong, cho2019positivity, Dahlberg, dahlberg2018lollipop, gebhard2001chromatic, harada2019cohomology, Isaiah, TomChromatic}.
Recently, Hikita \cite{HikitaChromatic} proved the Stanley--Stembridge conjecture using probabilistic techniques.

\begin{theorem}[Hikita's Theorem]
\label{SSconj}
\cite{HikitaChromatic}
Let $P$ be a \threeone -free poset, and let $\inc(P)$ denote its incomparability graph. Then $X_{\inc(P)}(\mathbf{x}) = \sum_{\lambda \vdash |P|} c_\lambda^P e_{\lambda}(\mathbf{x})$ is $e$-positive. In other words, each elementary basis coefficient $c_\lambda^P$ is a non-negative integer.
\end{theorem}

Shareshian--Wachs \cite{SWchromatic} introduced an \emph{inversion} statistic $\inv_G(\kappa)$ on proper colorings of graphs $G$ with vertex set $[n] = \{1,2,...,n\}$, given by
\begin{equation}
\inv_G(\kappa) = \#\{ \{i,j\} \in E(G) : i<j, \ \kappa(i) > \kappa(j)\}.
\end{equation}
Shareshian--Wachs \cite{SWchromatic} then used this inversion statistic to define the \emph{chromatic quasisymmetric function} $X_G(\mathbf{x}, q)$ by
\begin{equation}
X_{G}(\mathbf{x}, q) = \sum_{\kappa} q^{\inv_G(\kappa)} x_{\kappa(1)}x_{\kappa(2)} \cdots x_{\kappa(n)},
\end{equation} 
where the sum is over all proper colorings $\kappa: [n] \to \mathbb{Z}_{>0}$. In the case when $G$ is the incomparability graph of a \emph{natural unit interval order}, $X_{\inc(P)}(\mathbf{x}, q)$ is in fact a symmetric function \cite[Theorem 4.5]{SWchromatic}. The chromatic quasisymmetric functions for incomparability graphs of natural unit interval orders are closely related to the $q$-statistic from the work of Haiman \cite{Himmanant} on immanants in the setting of Hecke algebras. The connection between monomial Hecke algebra immanants in \cite{Himmanant} and the chromatic quasisymmetric functions of \cite{SWchromatic}  was established by Clearman--Hyatt--Shelton--Skandera \cite{CHSS}.
We write $c_\lambda^P$ for the $e$-coefficients of $X_{\inc(P)}(\mathbf{x})$ and $c_{\lambda}^P(q)$ for the \emph{polynomial} $e$-coefficients of $X_{\inc(P)}(\mathbf{x}, q)$ when $X_{\inc(P)}(\mathbf{x}, q)$ is a symmetric function. When $P$ is clear from context, we will write $c_\lambda$ and $c_\lambda(q)$.
 Shareshian--Wachs made the following refinement of the Stanley--Stembridge conjecture.
\begin{conjecture}\cite[Conjecture 1.3]{SWchromatic}
\label{SWePosConj}
If $P$ is a natural unit interval order, then $X_{\inc(P)}(\mathbf{x}, q) = \sum_{\lambda \vdash n} c_\lambda^P(q) e_\lambda(\mathbf{x})$ is $e$-positive. In other words, the polynomials $c_{\lambda}^P(q)$ have non-negative integer coefficients.
\end{conjecture} 

Each natural unit interval order on $[n]$ is indexed by a \emph{reverse Hessenberg function} $\bm.$ Write $P_\bm$ for the natural unit interval order associated to $\bm$. 
To prove the Stanley--Stembridge conjecture, Hikita introduced a probability distribution $\prob_\bm: \SYT_n \to \mathbb{Q}(q)$ on standard Young tableaux for each reverse Hessenberg function $\bm$. Hikita then showed that $c_{\lambda}^{P_\bm}(\tau) > 0$ for all $\tau > 0$ if and only if $\prob_\bm(T; q) \neq 0$ for some $T \in \SYT(\lambda)$. Define the set of \emph{$\bm$-Hikita tableaux of shape $\lambda$} by 
\begin{equation}
\posSYT(\bm, \lambda) = \{T \in \SYT(\lambda) \mid \prob_\bm(T; q) \neq 0\}.
\end{equation}
We recall the necessary details of Hikita's proof in Section~\ref{subsec Hik prelim}.
Hikita's proof of Theorem \ref{SSconj}, and the subsequent proofs of Theorem~\ref{SSconj} given by Griffin--Mellit--Romero--Weigl--Wen \cite{GMRWW} and Huh--Hwang--Kim--Kim--Oh \cite{HHKKO}, leave several questions unanswered. 
In particular, Conjecture \ref{SWePosConj} remains unresolved, and no combinatorial interpretation of the $e$-coefficients of $X_{\inc(P)}(\mathbf{x})$ is known. The problem of finding a combinatorial interpretation of the elementary basis coefficients $c_{\lambda}^P$ and $c_{\lambda}^P(q)$ is the focus of the current paper.
 
%

Our search for an interpretation of the $e$-coefficients of $X_{\inc(P)}(\mathbf{x})$ builds on previous work on the Schur expansion coefficients of $X_{\inc(P)}(\mathbf{x})$.
In 1996, Gasharov \cite{Gasharov} introduced $P$-tableaux $\T_P(\lambda)$ and standard $P$-tableaux $\ST_P(\lambda)$ as a generalization of Young tableaux and used them to give a combinatorial interpretation for the Schur coefficients of $X_{\inc(P)}(\mathbf{x})$. 
\noindent
Shareshian--Wachs \cite{SWchromatic} also introduced an inversion statistic $\inv_P$ on $P$-tableaux when $P$ is a natural unit interval order and proved an analogue of Gasharov's theorem for $X_{\inc(P)}(\mathbf{x},q)$. We recall the details of these results in Section~2.2.

By the triangularity of the change of basis transformation between the Schur and $e$ bases, $P$-tableaux give an upper bound on $c_\lambda^P$. Hence, we may expect to find combinatorial interpretations of $c_\lambda^P$ as subsets of $P$-tableaux.
In this paper we introduce \emph{strong} $P$-tableaux, $\strongSYT_P(\lambda)$, and \emph{powerful} $P$-tableaux, $\decSYT_P(\lambda),$ for \threeone -free posets $P$ such that
\begin{equation}
\strongSYT_P(\lambda) \subseteq \decSYT_P(\lambda) \subseteq \ST_P(\lambda).
\end{equation}
Using these $P$-tableaux, we obtain combinatorial interpretations of $c_\lambda^P$ in various cases.
Furthermore, we conjecture that $\strongSYT_P(\lambda)$ determines a lower bound on the $e$-coefficients in the following way. Additionally, we conjecture that $c_\lambda^P = 0$ whenever $\strongSYT_P(\lambda) = \emptyset$.
\begin{conjecture}
\label{boundsConjecture}
For a \threeone -free poset $P$, 
\begin{equation}
\#\strongSYT_P(\lambda) \leq c_\lambda^P.
\end{equation}
\end{conjecture}
\begin{conjecture}
\label{nonzeroCoeffConj}
For a \threeone -free poset $P$, if $\strongSYT_P(\lambda) = \emptyset$, then $c_\lambda^P = 0$.
\end{conjecture}

When $P$ is a natural unit interval order, we conjecture that $\strongSYT_P(\lambda)$ gives a stronger lower bound on the coefficients of $c_{\lambda}^P(q)$. 
Conjectures \ref{undercountqConj} and \ref{nonzeroCoeffConj} have been verified for all natural unit interval orders with at most 10 elements via ChatGPT and at most 8 elements with the author's personal computer.
\begin{conjecture}
\label{undercountqConj}
For a natural unit interval order $P$ with $n$ elements and $\lambda \vdash n$, the polynomial
\begin{equation}
c_\lambda^P(q) - \sum_{T\, \in\, \strongSYT_P(\lambda)} q^{\inv_P(T)}
\end{equation}
has non-negative integer coefficients.
\end{conjecture}

The main results of this paper support these conjectures. 
For a reverse Hessenberg function $\bm$, let $P_\bm$ be the associated natural unit interval order.
In this setting, we show that strong $P_\bm$-tableaux appear naturally in Hikita's proof of the Stanley--Stembridge conjecture. 
In particular, we prove a converse of Conjecture~\ref{nonzeroCoeffConj}, and connect Hikita's proof to the Shareshian--Wachs inversion statistic.
\begin{theorem}
\label{nonzeroCoeffThm}
Let $P$ be the natural unit interval order associated to a reverse Hessenberg function $\bm$. Then $\posSYT(\bm, \lambda) \subseteq \strongSYT_{P}(\lambda)$. Hence, if $c_\lambda^{P} > 0,$ then $\strongSYT_{P}(\lambda) \neq \emptyset$.
\end{theorem}

\begin{theorem}
\label{thm Hik inv}
Let $P$ be the natural unit interval order associated to a reverse Hessenberg function $\bm$. Then for each $T \in \posSYT(\bm, \lambda)$, there exists a function $h_T(q)$ which is a quotient of products of $q$-integers such that $0 < h_T(\tau) \leq 1$ for all real values $\tau \geq 0$, and 
\begin{equation}
 \frac{c_\lambda^{P_\bm}(q)}{[\lambda]_q!} = \sum_{T\, \in\, \posSYT(\bm, \lambda)} q^{\inv_{P_\bm}(T)} h_T(q).
\end{equation}
\end{theorem}

Using strong and powerful $P$-tableaux, we obtain combinatorial interpretations for various special cases of $c_\lambda^P(q)$. In particular, we use the \emph{greedy partition} $\lambda^{gr}(P)$ of a \threeone -free poset $P$ defined by Matherne--Morales--Selover \cite{MMSlorentzian} to show that strong $P$-tableaux are a combinatorial interpretation of $c_\lambda^P(q)$ for certain partitions $\lambda$. We give a combinatorial interpretation of $c_\lambda^P(q)$ when $\lambda = (k, 2^a)$ and $P$ is a natural unit interval order. Additionally, we show that powerful $P$-tableaux are a combinatorial interpretation of $c_\lambda^P(q)$ when $\inc(P)$ is the path graph.

\begin{theorem}
\label{dominantPositivity}
Let $P$ be a natural unit interval order on $[n]$, and let $\lambda^{gr} = \lambda^{gr}(P)$.
Let $S$ be a subset of $[n]$ that contains no adjacent entries, and let $k_i$ be a positive integer for each $i \in S$. 
If $\lambda$ is a partition such that the conjugate partition $\lambda'$ 
satisfies $\lambda'_i = \lambda_i^{gr} - k_i$ if $i \in S$, $\lambda'_i = \lambda^{gr}_i + k_{i-1}$ if $i-1 \in S$, and $\lambda'_i = \lambda_{i}^{gr}$ otherwise, 
then
\begin{equation}
c_{\lambda}^P(q) = \sum_{T\, \in\, \strongSYT_P(\lambda)} q^{\inv_P(T)}.
\end{equation}
\end{theorem}

\begin{theorem}
\label{thm k2positivity}
Let $P$ be a natural unit interval order on $[n]$, and let $\lambda = (k,2^a)\vdash n$. Then there exists a family of $P$-tableaux $\K_P(\lambda)$ defined in Definition~\ref{def KP} such that $\strongSYT_P(\lambda) \subseteq \K_P(\lambda) \subseteq \decSYT_P(\lambda),$ and
\begin{equation}
c_\lambda^P(q) = \sum_{T\, \in\, \K_P(\lambda)} q^{\inv_P(T)}.
\end{equation}
\end{theorem}

\begin{theorem}
\label{dumbbellqPositivityTheorem}
Let $P_n$ be the natural unit interval order on $[n]$ such that $\inc(P_n)$ is a path. Then for all partitions $\lambda \vdash n$,
\begin{equation}
c_\lambda^{P_n}(q) = \sum_{T\, \in\, \decSYT_{P_n}(\lambda)} q^{\inv_{P_n}(T)}.
\end{equation}
\end{theorem}

The paper is organized as follows. 
In Section 2, we recall previous results and definitions necessary for the rest of the paper. 
In Section 3, we define strong and powerful $P$-tableaux building on the work of \cite{BEPS} and \cite{Hwang}. 
In particular, in Lemma~\ref{twoColDecLemma} we show that $\#\strongSYT_P(\lambda) = c_\lambda^P = \#\decSYT_P(\lambda)$ for $\lambda \vdash n$ a two-column shape, and in Lemma~\ref{hookDecomp}, we show that $c_\lambda^P = \#\decSYT_P(\lambda)$ for $\lambda \vdash n$ a hook shape.
In Section 4, we connect our work to the work of Hikita \cite{HikitaChromatic}, proving Theorem~\ref{nonzeroCoeffThm} and Theorem~\ref{thm Hik inv}. Conjecture~\ref{conj strong implies Hik} is a reformulation of Conjecture~\ref{nonzeroCoeffConj} for natural unit interval orders.
In Section \ref{dominanceSection}, we use concatenation of $P$-tableaux to prove $e$-positivity results related to the greedy partition $\lambda^{gr}(P),$ including Theorem \ref{dominantPositivity}.
As a corollary, we obtain new positivity results in the context of the noncommutative symmetric functions studied in \cite{Hwang} and \cite{BEPS}.
In Section~\ref{sectionk2}, we consider partitions $\lambda$ of the form $(k, 2^a)$ and prove Theorem~\ref{thm k2positivity}. Results of Section~\ref{sectionk2} are given in terms of noncommutative symmetric functions.
In Section \ref{pathSection}, we use a formula of Shareshian--Wachs \cite{SWchromatic} to prove Theorem \ref{dumbbellqPositivityTheorem}.
Conjecture ~\ref{conj barbell powerful} states that $c_{\lambda}^P(q) = \sum_{T \in \decSYT_P(\lambda)} q^{\inv_P(T)}$ when $\inc(P)$ is two complete graphs connected by a path.
Appendix~\ref{append data} contains examples of strong and powerful $P$-tableaux and applications of results of this paper.

In a previous version of this paper, we conjectured that powerful $P$-tableaux give an upper bound for the polynomial $e$-coefficients of $c_\lambda^P(q)$ for natural unit interval orders $P$. The author is very grateful to an anonymous reviewer for pointing out a counterexample to this conjecture. Details of the counterexample are given in Appendix~\ref{append counterexamp}.
\section{Preliminaries}

We assume familiarity with posets and symmetric function theory. The necessary background can be found in \cite{EC1, St}. Some results are stated in terms of \emph{$q$-integers}. For a non-negative integer $n$, define the \emph{$q$-integer} $[n]_q$ to be the polynomial $[n]_q = \frac{1 - q^{n}}{1-q} = 1 + q + \dots + q^{n-1}$. Define the \emph{$q$-integer factorial} by $[n]_q! = [n]_q[n-1]_q \cdots [1]_q$.
For a partition $\lambda$, write $[\lambda]_q!$ for $\prod_{i} [\lambda_i]_q!$.

\subsection{Posets}

Let $P$ be a poset. For elements $a,b \in P$, we write $a <_P b$ if $a$ is less than $b$ in $P$.
\begin{defin}
Let $s$ and $t$ be positive integers. Write $\mathbf{s}$ for the total order on $s$ elements. A poset $P$ is \emph{$\mathbf{s}$-free} if $P$ has no induced subposet isomorphic to $\mathbf{s}$, and $P$ is \emph{$(\mathbf{s} + \mathbf{t})$-free} if $P$ has no induced subposet that is isomorphic to the disjoint union of $\mathbf{s}$ and $\mathbf{t}$.
\end{defin}

\begin{defin}
A poset $P$ on the set $[n]$ is a \emph{natural unit interval order} if there are real numbers $a_1 \leq a_2 \leq \ldots \leq a_n$ such that
\begin{equation}
i <_P j \iff a_{i} + 1 < a_j.
\end{equation}
\end{defin}

\begin{exmp}
\label{ex poset}
Let $ \mathbf{a} = (a_1,a_2,a_3,a_4,a_5) = (0, 0.5, 1.1, 1.4, 2.2)$. Then the natural unit interval order associated to $\mathbf{a}$ is the poset with Hasse diagram
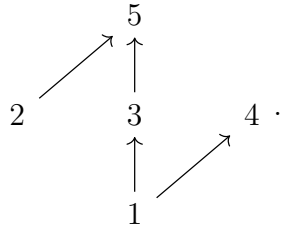
\[\begin{tikzcd}
	& 5 \\
	2 & 3 & 4 \\
	& 1
	\arrow[from=2-1, to=1-2]
	\arrow[from=2-2, to=1-2]
	\arrow[from=3-2, to=2-2]
	\arrow[from=3-2, to=2-3]
\end{tikzcd}.\]
\end{exmp}

We are primarily concerned with \threeone -free posets and \emph{natural unit interval orders}. 
Observe that if $P$ is a natural unit interval order and $i <_P j$, then $i < j$ in the usual order on $[n]$.
Natural unit interval orders are \threeone -free and  \twotwo -free, and every finite \threeone -free and \twotwo -free poset is isomorphic to a natural unit interval order. Throughout this paper, assume $P$ is a \threeone -free poset on vertex set $[n]$ with order relation $<_P$. When we need the more restrictive case of natural unit interval orders, we will specify that in the hypothesis.

\subsection{$P$-Tableaux}
Let $\lambda = (\lambda_1,\lambda_2,...,\lambda_l) \vdash n$ be a partition. The (English style) Ferrers diagram of shape $\lambda$ is the set of $n$ cells $\{(i,j) \in \mathbb{Z}_{>0}^2: j \le \lambda_i\},$ written using matrix coordinates and drawn as boxes in the plane such that cell $(i,j)$ is in row $i$ and column $j$. We often identify a partition $\lambda$ with its Ferrers diagram.  
We write $\ell(\lambda)$ for the number of nonzero parts of $\lambda$ and $\lambda'$ for the transpose of $\lambda$.

For a composition $\alpha = (\alpha_1,\alpha_2,...,\alpha_l) \vDash n,$ we define the \emph{row shape diagram} $\row(\alpha)$ to be the set of $n$ cells $\{(i,j) \in \mathbb{Z}_{>0}^2 : j \le \alpha_i\}$. Likewise, define the \emph{column shape diagram} $\col(\alpha)$ to be the set $\{(i,j) \in \mathbb{Z}_{>0}^2 : i \leq \alpha_j\}$. Unless specified otherwise, we identify a composition $\alpha$ with its row shape $\row(\alpha)$. For a composition $\alpha \vDash n$, let $\sort(\alpha)$ be the partition of $n$ obtained by sorting the entries of $\alpha$ in decreasing order. For a subset $\beta$ of $\mathbb{Z}^2_{>0}$, we refer to functions with domain $\beta$ as \emph{fillings} of $\beta$. For a filling $A$, we write $A_{i,j}$ for the entry in cell $(i,j)$

A semistandard Young tableau of shape $\lambda \vdash n$ is a filling of $\lambda$ with positive integers such that $T_{i, j} < T_{i+1, j}$ and $T_{i, j} \leq T_{i, j+1}$. 
We say the columns of $T$ are strictly increasing and rows of $T$ are weakly increasing to describe these properties.
Let $\SSYT(\lambda)$ be the set of semistandard Young tableaux of shape $\lambda$. A semistandard Young tableau $T \in \SSYT(\lambda)$ is a \emph{standard} Young tableau if the set of entries of $T$ is $\{1,2,...,n\}$. Let $\SYT(\lambda)$ be the set of standard Young tableaux of shape $\lambda$, and let $\SYT_n$ denote the set of all standard Young tableaux with $n$ cells. 

For \threeone -free posets, Gasharov \cite{Gasharov} introduced $P$-arrays and $P$-tableaux. He then used $P$-arrays to give a combinatorial interpretation for the coefficients of monomial symmetric functions in $X_{\inc(P)}(\mathbf{x})$.

\begin{defin}\cite{Gasharov}
For a column shape diagram $\col(\alpha)$, a \emph{$P$-array} of shape $\col(\alpha)$ is a function $A: \col(\alpha) \to P$ such that $A_{i,j} <_P A_{i+1, j}$. We say the columns of a $P$-array $A$ are \emph{increasing} in $P$. Write $E_P(\col(\alpha))$ for the set of $P$-arrays of shape $\col(\alpha)$. A $P$-array is \emph{injective} if each element of $P$ appears at most once and \emph{bijective} if each element of $P$ appears exactly once.
\end{defin}

\begin{exmp}
For the poset $P$ in Example~\ref{ex poset}, the following is a $P$-array of shape $\col(2,3,1,2)$.
\begin{equation}
\begin{ytableau}
3 & 1 & 4 & 1\\
5 & 3 & \none  & 4 \\
\none & 5 
\end{ytableau}
\end{equation}
\end{exmp}

\begin{theorem}\cite{Gasharov}
For a partition $\lambda \vdash n$, the coefficient of $m_\lambda(\mathbf{x})$ in $X_{\inc(P)}(\mathbf{x})$ is the number of injective $P$-arrays of shape $\lambda'$.
\end{theorem}

\begin{defin}\cite{Gasharov}
For a partition $\lambda$, a $P$-array $T \in E_P(\lambda)$ is a \emph{$P$-tableau of shape $\lambda$} if for all cells $(i,j) \in \lambda$, $T_{i,j} \not>_P T_{i,j+1}$. We say the rows of a $P$-tableau $T$ are \emph{non-descending}. We write $\T_P(\lambda)$ for the set of $P$-tableaux of shape $\lambda$, $\IYT_P(\lambda)$ for the set of injective $P$-tableaux of shape $\lambda$, and $\ST_P(\lambda)$ for the set of bijective $P$-tableaux of shape $\lambda$. We say elements of $\ST_P(\lambda)$ are \emph{standard} $P$-tableaux.
\end{defin}

Note that our definitions of $P$-arrays and $P$-tableaux are the transposes of the arrays and tableaux defined by Gasharov. We take this convention so that $P$-tableaux coincide with Young tableaux when $P$ is a total order.
Gasharov showed that $P$-tableaux give a combinatorial interpretation for the Schur coefficients of $X_{\inc(P)}(\mathbf{x})$. Shareshian--Wachs \cite{SWchromatic} then generalized Gasharov's theorem to the chromatic quasisymmetric function $X_{\inc(P)}(\mathbf{x}, q)$.
\begin{theorem}\cite{Gasharov}
For a \threeone -free poset $P$ and $\lambda \vdash n$, the coefficient of $s_{\lambda}(\mathbf{x})$ in $X_{\inc(P)}(\mathbf{x})$ is $\#\ST_P(\lambda')$.
\end{theorem}
\begin{theorem}\cite{SWchromatic}
Let $P$ be a natural unit interval order and $\lambda \vdash n$. Then the coefficient of $s_{\lambda}(\mathbf{x})$ in $X_{\inc(P)}(\mathbf{x}, q)$ is 
\begin{equation}
\sum_{T \in\, \ST_P(\lambda')} q^{\inv_P(T)}
\end{equation}
where
\begin{equation}
\label{eqn tab inv}
\inv_P(T) = \#\{(i < j) \mid i, j \text{ incomparable},\ i \text{ appears to the right of } j \text{ in } T\}.
\end{equation}
\end{theorem}

Let $A$ be a $P$-array with adjacent columns $C = v_1 <_P v_2 <_P ... <_P v_s$ and $D = w_1 <_P w_2 <_P ... <_P w_t$. As $P$ is \threeone -free, each entry of $C$ and $D$ will be incomparable to at most two of the entries in the other column. Furthermore, if an entry of one column is incomparable to two entries in the other column, the two entries will be adjacent in their column. Thus, the connected components of the induced subgraph of $\inc(P)$ in the entries of $C \cup D$ will be a disjoint union of paths and $4$-cycles. In the case when $P$ is a natural unit interval order, the connected components will be paths alternating between $C$ and $D$. 

Let $A$ be a $P$-array with columns $C_1, C_2,...,C_l$. We say a connected component of the induced subgraph of $\inc(P)$ on $C_i \cup C_{i+1}$ is a \emph{ladder} between columns $i$ and $i+1$.
A ladder is \emph{balanced} if it has the same number of elements in each column, \emph{left-unbalanced} if it has more elements in the left column, and \emph{right-unbalanced} if it has more elements in the right column. 
This definition is a slight generalization of \cite[Definition 3.7]{KPchromatic}.
For ladders $K, L$ between a pair of columns, we write $K <_P L$ if $x <_P y$ for all $x \in K$ and $y \in L$. This defines a total order on ladders between a given pair of adjacent columns.

\begin{defin}
Let $\alpha = (\alpha_1, \alpha_2,...,\alpha_l)$ be a composition. For an unbalanced ladder $K$ between columns $C_i$ and $C_{i+1}$ of a $P$-array $A \in E_P(\col(\alpha))$, performing a \emph{ladder swap} on $K$ exchanges the elements of $K$ in $C_i$ with the elements of $K$ in $C_{i+1}$.
Figure~\ref{fig ladder swap} shows a ladder swap.

\begin{lemma}
Let $A$ be a $P$-array of shape $\col(\alpha)$ with unbalanced ladder $K$ between columns $i$ and $i+1$.
Performing a ladder swap on $K$ produces a $P$-array of shape $\col(\tilde{\alpha})$ where $\tilde{\alpha}$ differs from $\alpha$ only in entries $i$ and $i+1$. 
Specifically, if $K$ is a left-unbalanced ladder, then $\tilde{\alpha}_i = \alpha_{i} + 1$ and $\tilde{\alpha}_{i+1} = \alpha_{i+1} - 1$, and if $K$ is a right-unbalanced ladder, then $\tilde{\alpha}_i = \alpha_i - 1$ and $\tilde{\alpha}_{i+1} = \alpha_{i+1} + 1$. 
\end{lemma}
\begin{proof}
The result follows from the definitions and the fact that the number of elements in each column in an unbalanced ladder differ by 1.
\end{proof}

\begin{figure}[h]
\caption{Ladder Swap}
\label{fig ladder swap}
\[\begin{tikzcd}
	\vdots & \vdots && \vdots & \vdots \\
	{c_1} & {d_1} && {d_1} & {c_1} \\
	{c_2} & {d_2} && {d_2} & {c_2} \\
	\vdots & \vdots & \leftrightarrow & \vdots & \vdots \\
	{c_k} & {d_k} && {d_k} & {c_k} \\
	{c_{k+1}} & \vdots && \vdots & {c_{k+1}} \\
	\vdots &&&& \vdots
	\arrow[from=1-1, to=2-1]
	\arrow[from=1-1, to=2-2]
	\arrow[from=1-2, to=2-1]
	\arrow[from=1-2, to=2-2]
	\arrow[from=1-4, to=2-4]
	\arrow[from=1-4, to=2-5]
	\arrow[from=1-5, to=2-4]
	\arrow[from=1-5, to=2-5]
	\arrow[dashed, no head, from=2-1, to=2-2]
	\arrow[from=2-1, to=3-1]
	\arrow[dashed, no head, from=2-2, to=3-1]
	\arrow[from=2-2, to=3-2]
	\arrow[dashed, no head, from=2-4, to=2-5]
	\arrow[from=2-4, to=3-4]
	\arrow[dashed, no head, from=2-4, to=3-5]
	\arrow[from=2-5, to=3-5]
	\arrow[dashed, no head, from=3-1, to=3-2]
	\arrow[from=3-1, to=4-1]
	\arrow[dashed, no head, from=3-2, to=4-1]
	\arrow[from=3-2, to=4-2]
	\arrow[dashed, no head, from=3-4, to=3-5]
	\arrow[from=3-4, to=4-4]
	\arrow[dashed, no head, from=3-4, to=4-5]
	\arrow[from=3-5, to=4-5]
	\arrow[from=4-1, to=5-1]
	\arrow[from=4-2, to=5-2]
	\arrow[from=4-4, to=5-4]
	\arrow[dashed, no head, from=4-4, to=5-5]
	\arrow[from=4-5, to=5-5]
	\arrow[dashed, no head, from=5-1, to=4-2]
	\arrow[dashed, no head, from=5-1, to=5-2]
	\arrow[from=5-1, to=6-1]
	\arrow[from=5-2, to=6-2]
	\arrow[from=5-2, to=7-1]
	\arrow[dashed, no head, from=5-4, to=5-5]
	\arrow[from=5-4, to=6-4]
	\arrow[dashed, no head, from=5-4, to=6-5]
	\arrow[from=5-4, to=7-5]
	\arrow[from=5-5, to=6-5]
	\arrow[dashed, no head, from=6-1, to=5-2]
	\arrow[from=6-1, to=6-2]
	\arrow[from=6-1, to=7-1]
	\arrow[from=6-5, to=6-4]
	\arrow[from=6-5, to=7-5]
\end{tikzcd}\]
\end{figure}
\end{defin}

%
%

\subsection{Noncommutative $P$-Symmetric Functions}
The main tools we will use to study positivity of $X_{\inc(P)}(\bx, q)$ are the noncommutative symmetric functions studied by Hwang \cite{Hwang} in the natural unit interval case and by Blasiak, Eriksson, Pylyavskyy, and the author \cite{BEPS} in the \threeone -free case. These are a noncommutative generalization of the polynomials introduced by Stanley \cite{StanleyChromatic1998} in 1998, following the noncommutative Schur function techniques developed by Fomin--Greene \cite{FG} and Blasiak--Fomin \cite{BF}.
\begin{defin}
\label{def IP}
 Recall that $P$ is a \threeone -free poset on $[n]$, and let $\mathcal{U} = \{u_1,u_2,...,u_n\}$ be a set of noncommuting indeterminates. Write $\mathbb{Q}\langle \mathcal{U} \rangle$ for the free noncommutative $\mathbb{Q}$-algebra over $\mathcal{U}$.
For $a,b \in P$, write $a \dote_P b$ to mean $a$ and $b$ are equal or incomparable in $P$.
Let $I_P$ denote the ideal of $\mathbb{Q}\langle \mathcal{U} \rangle$ generated by
\begin{align}
u_a u_b - &u_b u_a & &(\text{for all } a <_P b) \\
u_c u_a u_b - &u_b u_c u_a & &(\text{for all } a \dote_P b \dote_P c \text{ and } a <_P c).
\end{align}
Write $\mathcal{R}_P$ for the quotient $\mathbb{Q}\langle \mathcal{U} \rangle/I_P$.
We write $\bu_\bw = u_{w_1}u_{w_2}\dots u_{w_k}$ for a word $\bw = w_1w_2\dots w_k$ with entries in $P$. For words $\bw,\bv$ with entries in $P$, we write $\bw \equiv \bv \mod I_P$ to mean $\bu_{\bw} \equiv \bu_{\bv} \mod I_P$.
\end{defin}
\begin{defin}
For integers $k \geq 1$ and partitions $\lambda = (\lambda_1,\lambda_2,...,\lambda_l)$,
define the \emph{elementary $P$-symmetric functions} in $\mathcal{R}_P$ by 
\begin{equation}
e_k^P(\mathbf{u}) := \sum_{i_1 <_P\, i_2 <_P ... <_P\, i_k} u_{i_k}u_{i_{k-1}}\cdots u_{i_1},
\end{equation}
\begin{equation}
e_{\lambda}^P(\mathbf{u}) := e_{\lambda_1}^P(\mathbf{u})e_{\lambda_2}^P(\mathbf{u}) \cdots e_{\lambda_l}^P(\mathbf{u}).
\end{equation}
\end{defin}
It is useful to think of $e_\lambda^P(\mathbf{u})$ as a generating function for the set of $P$-arrays $E_P(\lambda')$. 
Specifically, we have
\begin{equation}
e_{\lambda}^P(\mathbf{u}) = \sum_{A \in E_P(\lambda')} \mathbf{u}_{\colword(A)},
\end{equation}
where $\colword(A)$ is the word obtained by reading the entries of $A$ bottom-to-top left-to-right.
\begin{defin}
For a symmetric function $f(\mathbf{x}) = \sum_{\lambda} a_\lambda e_\lambda(\mathbf{x})$, define the $P$-analogue $f^P(\mathbf{u})$ by $f^P(\mathbf{u}) = \sum_{\lambda} a_\lambda e^P_\lambda(\mathbf{u})$.
\end{defin}
In particular, we have $P$-analogues of the Schur and monomial symmetric functions $\{s_\lambda^P(\mathbf{u})\}$ and $\{m_\lambda^P(\mathbf{u})\}$. Hwang \cite[Theorem 4.2]{Hwang} showed that the $P$-elementary symmetric functions commute in $\mathcal{R}_P$. Hence, any symmetric function identity holds for $P$-symmetric functions. For example, we have the Jacobi--Trudi formula $s_\lambda^P(\mathbf{u}) = \det[e_{\lambda'_{i} - i + j}^P(\mathbf{u})]$ in $\mathcal{R}_P$.

\begin{defin} Let $\Lambda(\mathbf{x})$ denote the ring of symmetric functions in indeterminates $x_1,x_2,...$. Define the \emph{noncommutative $P$-Cauchy product} $\Omega^P(\mathbf{u}, \mathbf{x}) \in \mathcal{R}_P \otimes \Lambda(\mathbf{x})$ by
\begin{equation}
\Omega^P(\mathbf{u}, \mathbf{x}) := \prod_{i \geq 0} \sum_{k \geq 0} e_k^P(\bu) x_i^k = \sum_{\lambda} e_{\lambda}^P(\mathbf{u})m_{\lambda}(\mathbf{x})
\end{equation}
in $\mathcal{R}_P \otimes \Lambda(\bx)$.
\end{defin}
\begin{theorem}\cite[Proposition 2.9]{BEPS}
Let $P$ be a \threeone -free poset. Then
\begin{equation}
\Omega^P(\mathbf{u}, \mathbf{x}) = \sum_{\lambda} s_\lambda^P(\mathbf{u}) s_{\lambda'}(\mathbf{x}) = \sum_{\lambda} m_{\lambda}^P(\mathbf{u}) e_\lambda(\mathbf{x}).
\end{equation} 
\end{theorem}
\begin{defin}
Define the \emph{evaluation functions} $\eval: \mathcal{R}_P \to \mathbb{Q}$  and, for $P$ a natural unit interval order, $\eval_q: \mathcal{R}_P \to \mathbb{Q}[q]$ to be the linear functions defined by
\begin{equation}
\eval(\bu_{\mathbf{w}}) = \begin{cases} 1 & \text{if }\mathbf{w} \in \mathfrak{S}_n, \\
0 & \text{otherwise,}
\end{cases} \hspace{1cm} \text{ and } \hspace{1cm} \eval_q(\bu_{\mathbf{w}}) = \begin{cases} q^{\inv_P(\mathbf{w})} & \text{if } \mathbf{w} \in \mathfrak{S}_n, \\
0 & \text{otherwise,}
\end{cases}
\end{equation}
where $\inv_P(\mathbf{w}) = \#\{(i, j) \mid i < j, w_i > w_j, \text{ and } w_i \dote_P w_j \}$ and $\mathfrak{S}_n$ is the set of permutations of $[n]$ written as words in one-line notation.
\end{defin}

\begin{theorem}\label{evalTheorem}\cite[Corollary 2.10]{BEPS}
Let $P$ be a \threeone -free poset.
Applying the evaluation maps to $\Omega^P(\mathbf{u}, \mathbf{x})$, we have
\begin{equation}
X_{\inc(P)}(\mathbf{x}) = \eval(\Omega^P(\mathbf{u},\mathbf{x})) = \sum_{\lambda} \eval(s_{\lambda}^P(\mathbf{u}))s_{\lambda'}(\mathbf{x}) = \sum_{\lambda} \eval(m_{\lambda}^P(\mathbf{u}))e_{\lambda}(\mathbf{x}),
\end{equation}
and if $P$ is a natural unit interval order,
\begin{equation}
X_{\inc(P)}(\mathbf{x}, q) = \eval_q(\Omega^P(\mathbf{u}, \mathbf{x})) =  \sum_{\lambda} \eval_q(s_{\lambda}^P(\mathbf{u}))s_{\lambda'}(\mathbf{x}) = \sum_{\lambda} \eval_q(m_{\lambda}^P(\mathbf{u}))e_{\lambda}(\mathbf{x}).
\end{equation}
In particular, for $\lambda \vdash n$, $c_{\lambda}^P = \eval(m_{\lambda}^P(\mathbf{u}))$ and $c_{\lambda}^P(q) = \eval_q(m_{\lambda}^P(\mathbf{u}))$.
\end{theorem}

Thus, using the evaluation maps, if one can prove positivity results for $P$-symmetric functions, one can obtain positivity results for $X_{\inc(P)}(\mathbf{x})$ and $X_{\inc(P)}(\mathbf{x}, q)$. In particular, Theorem \ref{evalTheorem} implies that if $s_\lambda^P(\mathbf{u})$ can be written as a positive sum of $\mathbf{u}$-monomials, then the coefficient of $s_{\lambda'}(\mathbf{x})$ in $X_{\inc(P)}(\mathbf{x})$ is nonnegative. Likewise, if $m_{\lambda}^P(\mathbf{u})$ can be written as a positive sum of $\mathbf{u}$-monomials, then the coefficient of $e_{\lambda}(\mathbf{x})$ is nonnegative. In the Schur basis, this is achieved through the following theorem.
\begin{theorem}
\cite{BEPS, Hwang}
Let $P$ be a \threeone -free poset. Then
\begin{equation}
s_\lambda^P(\mathbf{u}) \equiv \sum_{T\, \in\, \T_P(\lambda)} \mathbf{u}_{\colword(T)} \mod I_P.
\end{equation}
\end{theorem}
\noindent

Positivity results for $P$-monomial symmetric functions can be obtained using symmetric function identities. For a partition $\lambda$, write $\lambda^-$ for the partition obtained by removing the first column of $\lambda$.
\begin{theorem}
\label{HwangMaxChainThm}
\cite[Theorem 4.37]{Hwang}
Let $r$ be the number of elements in a maximum length chain of $P$. If $\lambda$ is a partition such that $\ell(\lambda) = r$, then
\begin{equation}
m_{\lambda}^P(\mathbf{u}) = e_r^P(\mathbf{u}) m_{\lambda^-}^P(\mathbf{u}).
\end{equation}
\end{theorem}
\begin{corollary}
Let $r$ be the number of elements in a maximum length chain of $P$. Then for any positive integer $c$,
\begin{equation}
m_{c^r}^P(\mathbf{u}) = e_{r^c}^P(\mathbf{u}) = s_{c^r}^P(\mathbf{u}).
\end{equation}
\end{corollary}
\begin{proof}
The identity $m_{c^r}^P(\mathbf{u}) = e_{r^c}^P(\mathbf{u})$ follows directly from Theorem~\ref{HwangMaxChainThm}. The identity $e_{r^c}^P(\mathbf{u}) = s^P_{c^r}(\mathbf{u})$ follows from the Jacobi--Trudi identity and the fact that $e_{k}^P(\mathbf{u}) = 0$ for $k > r$.
\end{proof}

\subsection{Previous $m_\lambda^P(\mathbf{u})$ Positivity Results}
\label{sec prev mlamP results}

One approach to finding a combinatorial interpretation for $c_{\lambda}^P(q)$ is to find sets of \emph{key $P$-tableaux} $\keySSYT_P(\lambda) \subseteq \T_P(\lambda)$ such that
\begin{align}
\label{keySSYTeq}
m^P_{\lambda}(\mathbf{u}) = \sum_{T\, \in\, \keySSYT_P(\lambda)} \mathbf{u}_{\colword(T)} \mod I_P.
\end{align}
We will write $\keySSYT_P(\lambda)$ for any subset of $\T_P(\lambda)$ that satisfies \eqref{keySSYTeq}. 
Towards this goal, such a set $\keySSYT_P(\lambda)$ has been determined when $\lambda$ is a two-column \cite[Theorem 4.29]{Hwang} \cite[Theorem 3.38]{BEPS} or hook \cite[Theorem 4.34]{Hwang} \cite[Theorem 3.32]{BEPS} shape.

\begin{defin}
Define the \emph{partial evaluation} function $\partEval: \mathcal{R}_P \to \mathcal{R}_P$ by
\begin{equation}
\partEval(\mathbf{u}_{\mathbf{w}}) = \begin{cases} \mathbf{u}_{\mathbf{w}} & \text{if }\mathbf{w} \text{ has no repeated entries,} \\
0 & \text{otherwise.}
\end{cases}
\end{equation}
We write $\keyIYT_P(\lambda) \subseteq \IYT_P(\lambda)$ for any set satisfying
\begin{equation}
\label{eq keyIT def}
\partEval(m_{\lambda}^P(\mathbf{u})) \equiv \sum_{T\, \in\, \keyIYT_P(\lambda)} \mathbf{u}_{\colword(T)} \mod I_P.
\end{equation}
Observe that $\eval \circ \partEval = \eval$ and $\eval_q \circ \partEval = \eval_q$, so for the purposes of obtaining a combinatorial interpretation of $c_\lambda^P$ or $c_{\lambda}^P(q)$, it suffices to determine $\keyIYT_P(\lambda)$.
\end{defin}

%

\begin{theorem}\label{twoColmPpos}
\cite[Theorem 4.29]{Hwang}\cite[Theorem 3.38]{BEPS}
For a two-column shape $\lambda$, define the subset $\keySSYT_P(\lambda)$ to be the set of $P$-tableaux with no right-unbalanced ladders. Then 
\begin{equation}
m_{\lambda}^P(\mathbf{u}) \equiv \sum_{T\, \in\, \keySSYT_P(\lambda)} \mathbf{u}_{\colword(T)} \mod I_P.
\end{equation}
\end{theorem}

\begin{defin}
For a word $\mathbf{w} = w_1w_2...w_k$ in $P^k$ with no $P$-descents, we say $\mathbf{w}$ has a right-left (RL) $P$-minima in position $i$ if $w_i <_P w_j$ for all $j > i$. Note that every word of length $k$ will have a RL $P$-minima in position $k$. We say a word $\mathbf{w}$ of length $k$ has no nontrivial RL $P$-minima if the only RL $P$-minima is in position $k$.
\end{defin}

\begin{defin}
\label{def powersum word}
If a word $\mathbf{w} = w_1w_2 ...w_k$ has no $P$-descents and no non-trivial RL $P$-minima, we say $\mathbf{w}$ is a \emph{powersum word} for $P$. If $\mathbf{w}$ is a powersum word with no repeated entries, then $\mathbf{w}$ is a \emph{powersum permutation} of $\{w_1,w_2,...,w_k\}$ for $P$.
\end{defin}

Athanasiadis \cite{Athanasiadis} used powersum permutations to describe the coefficient of powersum symmetric functions $p_{\lambda}(\mathbf{x})$ in $X_{\inc(P)}(\mathbf{x}, q)$, proving a conjecture of Shareshian--Wachs \cite{SWchromatic}. Using this framework, Hwang \cite{Hwang} and Blasiak--Eriksson--Pylyavskyy--Siegl \cite{BEPS} gave a combinatorial interpretation of $c_{\lambda}^P$ for hook shape partitions $\lambda$ when $P$ is a natural unit interval order and \threeone -free poset, respectively.
 
\begin{theorem}\cite[Theorem 4.34]{Hwang},\cite[Theorem 3.32]{BEPS}.
\label{hookShapemPpos} 
Let $\lambda = a1^b$ be a hook shape partition, and let
\begin{equation}
\keySSYT_P(\lambda) = \{T \in \T_P(\lambda) \mid T_{i,1}T_{1,2}T_{1,3} \cdots T_{1,a} \text{ is a powersum word for some } i \in [b]\}.
\end{equation}
Then 
\begin{equation}
m_\lambda^P(\bu) \equiv \sum_{T\, \in \, \keySSYT_P(\lambda)} \bu_{\colword(T)} \mod I_P.
\end{equation}
\end{theorem}

As a corollary of Theorem~\ref{HwangMaxChainThm}, Hwang \cite{Hwang} determined $\keySSYT_P(\lambda)$ when $P$ is a $\mathbf{3}$-free poset. This result refines theorems of Stanley \cite[Corollary 3.6]{Stanleychromatic} and Harada--Precup \cite{harada2019cohomology}.
\begin{corollary}
\cite[Corollary 4.38]{Hwang}
\label{3freeKeyTabHwang}
Let $P$ be a $\mathbf{3}$-free poset. For a partition $\lambda$, we have
\begin{equation}
m_\lambda^P(\bu) \equiv \begin{cases} 0 & \text{if } \ell(\lambda) > 2, \\
\sum_{\bw} \bu_{\bw}  \mod I_P & \text{otherwise},
\end{cases}
\end{equation}
where the sum is over words $\bw = w_1w_2\dots w_n$ in $P^n$ such that $w_{2i} >_P w_{2i-1}$ for $1 \leq i \leq \lambda_1$ and $w_{2\lambda_2+ 1} w_{2\lambda_2 + 2} \dots w_n$ is a powersum word.
\end{corollary}

\subsection{General $e$-Positivity Results}
\label{subsec Hik prelim}
Hikita \cite{HikitaChromatic} recently gave a probabilistic proof of the Stanley--Stembridge conjecture. Furthermore, for $P$ a natural unit interval order, Hikita characterized when the coefficient $c_\lambda^P$ is non-zero in terms of classical standard tableaux of shape $\lambda$. We follow the presentation of version 1 of Hikita's paper \cite{HikitaChromatic}.  To state Hikita's result, we first recall the definition of \emph{reverse Hessenberg functions}.

\begin{defin}\cite{HikitaChromatic}
A \emph{reverse Hessenberg function} is a function $\mathbf{m}: [n] \to \mathbb{Z}_{\geq 0}$ such that
\begin{enumerate}
\item $\mathbf{m}(i) < i$, and
\item $\mathbf{m}(i) \leq \mathbf{m}(i+1)$
\end{enumerate}
for all $i \in [n]$. Write $\mathbb{E}_n$ for the set of reverse Hessenberg functions from $[n]$ to $\mathbb{Z}_{\geq 0}$. For $\mathbf{m} \in \mathbb{E}_n$ and an integer $1 \leq k < n$, write $\mathbf{m}^{(k)}: [k] \to \mathbb{Z}_{\geq 0}$ for the element of $\mathbb{E}_{k}$ such that $\mathbf{m}^{(k)}(i) = \mathbf{m}(i)$ for $i = 1,2,...,k$.

For a reverse Hessenberg function $\mathbf{m} \in \mathbb{E}_n$, define the associated natural unit interval order $P_{\mathbf{m}}$ on $[n]$ by letting $i <_{P_{\mathbf{m}}} j$ if $i \leq \mathbf{m}(j)$. Every natural unit interval order on $[n]$ is $P_{\mathbf{m}}$ for some $\mathbf{m} \in \mathbb{E}_n$.
\end{defin}

Hikita's probabilistic approach to $e$-positivity builds on the modular law of Abreu--Nigro \cite{AbreuNigro}. In particular, the modular law gives rise to the following identity. Hikita's proof involves constructing probability distributions $\prob_\bm(-;q)$ on standard Young tableau parameterized by nonnegative real values of $q$ such that $\sum_{T \in \SYT(\lambda)} \prob_{\bm}(T; q) =  q^{|\bm| - \lambda^\star} \frac{c_\lambda^{P}(q)}{[\lambda]_q!}$. We recall the construction of $\prob_\bm(-;q)$ below.
\begin{theorem}
Let $\bm \in \mathbb{E}_n$ be a reverse Hessenberg function, and let $P = P_\bm$. Then
\begin{equation}
\sum_{\lambda \vdash n} q^{|\bm| - \lambda^\star} \frac{c_\lambda^{P}(q)}{[\lambda]_q!} = 1,
\end{equation}
where $\lambda^\star = \sum_{i < j} \lambda_i \lambda_j$.
\end{theorem}

\begin{defin}
\label{def hik prob}
\cite{HikitaChromatic}
Let $\lambda \vdash n$, and let $T \in \SYT(\lambda)$. Let $0 \leq r \leq n$ be an integer, and let $t_i = T_{\lambda_i', i}$ be the largest entry in the $i^{th}$ column of $T$. Define the sequence $\delta^{(r)}(T) = (\delta_1,\delta_2,...,\delta_{n+1})$ by
\begin{equation}
\label{eq color seq}
\delta_i = \begin{cases} 1 & \text{if } t_i > r, \\
0 & \text{otherwise.}
\end{cases}
\end{equation}
Collecting consecutive strings of 1s and 0s, write $\delta^{(r)}(T) = (1^{b_0},0^{a_1},1^{b_1},..., 1^{b_\ell}, 0^{a_{\ell+1}})$. For $0 \leq k \leq \ell$, let 
\begin{equation}
c_k^{(r)}(T) = 1 + \sum_{i = 1}^k a_i + \sum_{i=0}^k b_i.
\end{equation}
For each $k$, the largest entry in column $c_k^{(r)}(T) - 1$ of $T$ is less than the largest entry of column $c_k^{(r)}(T)$. Hence, inserting $n+1$ into column $c_k^{(r)}(T)$ of $T$ gives a valid Young tableau. 
Let $f_k^{(r)}(T)$ be the standard Young tableau of size $n+1$ obtained by adding entry $n+1$ to column $c_k^{(r)}(T)$ of $T$. Define the \emph{transition probability} $\varphi_k^{(r)}(T; q)$ in the growth process from $T$ to $f_k^{(r)}(T)$ by
\begin{equation}
\varphi_k^{(r)}(T; q) = q^{\sum_{i=1}^k a_i} \prod_{i=1}^k \frac{\left[\sum_{j=i+1}^k a_j + \sum_{j=i}^k b_j\right]_q}{\left[\sum_{j=i}^k a_j + \sum_{j=i}^k b_j \right]_q} \prod_{i = k+1}^\ell \frac{\left[\sum_{j=k+1}^i a_j + \sum_{j = k+1}^{i-1} b_j \right]_q}{\left[\sum_{j=k+1}^i a_j + \sum_{j=k+1}^i b_j \right]_q}.
\end{equation} 
\noindent
For a reverse Hessenberg function $\bm \in \mathbb{E}_n$ and $q \in \mathbb{R}_{\geq 0}$,
Hikita defines a probability distribution $\prob_{\bm}$ on standard Young tableaux of size $n$ by
\begin{equation}
\label{eq hik prob recur}
\prob_{\bm}(T; q) = \begin{cases} 1 & \text{if } n = 0, \\\varphi^{(r)}_k(S; q) \prob_{\bm^{(n-1)}}(S; q) & \text{if } T = f_k^{(r)}(S), \\
0 & \text{otherwise},
\end{cases}
\end{equation}
where $r = \bm(n)$.
\end{defin}

\begin{theorem}
\label{thm Hikita}
\cite[Theorem 3]{HikitaChromatic} Let $\bm \in \mathbb{E}_n$, and let $P = P_{\bm}$ be the associated natural unit interval order. Then for $\lambda \vdash n$, 
\begin{equation}
\sum_{T\, \in\, \SYT(\lambda)} \prob_\bm(T; q) = q^{|\bm| - \lambda^{\star}}\frac{c_\lambda^P(q)}{[\lambda]_q!},
\end{equation}
where $|\bm| = \bm(1) + \bm(2) + \dots + \bm(n)$ and $\lambda^{\star} = \sum_{i < j} \lambda_i \lambda_j$. Hence, $c_\lambda^P(\tau) \geq 0$ for all real numbers $\tau \geq 0$.
\end{theorem}

\begin{corollary}
\label{cor hik tab pos}
Let $\mathbf{m} \in \mathbb{E}_n$, and let $P = P_{\mathbf{m}}$ be the associated natural unit interval order. Then $c_{\lambda}^P > 0$ if and only if there is some $T \in \SYT(\lambda)$ such that $\prob_\bm(T; q) \neq 0$.
\end{corollary}

\subsection{Greedy Partition of $P$}
For partitions $\lambda, \mu \vdash n$, recall that $\lambda$ \emph{dominates} $\mu$ and write $\lambda \trianglerighteq \mu$ if for all $i$, $\lambda_1 + \lambda_2 + \dots +\lambda_i \geq \mu_1 + \mu_2 + \dots + \mu_i$. 
We have that conjugation of partitions is order-reversing. 

Matherne--Morales--Selover \cite[Theorem 5.7]{MMSlorentzian} observed that for a $\threeone$-free poset $P$, there is a unique partition $\lambda^{gr} = \lambda^{gr}(P)$, called the \emph{greedy partition} of $P$, that is maximal in dominance order among partitions $\lambda$ such that the coefficient of $m_{\lambda}(\mathbf{x})$ is non-zero in $X_{\inc(P)}(\mathbf{x})$.
By triangularity of the change of basis transformations, this means that $\lambda^{gr}$ is maximal among partitions $\lambda$ such that $s_{\lambda}(\mathbf{x})$ has a non-zero coefficient in $X_{\inc(P)}(\mathbf{x})$, and $\lambda^{gr'}$ is minimal among partitions such that the coefficient of $e_{\lambda}(\mathbf{x})$ is non-zero. 
Furthermore, the coefficients of $m_{\lambda^{gr}}(\mathbf{x}), s_{\lambda^{gr}}(\mathbf{x})$, and $e_{\lambda^{gr'}}(\mathbf{x})$ in $X_{\inc(P)}(\mathbf{x})$ are all equal.
We can interpret their theorem in terms of $P$-arrays as follows.
\begin{theorem} \label{greedyPartThm}
\cite{MMSlorentzian}
Let $P$ be a \threeone -free poset with $n$ elements. There is a unique partition $\lambda^{gr} = \lambda^{gr}(P)$ such that $\lambda^{gr} \trianglerighteq \mu'$
for all $\mu \vdash n$ such that the set of bijective $P$-arrays of shape $\mu$ is non-empty.
\end{theorem}

\begin{corollary}
\label{partialDominantCor}
For a \threeone -free poset $P$ and partition $\mu = (\mu_1, \mu_2,...,\mu_r) \vdash m \leq n$, if there is an injective $P$-array of shape $\mu'$, then
\begin{equation}
\mu_1 + \mu_2 + \dots + \mu_i \leq \lambda^{gr}_1 + \lambda^{gr}_2 + \dots + \lambda^{gr}_i
\end{equation}
for all $i = 1,2,...,r$. 
\end{corollary}
\begin{proof}
If $A$ is an injective $P$-array of shape $\mu'$, then we can form a bijective $P$-array $\tilde{A}$ of shape $\tilde{\mu}' = (\mu_1, \mu_2,...,\mu_r, 1,1,...,1)' \vdash n$  by putting each element of $P$ not appearing in $A$ in its own column. Then $\tilde{\mu} \trianglelefteq \lambda^{gr}$ and the statement follows from Theorem~\ref{greedyPartThm}.
\end{proof}

For $P$ a natural unit interval order, define the greedy chain $C^{gr}(P)$ of $P$ to be the chain $i_1 <_P i_2 <_P \dots <_P i_k$ obtained by letting $i_1 = \min P$ and $i_j = \min\{p \in P \mid p >_P i_{j-1}\}$. Let $C_1^{gr} = C^{gr}(P)$ and $C_i^{gr}(P) = C^{gr}(P \setminus \bigcup_{j=1}^{i-1} C_j^{gr}(P))$.
\begin{theorem}\cite[Proposition~4.6]{MMSlorentzian}
If $P$ is a natural unit interval order, then $\lambda^{gr}(P) = (|C_1^{gr}(P)|, |C_2^{gr}(P)|,\dots)$.
\end{theorem}

\begin{exmp}
For $\bm = (0,0,1,1,2,4,4,6)$, $\lambda^{gr}(P) = (4,2,2)$. The chains $C_1^{gr}(P_\bm)$, $C_2^{gr}(P_\bm)$, and $C_3^{gr}(P_\bm)$ are labeled with $1,2,$ and $3$ respectively in the Hasse diagram shown below.
\[\begin{tikzcd}[ampersand replacement=\&,column sep=small, row sep=small]
	{8^{\color{red} 1}} \\
	{5^{\color{green} 2}} \& {6^{\color{red} 1}} \& {7^{\color{blue} 3}} \\
	{2^{\color{green} 2}} \& {3^{\color{red} 1}} \& {4^{\color{blue} 3}} \\
	\& {1^{\color{red} 1}}
	\arrow[from=2-1, to=1-1]
	\arrow[from=2-2, to=1-1]
	\arrow[from=3-1, to=2-1]
	\arrow[from=3-1, to=2-2]
	\arrow[from=3-1, to=2-3]
	\arrow[from=3-2, to=2-2]
	\arrow[from=3-3, to=2-2]
	\arrow[from=3-3, to=2-3]
	\arrow[from=4-2, to=3-2]
	\arrow[from=4-2, to=3-3]
	\arrow[from=3-2, to=2-3]
\end{tikzcd}\]
\end{exmp}

\subsection{Positivity for the Path Graph}
One important special case of incomparability graphs of natural unit interval orders are the naturally labeled paths. Write $P_n$ for the poset on $[n]$ associated to the reverse Hessenberg function $(0,0,1,2,\ldots, n-2) \in \mathbb{E}_n$. Observe that $\inc(P_n)$ is the path graph on $[n]$ labeled such that vertices $i,j$ are adjacent if and only if $|i - j| = 1$. See Figure~\ref{fig path}. In this case, Shareshian--Wachs \cite[Theorem C.4]{SWchromatic} gave a positive formula for $c_\lambda^{P_n}(q)$ in terms of $q$-integers.

\begin{theorem}
\label{pathSWqFormulaThm}
\cite[Theorem C.4]{SWchromatic} For a positive integer $n$, we have 
\begin{equation}
X_{\inc(P_n)}(\mathbf{x}, q) = \sum_{\alpha\, \vDash\, n} q^{\ell(\alpha)-1} [\alpha_\ell]_q \prod_{i=1}^{\ell(\alpha) - 1} [\alpha_i - 1]_q e_{\sort(\alpha)}(\mathbf{x}).
\end{equation}
\end{theorem}
\begin{figure}[h]
\caption{The Path Graph}
\label{fig path}
\[\begin{tikzcd}[ampersand replacement=\&]
	\& 5 \& 4 \\
	{P_5 = } \& 3 \& 2 \&\& {\textrm{inc}(P_5) =} \& 1 \& 2 \& 3 \& 4 \& 5 \\
	\& 1
	\arrow[dashed, no head, from=1-2, to=1-3]
	\arrow[from=2-2, to=1-2]
	\arrow[dashed, no head, from=2-2, to=1-3]
	\arrow[dashed, no head, from=2-2, to=2-3]
	\arrow[from=2-3, to=1-2]
	\arrow[from=2-3, to=1-3]
	\arrow[no head, from=2-6, to=2-7]
	\arrow[no head, from=2-7, to=2-8]
	\arrow[no head, from=2-8, to=2-9]
	\arrow[no head, from=2-9, to=2-10]
	\arrow[from=3-2, to=1-3]
	\arrow[from=3-2, to=2-2]
	\arrow[dashed, no head, from=3-2, to=2-3]
\end{tikzcd}\]
\end{figure}

An important generalization of the path graph in the context of finding a positive formula for $c_{\lambda}^P(q)$ are \emph{$K$-chains}.
For graphs $G_1, G_2$ with vertex sets $[n_1]$ and $[n_2]$ respectively, we define the \emph{sum} $G_1 + G_2$ of $G_1$ and $G_2$ to be the graph on vertex set $[n_1 + n_2 - 1]$ obtained by identifying vertex $n_1$ in $G_1$ with vertex $1$ in $G_2$ and relabeling vertices. 
\begin{defin}
\label{def K chain}
For a composition $\gamma = (\gamma_1, \gamma_2,...,\gamma_m)$ with each $\gamma_i \geq 2$, let $K_\gamma = K_{\gamma_1} + K_{\gamma_2} + \dots + K_{\gamma_m}$ be the \emph{$K$-chain associated to $\gamma$}, where $K_n$ denotes the complete graph on $n$ vertices.
\end{defin}
\begin{figure}
\caption{The $K$-Chain $K_{(2,3,4)}$}
\[
K_{(2,3,4)} = 
\begin{tikzcd}[ampersand replacement=\&]
	1 \& 2 \& 3 \& 4 \& 5 \& 6 \& 7
	\arrow[no head, from=1-1, to=1-2]
	\arrow[curve={height=-12pt}, no head, from=1-2, to=1-4]
	\arrow[no head, from=1-2, to=1-3]
	\arrow[no head, from=1-3, to=1-4]
	\arrow[no head, from=1-4, to=1-5]
	\arrow[curve={height=-12pt}, no head, from=1-4, to=1-6]
	\arrow[curve={height=-18pt}, no head, from=1-4, to=1-7]
	\arrow[no head, from=1-5, to=1-6]
	\arrow[curve={height=-12pt}, no head, from=1-5, to=1-7]
	\arrow[no head, from=1-6, to=1-7]
\end{tikzcd}\]
\end{figure}
The $K$-chains are a special case of incomparability graphs of natural unit interval orders. Gebhard--Sagan \cite{gebhard2001chromatic} proved $e$-positivity of the chromatic symmetric function $X_{K_\gamma}(\mathbf{x})$. More recently, Tom \cite{TomChromatic} proved $e$-positivity of the chromatic quasisymmetric function $X_{K_\gamma}(\mathbf{x}, q)$, giving a formula in terms of $q$-integers.

\begin{defin}\cite[Definition 4.15]{TomChromatic} For a composition $\gamma = (\gamma_1, \gamma_2,...,\gamma_l)$ with parts at least 2, define $A_\gamma$ to be the set of compositions $\alpha = (\alpha_1,\alpha_2,...,\alpha_l)$ such that $\alpha_1 \geq 1$ and for each $2 \leq i \leq l$ we have either
\begin{enumerate}
\item $\alpha_i < \gamma_{i-1} - 1$ and $\alpha_i+ \alpha_{i+1} + \dots + \alpha_l < \gamma_i + \gamma_{i+1} + \dots + \gamma_l - (l - i)$, or
\item $\alpha_i \geq \gamma_{i-1}$ and $\alpha_i + \alpha_{i+1} + \dots + \alpha_l \geq \gamma_i + \gamma_{i+1} + \dots + \gamma_l - (l-i).$
\end{enumerate}
\end{defin}
\begin{theorem}
\label{TomFormulaThm}
\cite[Corollary 4.18]{TomChromatic} For a composition $\gamma = (\gamma_1,\gamma_2,...,\gamma_l)$ with parts at least~2, 
\begin{equation}
X_{K_\gamma}(\mathbf{x}, q) = [\gamma_1 - 2]_q![\gamma_2 - 2]_q! \cdots [\gamma_{l-1} - 2]_q! [\gamma_l - 1]_q! \sum_{\alpha \in A_\gamma} [\alpha_1]_q\prod_{i=2}^l q^{m_i} [|\alpha_i - (\gamma_{i-1} - 1)|]_q e_{\sort(\alpha)}(\mathbf{x}),
\end{equation}
where $m_i = \min\{\alpha_i, \gamma_{i-1}-1\}$.
\end{theorem}

\section{Strong and Powerful $P$-Tableaux}

Recall that $P$ is assumed to be a \threeone -free poset on $[n]$ throughout the paper, where $n$ is a non-negative integer, including the empty poset.
In this section, we define the sets $\strongSSYT_P(\lambda)$ and $\decSSYT_P(\lambda)$ of \emph{strong} and \emph{powerful} $P$-tableaux for all partitions $\lambda$. We show that $\strongSSYT_P(\lambda) \subseteq \decSSYT_P(\lambda)$ and state a refinement of Conjecture~\ref{undercountqConj} in terms of $P$-analogues. We then reinterpret Theorem~\ref{twoColmPpos} and Theorem~\ref{hookShapemPpos} in terms of strong and powerful $P$-tableaux. Recall the definition of \emph{powersum words} from Definition~\ref{def powersum word}.


\begin{defin}
\label{def dec array}
For a composition $\alpha = (\alpha_1, \alpha_2,...,\alpha_l)$ with positive entries, a \emph{powerful array} $A$ of shape $\alpha$ is a function $A: \row(\alpha) \to P$ such that
\begin{enumerate}
\item each row of $A$ is a powersum word,
\item $A_{r,t} <_P A_{s,t}$ for all cells $(r, t), (s,t) \in \alpha$ with $r < s$, and
\item for each row $r$, $A_{r,\alpha_r} <_P A_{s, t}$ for all $s > r$ and $t > \alpha_r$.
\end{enumerate}
Let $\decArray_P(\alpha)$ be the set of powerful arrays of $P$ of shape $\alpha$. 
\end{defin}

\begin{defin}
For a filling $A$ of a composition shape $\alpha$, define $\tabl(A)$ to be the filling of $\sort(\alpha)$ obtained by pushing entries upward in their column so that the columns are top justified. We say $\tab$ is the \emph{tableau map}. 
\end{defin}
Figure~\ref{fig tab map} shows a poset $P$, a powerful array $A \in \decArray_P(1,3,2)$, and the image $\tab(A)$ of $A$ under the tableau map.
Note that $A_{1,1} = 1 <_P A_{2,2} = 2$ and $A_{1,1} = 1 <_P A_{2,3} = 4$, so $A$ satisfies condition $(3)$ in Definition~\ref{def dec array}.
\begin{figure}[h]
\caption{The tableau map}
\label{fig tab map}
\begin{align*}
\begin{tikzcd}[ampersand replacement=\&]
	\& 5 \& 6 \\
	{P=} \& 2 \& 3 \& 4 \\
	\&\& 1
	\arrow[from=2-2, to=1-2]
	\arrow[from=2-2, to=1-3]
	\arrow[from=2-3, to=1-3]
	\arrow[from=3-3, to=2-2]
	\arrow[from=3-3, to=2-3]
	\arrow[from=3-3, to=2-4]
\end{tikzcd} & &
A = \begin{ytableau} 1 \\ 3 & 2 & 4 \\ 6 & 5 \end{ytableau} & & \tab(A) = \begin{ytableau} 1 & 2 & 4 \\ 3 & 5 \\ 6 \end{ytableau}
\end{align*}
\end{figure}

\begin{lemma}
\label{lem tab well defined}
For $A \in \decArray_P(\alpha)$, $\tab(A)$ is a $P$-tableau of shape $\sort(\alpha)$.
\end{lemma}
\begin{proof}
Let $A$ be a powerful array of shape $\alpha$, and let $T = \tab(A)$. 
As $\tab$ pushes the entries of $A$ upward, $T$ is a filling of $\sort(\alpha)$.
As each column of $A$ is increasing in $P$, it suffices to show that the rows of $T$ are non-$P$-descending. Let $T_{i,j}$ and $T_{i,j+1}$ be adjacent entries of $T$. 
Observe that $T_{i,j} = A_{r,j}$ and $T_{i, j+1} = A_{s, j+1}$ for some $r \leq s.$ 
If $r = s$, then $T_{i,j} = A_{s,j} \not>_P A_{s, j+1} = T_{i, j+1}$. 
If $r < s$ and $j < \alpha_r$, then $T_{i,j} = A_{r,j} \not>_P A_{r, j+1} <_P A_{s, j+1} = T_{i,j+1}$, so $T_{i,j} \not>_P T_{i, j+1}$. 
If $r < s$ and $j = \alpha_r$, then $T_{i,j} = A_{r,j} <_P A_{s,j+1} = T_{i,j+1}$. Thus, $T$ has no $P$-descents in its rows and is therefore a $P$-tableau.
\end{proof}

\begin{defin}
\label{def pow tabx}
Define the set of \emph{powerful $P$-tableaux} $\decSSYT_P(\lambda)$ of shape $\lambda$ to be all $P$-tableaux of shape $\lambda$ that are in the image of $\tab$. Let $\decSYT_P(\lambda)$ be the set of standard powerful $P$-tableaux of shape $\lambda$, and let $\decIYT_P(\lambda)$ be the set of injective powerful $P$-tableaux.
\end{defin}

Consider the reverse Hessenberg function $\mathbf{m} = (0, 0, 1, 1, 3, 4, 5) \in \mathbb{E}_7$.
Figure~\ref{fig powerful Ptabx} shows all powerful $P_{\mathbf{m}}$-tableaux of shape $(3,2,2)$.

\begin{figure}[h]
\caption{Powerful $P_{\mathbf{m}}$-tableaux of shape $(3,2,2)$ with powersum words highlighted}
\label{fig powerful Ptabx}
\begin{align*}
\begin{ytableau}
*(red!20) 2 & *(red!20) 1 & *(green!20)3 \\
*(green!20)5 & *(green!20)4 \\
*(yellow!20) 7 & *(yellow!20)6
\end{ytableau}
\ \ \
\begin{ytableau}
*(red!20)1 & *(red!20)2 & *(red!20)3 \\
*(green!20)4 & *(green!20)5  \\
*(yellow!20)6 & *(yellow!20)7
\end{ytableau}
\ \ \
\begin{ytableau}
*(red!20)1 & *(red!20)3 & *(red!20)2 \\
*(green!20)4 & *(green!20)5 \\
*(yellow!20)6 & *(yellow!20)7
\end{ytableau}
\end{align*}
\end{figure}

\begin{lemma}
\label{lem tab injective}
For a partition $\lambda \vdash n$, 
the restriction of the tableau map 
\begin{equation}
\tabl: \bigsqcup_{\substack{\alpha\, \vDash\, n \\ \sort(\alpha) = \lambda}} \decArray_P(\alpha) \to \decSSYT_P(\lambda)
\end{equation}
is a bijection.
\end{lemma}

\begin{proof}
Let $A$ be a powerful array of shape $\alpha = (\alpha_1, \alpha_2,..., \alpha_l),$ and let $T = \tab(A)$. 
To prove injectivity, we perform induction on the number of rows $l$ of $A$.
If $l = 1$, then $A = T$ and $\tab$ is injective.
Suppose $\tab$ is injective when $A$ has fewer than $l$ rows.
Observe that $T_{1,j} = A_{1,j}$ for $j=1,2,...,\alpha_1$.
By Definition~\ref{def dec array}(3), $T_{1, \alpha_1} <_P T_{1, j}$ for all $j > \alpha_1$, so the first row of $T$ has a RL $P$-minima in position $\alpha_1$. Furthermore, as $T_{1, 1}T_{1,2}\cdots T_{1, \alpha_1}$ is a powersum word, $T_{1,\alpha_1}$ is the left-most RL $P$-minima in the first row of $T$. 
Thus, we can recover the first row of $A$ from $T$. 
Let $\hat{A}$ be the powerful array obtained by removing the first row of $A$.
Then $\hat{T} = \tab(\hat{A})$ is a $P$-tableau with entries
\begin{equation}
\hat{T}_{i,j} = \begin{cases} T_{i+1,j} & \text{if }j \leq \alpha_1, \\ T_{i,j} & \text{if } j > \alpha_1. \end{cases}
\end{equation}
Then by the inductive hypothesis, we can recover the rows of $\hat{A}$ from $\hat{T}$. Hence, $\tab$ is injective. As $\decSSYT_P(\lambda)$ is the image of $\bigsqcup_{\sort(\alpha) = \lambda} \decArray_P(\alpha)$ under the tableau map, the restricted tableau map is bijective as desired. 
\end{proof}

\begin{defin} For a \threeone -free poset $P$, define the set of \emph{strong $P$-tableaux} to be the set of $P$-tableaux such that each pair of adjacent columns has no right-unbalanced ladders. Write $\strongSSYT_P(\lambda)$ for the set of strong $P$-tableaux of shape $\lambda$, $\strongIYT_P(\lambda)$ for injective strong $P$-tableaux, and $\strongSYT_P(\lambda)$ for bijective strong $P$-tableaux, so $\strongSYT_P(\lambda) \subseteq \strongIYT_P(\lambda) \subseteq \strongSSYT_P(\lambda) \subseteq \T_P(\lambda)$.
\end{defin}

The first two $P$-tableaux in Figure~\ref{fig powerful Ptabx} are strong $P$-tableaux. The third $P$-tableau has a right-unbalanced ladder with entries $\{3,4,5,6,7\}$ between columns 1 and 2. More examples of strong $P$-tableaux can be found in Appendix~\ref{Appendix}.

\begin{lemma} Let $P$ be a \threeone -free poset. If $V = \{v_1 <_P v_2 <_P ... <_P v_s\}$ and $W = \{w_1 <_P w_2 <_P ... <_P w_t\}$ are chains in $P$ considered as columns in a $P$-array $VW$, then $VW$ has no right-unbalanced ladders if and only if there is an injection $g: [t] \to [s]$ such that $w_i \dote_P v_{g(i)}$ for each $i \in [t]$.
\end{lemma}
\begin{proof}
It suffices to consider the case when $VW$ consists of a single ladder. If $VW$ is a right-unbalanced ladder, then $t > s$ and there are no injections from $[t]$ to $[s]$. If $VW$ is not right-unbalanced, then for each $i \in [t]$, $w_i \dote_P v_i$. Therefore we can take $g: [t] \to [s]$ to be the function $g: i \mapsto i$. 
\end{proof}

\begin{corollary}
\label{cor strong inj}
Let $P$ be a \threeone -free poset and $\lambda$ a partition. A $P$-tableau $T \in \T_P(\lambda)$ is a strong $P$-tableau if and only if for each $i \in [\ell(\lambda') - 1]$ there is an injection $g_i: [\lambda_{i+1}'] \to [\lambda_i']$ such that $T_{r,i+1} \dote_P T_{g_i(r), i}$ for each $r \in [\lambda_{i+1}']$.
\end{corollary}

\begin{lemma}
\label{lem strong concatenation}
Let $T$ be a $P$-tableau with no right-unbalanced ladders between the first two columns, and let $T^-$ be the $P$-tableau obtained by deleting the first column of $T$. If $T^-$ is a powerful $P$-tableau, then $T$ is a powerful $P$-tableau. 
\end{lemma}
\begin{proof}
Let $V = v_1 <_P v_2 <_P ... <_P v_s$ and $W = w_1 <_P w_2 <_P ... <_P w_t$ be the entries of the first and second columns of $T$, and suppose $VW$ has no right-unbalanced ladders.
Let $A^-$ denote the powerful array such that $T^- = \tab(A^-)$.
Consider the set $\mathcal{I}$ of injective functions $g: [t] \to [s]$ such that $v_{g(i)}A^-_i$ is a powersum word for all $i \in [t]$, where $A^-_i$ is row $i$ of $A^-$. As there are no right-unbalanced ladders between $V$ and $W$, there is an injective function $\tilde{g}:[t] \to [s]$ such that $v_{\tilde{g}(i)} \dote_P w_i$ for all $i \in [t]$. Thus, we have $\tilde{g} \in \mathcal{I}$ and $\mathcal{I}$ is non-empty.  

Let $\hat{g}$ be the lexicographic minimum of $\mathcal{I}$, meaning for all $h \in \mathcal{I}$, $\hat{g}(i) < h(i)$ for the smallest $i \in [t]$ such that $\hat{g}(i) \neq h(i)$. We can then take $A$ to be the powerful array with rows
\begin{equation}
A_j = 
\begin{cases} v_j & \text{if } j \neq \hat{g}(i) \text{ for any } i \in [t], \\
					v_jA^-_i & \text{if }j = \hat{g}(i).
					\end{cases}
\end{equation}
By definition of $\mathcal{I}$, the rows of $A$ are powersum words. Furthermore, by the minimality of $\hat{g}$, if $j \neq \hat{g}(i)$ for any $i$, then $v_j <_P A_{r,t}$ for all $r > j$ and $t > 1$. Thus, $A$ is a powerful array, and $T$ is a powerful $P$-tableau.
\end{proof}

\begin{corollary}\label{strongDecLemma} For a \threeone -free poset $P$, $\strongSSYT_P(\lambda) \subseteq \decSSYT_P(\lambda)$.
\end{corollary}
\begin{proof}
Let $T \in \strongSSYT_P(\lambda)$. We perform induction on the number of columns of $T$. If $\lambda$ is a single column shape, then $\strongSSYT_P(\lambda) = \decSSYT_P(\lambda) = \T_P(\lambda)$, and we are done.
Suppose $T$ has $m$ columns and every strong $P$-tableau with $m-1$ columns is a powerful $P$-tableau. Let $T^-$ denote the $P$-tableau obtained by removing the first column of $T$.
Then $T^-$ is a strong $P$-tableau with $m-1$ columns, so $T^-$ is powerful by the inductive hypothesis. As there are no right-unbalanced ladders between the first two columns of $T$, we have that $T$ is a powerful $P$-tableau by Lemma~\ref{lem strong concatenation}.
\end{proof}

We can now state a refinement of Conjecture~\ref{undercountqConj} in terms of $P$-analogues. Applying $\eval_q$ to Conjecture~\ref{undercountConj}, we obtain Conjecture~\ref{undercountqConj}.

\begin{conjecture}
\label{undercountConj}
Let $P$ be a $\threeone$-free poset. Then for $\lambda \vdash n$,
\begin{equation}
m_{\lambda}^P(\mathbf{u}) - \left( \sum_{T\, \in\, \strongSSYT_P(\lambda)} \mathbf{u}_{\colword(T)} \right)
\end{equation}
has a $\mathbf{u}$-positive expansion as an element of $\mathcal{R}_P = \mathbb{Q}\langle u_1, u_2, ..., u_n \rangle /I_P$.
\end{conjecture} 

Special cases of strong and powerful $P$-tableaux appear in the results of \cite{BEPS, Hwang}. 
We can reinterpret Theorems \ref{twoColmPpos} and \ref{hookShapemPpos}, in terms of these $P$-tableaux as follows.
\begin{lemma}
\label{twoColDecLemma}
For $\lambda$ a two-column shape, let $\keySSYT_P(\lambda)$ be as defined in Theorem~\ref{twoColmPpos}. Then
$\decSSYT_P(\lambda) = \keySSYT_P(\lambda) = \strongSSYT_P(\lambda)$.
\end{lemma}
\begin{proof}
Let $P$ be a \threeone -free poset, and let $\lambda$ be a two-column partition.
From Theorem \ref{twoColmPpos}, $\keySSYT_P(\lambda) = \strongSSYT_P(\lambda)$. Suppose $T \in \decSSYT_P(\lambda)$. Then there is a powerful array $A$ such that $\tab(A) = T$. Let $g: [\lambda_2'] \to [\lambda_1']$ be the function that sends $i \in [\lambda_2']$ to the row $g(i)$ of $A$ containing $T_{i,2}$. Then $g$ is an injection, and $T_{g(i), 1}T_{i,2}$ is a powersum word for each $i \in [\lambda_2']$. Hence, $T_{i,2} \dote_P T_{g(i),1}$, and $T$ is a strong $P$-tableau from Corollary \ref{cor strong inj}.
\end{proof}

\begin{lemma}
\label{hookDecomp}
For $\lambda$ a hook shape, let $\keySSYT_P(\lambda)$ be as defined in Theorem~\ref{hookShapemPpos}. Then $\decSSYT_P(\lambda) = \keySSYT_P(\lambda)$.
\end{lemma}
\begin{proof}
Recall that a hook shape tableau $T$ with first column $v_1 <_P v_2 <_P \ldots <_P v_k$ and arm $w_1 \not >_P w_2 \not>_P \ldots \not>_P w_l$ is a key $P$-tableau if $v_iw_1w_2\ldots w_l$ is a powersum word for some $i \in [k]$.
Taking $i$ to be the least index such that $v_iw_1w_2\ldots w_l$ is a powersum word, we have a powerful array $A$ with $j^{th}$ row $v_j$ if $j \neq i$ and $i^{th}$ row $v_iw_1w_2\ldots w_l.$ Then $\tab(A) = T$. As any powerful $P$-tableau of shape $\lambda$ has some $v_i$ such that $v_iw_1w_2\ldots w_l$ is a powersum word, $\decSSYT_P(\lambda) = \keySSYT_P(\lambda)$ as desired.
\end{proof}
\section{Tableaux Witnesses for Positivity}
In this section, we connect the work of Hikita \cite{HikitaChromatic} to Conjecture~\ref{nonzeroCoeffConj}.
Recall from Corollary~\ref{cor hik tab pos} that, for a reverse Hessenberg function $\mathbf{m} \in \mathbb{E}_n$ and partition $\lambda \vdash n$, the $e$-coefficient $c_\lambda^{P_{\mathbf{m}}}$ is non-zero if and only if there is a tableau $T \in \SYT(\lambda)$ such that $\prob_\bm(T; q) \neq 0$. We say such a tableau \emph{witnesses positivity} of $c_{\lambda}^{P_\bm}$. 
As $P_{\mathbf{m}}$ is a poset on $[n]$, we can interpret the entries of a standard Young tableau $T \in \SYT(\lambda)$ as elements of $P_{\mathbf{m}}$. We show that if $T$ is a witness to positivity of $c_\lambda^{P_\bm}$, then $T$ is a strong $P_{\bm}$-tableau when the entries of $T$ are interpreted as elements of $P_\bm$. 

Recall that for a reverse Hessenberg function $\mathbf{m} \in \mathbb{E}_n$ and integer $0 \leq k < n$, $\mathbf{m}^{(k)}\in \mathbb{E}_k$ is the reverse Hessenberg function such that $\mathbf{m}^{(k)}(i) = \mathbf{m}(i)$ for $1 \leq i \leq k$. For a standard Young tableau $T$ with $n$ cells, $T^{(k)}$ is the tableau obtained by removing entries $k+1, k+2,...,n$. 
\begin{defin}
\label{def Hikita tabs}
Let $\bm \in \mathbb{E}_n$, and let $T$ be a standard tableau with $n$ cells. We say $T$ is an \emph{$\bm$-Hikita tableau} if $\prob_{\bm}(T;q) \neq 0$. Write $\posSYT(\bm, \lambda)$ for the set of $\bm$-Hikita tableaux of shape $\lambda$.
\end{defin}

\begin{lemma}
\label{lem Hikita tabs}
Let $\mathbf{m} \in \mathbb{E}_n$, let $T$ be a standard tableau with $n$ cells, and let $(i,j)$ be the cell of $T$ containing $n$. 
Then $T$ is an $\mathbf{m}$-Hikita tableau if and only if $n = 0$ or
\begin{enumerate}
\item \label{def hik tab recur} $T^{(n-1)}$ is an $\mathbf{m}^{(n-1)}$-Hikita tableau, 
\item if $i > 1$, then $T_{i-1,j} \leq \mathbf{m}(n)$, and 
\item if $j > 1$ and $x$ is the maximal entry of column $j-1$ of $T$, then $x > \mathbf{m}(n)$.
\end{enumerate}
Hence, $T^{(k)}$ is an $\mathbf{m}^{(k)}$-Hikita tableau for all $0 \leq k < n$.
\end{lemma}
\begin{proof}
For $n = 0$, $T$ must be the empty tableau, so $\prob_\bm(T; q) = 1$, and $T$ is an $\bm$-Hikita tableau. 
Suppose $n > 0$.
Let $r = \bm(n)$, and let $S = T^{(n-1)}$.
If $T$ is a $\bm$-Hikita tableau, then by \eqref{eq hik prob recur}, $\prob_\bm(T; q)$ is a multiple of $\prob_{\bm^{(n-1)}}(S; q)$, so $S$ is an $\bm^{(n-1)}$-Hikita tableau. Furthermore, there is some $k$ such that $T = f_k^{(r)}(S)$.
Let $\delta^{(r)}(S) = (\delta_1, \delta_2,\dots, \delta_{n})$ as defined in \eqref{eq color seq}. Then by the definition of $f_k^{(r)}(S)$, we have $\delta_j = 0$. Thus, if $i > 1$, we have $T_{i-1, j} \leq r.$
Likewise, if $j > 1$, we have $\delta_{j-1} = 1$. Thus, if $j > 1$ and $x$ is the maximum entry of the $j-1^{th}$ column of $T$, then $x > r$.
Therefore, $T$ satisfies conditions $(1),(2)$, and $(3)$.
Conversely, if $T$ satisfies conditions $(1)$, $(2)$, and $(3)$, then $T = f_k^{(r)}(S)$ for some $k$, $\varphi_{k}^{(r)}(S; q) \neq 0$, and $\prob_{\bm^{(n-1)}}(S; q) \neq 0$, so $\prob_{\bm}(T; q) \neq 0$, as desired.
\end{proof}

\begin{theorem}[Theorem~\ref{nonzeroCoeffThm}]
Let $\bm$ be a reverse Hessenberg function on $[n]$, and let $\lambda \vdash n$. Then $\posSYT(\bm, \lambda) \subseteq \strongSYT_{P_{\bm}}(\lambda)$. Hence, if $c_\lambda^{P_{\bm}} > 0,$ then $\strongSYT_{P_{\bm}}(\lambda) \neq \emptyset$.
\end{theorem}
\begin{proof}
We perform induction on $n$. If $n = 0$, then $T$ must be the empty tableau which is by definition a strong $P$-tableau.
Suppose for all $k < n$, $\mathbf{m}' \in \mathbb{E}_k$, $\mu \vdash k$, and $S \in \posSYT(\mathbf{m}', \mu)$, we have that $S$ is a strong $P_{\mathbf{m}'}$-tableau. Let $\mathbf{m} \in \mathbb{E}_{n}$, $\lambda \vdash n$, and $T \in \posSYT(\mathbf{m}, \lambda)$. We want to show that $T \in \strongSYT_P(\lambda)$.

We first show that the columns of $T$ are increasing chains in $P$. By Lemma~\ref{lem Hikita tabs} and the inductive hypothesis, $T^{(n-1)}$ is a strong $P_{\mathbf{m}^{(n-1)}}$-tableau. 
Let $T_{i,j} = n$.
From condition $(2)$ of Lemma~\ref{lem Hikita tabs}, if $i > 1$, then $T_{i-1, j} \leq \mathbf{m}(n)$, which means that $T_{i-1,j} <_{P_{\mathbf{m}}} T_{i,j} = n$. Thus, as $T^{(n-1)}$ is a $P_{\mathbf{m}^{(n-1)}}$-tableau by the inductive hypothesis, the columns of $T$ are increasing in $P_{\mathbf{m}}$.

Next we show that $T$ has no right-unbalanced ladders. Recall from Definition~\ref{def IP} that we write $a \dote_P b$ to mean $a$ is equal or incomparable to $b$ in $P$. From Corollary~\ref{cor strong inj}, it suffices to show that for each pair of adjacent columns $s$, $s+1$ of $T$, there is an injection $g_{s}: [\lambda_{s+1}'] \to [\lambda_s']$ such that $T_{t,s+1} \dote_P T_{g_s(t), s}$ for each $t \in [\lambda_{s+1}']$. 
As the columns of $T$ are increasing in $P_{\mathbf{m}}$ and $n$ is a maximal element of $P_{\mathbf{m}}$, we have that $i = \lambda_j'$.
As $T^{(n-1)}$ is a strong $P_{\mathbf{m}^{(n-1)}}$-tableau, we have that $T$ is a strong $P_{\mathbf{m}}$-tableau if $j = 1$. Suppose then that $j > 1$. 
By property $(3)$ of Lemma~\ref{lem Hikita tabs}, $T_{\lambda_{j-1}', j-1} > \mathbf{m}(n)$. 
Hence, we have that  $T_{\lambda_{j-1}', j-1} \dote_{P_{\mathbf{m}}} T_{\lambda_j', j}$.
Let $k = T_{\lambda_{j-1}', j-1}$. By the inductive hypothesis, $T^{(k-1)}$ is a strong $P_{\mathbf{m}^{(k-1)}}$-tableau.
Furthermore, if $\mu$ is the shape of $T^{(k-1)}$, then $\mu_{j-1}' = \lambda_{j-1}'-1$ as $k$ is the largest entry in column $j-1$, and $\mu_j' = \lambda_j' - 1$ as $k > \mathbf{m}(n) \geq T_{\lambda_j'-1,j}$.
Therefore, there is an injection $g_{j-1}^{(k-1)}: [\lambda_j' - 1] \to [\lambda_{j-1}' - 1]$ such that $T_{g_{j-1}^{(k-1)}(t), j-1} \dote_{P_{\mathbf{m}^{(k-1)}}} T_{t, j}$ for all $t \in [\lambda_j' - 1]$.
Hence, $g_{j-1}: [\lambda_j'] \to [\lambda_{j-1}']$ defined by $g_{j-1}(t) = g_{j-1}^{(k-1)}(t)$ if $t \in [\lambda_j'-1]$ and $g_{j-1}(\lambda_j') = \lambda_{j-1}'$ is an injection.
Therefore, $T_{g_{j-1}(t), j-1} \dote_{P_{\mathbf{m}}} T_{t, j}$ for all $t \in [\lambda_j']$, so $T$ is a strong $P_{\mathbf{m}}$-tableau.
\end{proof}

The following definition and lemma will be used in the proof of Theorem~\ref{thm Hik inv}. Recall the notation found in Definition~\ref{def hik prob}.

\begin{defin}
\label{def mod Hik prob}
Let $\bm \in \mathbb{E}_n, \lambda \vdash n,$ and $T \in \SYT(\lambda)$. Let $0 \leq r \leq n$ be an integer such that $\bm + (r) := (\bm(1), \bm(2),...,\bm(n), r)$ is a reverse Hessenberg function. Let $\delta^{(r)}(T) = (1^{b_0}, 0^{a_1}, 1^{b_1},0^{a_2},...,1^{b_\ell}, 0^{a_{\ell + 1}})$.  
Define the \emph{normalization factor} $A_k^{(r)}(T;q)$ by
\begin{equation}
A_k^{(r)}(T;q) = q^{\sum_{i=1}^k a_i}.
\end{equation}
Define $\zeta_{\bm}(T; q)$ by 
\begin{equation}
\zeta_{\bm}(T; q) = \begin{cases} 
1 & \text{if } n = 0, \\
A^{(\bm(n))}_k(S; q) \zeta_{\bm^{(n-1)}}(S; q) & \text{if } T = f_k^{(\bm(n))}(S),\\
0 & \text{otherwise.}
\end{cases}
\end{equation}
\end{defin}

\begin{lemma}
\label{lem hik inv eq}
Let $\bm \in \mathbb{E}_n$, $\lambda \vdash n$, and $T \in \posSYT(\bm, \lambda)$. Then
\begin{equation}
\label{eq inv zeta}
q^{\inv_{P_{\bm}}(T)} = q^{\lambda^{\star} - |\bm|} \zeta_\bm(T; q),
\end{equation}
where $\lambda^\star = \sum_{i < j} \lambda_i \lambda_j$.
\end{lemma}
\begin{proof}
We prove the statement by induction on $n$. For $n = 0$, we have $\lambda^\star = 0$, $|\bm| = 0$, and $\zeta_\bm(T; q) = 1$, so the right hand side of \eqref{eq inv zeta} is 1. As $T$ must be the empty tableau, the left hand side of \eqref{eq inv zeta} is 1, and the statement holds for $n = 0$.

Suppose $n > 0$. Let $T \in \posSYT(\bm, \lambda)$, and let $S = T^{(n-1)} \in \posSYT(\bm^{(n-1)}, \mu)$.
Suppose \eqref{eq inv zeta} holds for $S$. 
Let $r = \bm(n)$, and let $k$ be the index such that $f_k^{(r)}(S)$.
Say $\delta^{(r)}(S) = (1^{b_0}, 0^{a_1}, 1^{b_1},0^{a_2},...,1^{b_\ell}, 0^{a_{\ell + 1}})$ as defined in \eqref{eq color seq}.
Observe that $r+1,r+2,...,n-1$ are maximal elements of $P_{\bm^{(n-1)}}$. Hence, $r+1, r+2,...,n-1$ are the largest entries in their columns of $S$. As $n$ is incomparable to $r+1,r+2,...,n-1$ in $P_{\bm}$, we have $n- 1 - r = \sum_{i = 0}^\ell b_i.$ As $b_{k+1} + b_{k+2} + \dots + b_{\ell}$ is the number of elements of $\{r+1,r+2,\dots, n-1\}$ appearing to the right of $n$ in $T$, we have 
\begin{equation}
\inv_{P_{\bm}}(T) = \inv_{P_{\bm'}}(S) + b_{k+1} + b_{k+2} + \dots + b_{\ell}.
\end{equation}
Observe that if entry $n$ is in row $i$ of $T$, then $\lambda^\star = \mu^\star + \sum_{j \neq i} \mu_j$ and $\mu_i = \sum_{j=0}^k b_j + \sum_{j=1}^k a_j$. Hence, $\lambda^\star = \mu^\star + n-1 - \sum_{j=0}^k b_j - \sum_{j=1}^k a_j.$
Therefore,
\begin{align}
q^{\inv_{P_{\bm}}(T)} &= q^{\inv_{P_{\bm'}}(S) + b_{k+1} + b_{k+2} + \dots + b_{\ell}} \\
&=  q^{\mu^\star - |\bm'|}\zeta_{\bm'}(S; q)q^{b_{k+1} + b_{k+2} + \dots + b_{\ell}} \\
&= q^{\lambda^\star -n + 1 + \sum_{j=0}^kb_j + \sum_{j=1}^k a_j - |\bm| + r}\zeta_{\bm'}(S; q) q^{b_{k+1} + b_{k+2} + \dots + b_{\ell}} \\
&= q^{\lambda^\star +(r + 1 - n + \sum_{j=0}^\ell b_j)+ \sum_{j=1}^k a_j - |\bm|}\zeta_{\bm'}(S; q) \\
&= q^{\lambda^\star - |\bm| + \sum_{j=1}^k a_j} \zeta_{\bm'}(S; q).
\end{align}
From Definition~\ref{def mod Hik prob}, we have $\zeta_{\bm'}(S; q) = q^{-\sum_{j=1}^k a_j} \zeta_\bm(T;q)$, so
\begin{align}
q^{\inv_{P_\bm}(T)} &= q^{\lambda^\star - |\bm| + \sum_{j=1}^k a_j} \zeta_{\bm'}(S; q) \\
&=  q^{\lambda^\star - |\bm| + \sum_{j=1}^k a_j} q^{-\sum_{j=1}^k a_j} \zeta_{\bm}(T; q) \\
&= q^{\lambda^\star - |\bm|} \zeta_{\bm}(T; q),
\end{align}
as desired.
\end{proof}
\begin{theorem}[Theorem~\ref{thm Hik inv}]
Let $\bm$ be a reverse Hessenberg function on $[n]$, and let $\lambda \vdash n$. Then for each $T \in \posSYT(\bm, \lambda)$, there exists a function $h_T(q)$ such that $0 < h_T(\tau) \leq 1$ for all $\tau \geq 0$, and 
\begin{equation}
\sum_{T\, \in\, \posSYT(\bm, \lambda)} q^{\inv_{P_\bm}(T)} h_T(q) = \frac{c_\lambda^{P_\bm}(q)}{[\lambda]_q!}.
\end{equation}
\end{theorem}
\begin{proof}
Let $T \in \posSYT(\bm, \lambda)$.
Define $h_T(q)$ by
\begin{equation}
h_T(q) = \begin{cases}
1 & \text{if } T = \emptyset, \\
\frac{\varphi_k^{(r)}(S; q)}{A_k^{(r)}(S; q)} h_S(q) & \text{if } T = f_k^{(r)}(S).
\end{cases}
\end{equation}
From Definition~\ref{def hik prob} and Definition~\ref{def mod Hik prob}, we have $0 < \frac{\varphi_k^{(r)}(T; \tau)}{A_k^{(r)}(T; \tau)} \leq 1$ for all real $\tau \geq 0$, so $0 < h_T(\tau) \leq 1$ for all $T \in \posSYT(\bm; \lambda)$. Furthermore, we have $h_T(q) = \frac{\prob_{\bm}(T; q)}{\zeta_{\bm}(T; q)}$. Then by Theorem~\ref{thm Hikita} and Lemma~\ref{lem hik inv eq}, 
\begin{equation}
\sum_{T \, \in \, \posSYT(\bm, \lambda)} q^{\inv_{P_\bm}(T)} h_T(q) = \frac{c_\lambda^{P_\bm}(q)}{[\lambda]_q!}.
\end{equation}
\end{proof}

\begin{corollary}
Let $\bm \in \mathbb{E}_n$, and let $\lambda \vdash n$. Then for all $\tau \geq 0$, 
\begin{equation}
\sum_{T \, \in\, \posSYT(\bm, \lambda)} \tau^{\inv_{P_\bm}(T)} \geq \frac{c_\lambda^{P_\bm}(\tau)}{[\lambda]_{\tau}!}.
\end{equation}
\end{corollary}

\begin{exmp}
Figure~\ref{fig Hik Tabs} shows $\prob_{\bm}$ and related functions for $\bm = (0,0,1) \in \mathbb{E}_3$ and $\bm = (0,0,1,2) \in \mathbb{E}_4$. Note that each tableau is a strong $P_{\bm}$-tableau.
\end{exmp}

\begingroup
\renewcommand{\arraystretch}{2} 
\ytableausetup{boxsize = 1em}
\begin{figure}
\caption{Hikita Tableaux for $\bm = (0,0,1)$ and $\bm = (0,0,1,2)$}
\label{fig Hik Tabs}
\begin{tabular}{c|c|c|c|c|c|c|c|c}
$\bm$ & $T$ & $k$ & $\delta$ & $\prob_{\bm}(T;q)$ & $h_T(q)$ & $q^{\inv_{P_\bm}(T)}$ & $\zeta_\bm(T;q)$ & $q^{\lambda^\star - |\bm|}$ \\
\hline
$(0,0,1)$ & $\begin{ytableau} 1 & 2 & 3 \end{ytableau}$ & $1$ & $[0,1,0]$ & $\frac{q}{[2]_q}$ & $\frac{1}{[2]_q}$ & $q^0$ & $q^1$ & $q^{-1}$
\\
\hline
$(0,0,1)$ & $\begin{ytableau} 1 & 2 \\ 3 \end{ytableau}$ & $0$ & $[0, 1, 0]$ & $\frac{1}{[2]_q}$ & $\frac{1}{[2]_q}$ & $q^1$ & $q^0$ & $q^1$
\\
\hline
$(0,0,1,2)$ & $\begin{ytableau}1 & 2 & 3 & 4 \end{ytableau}$ & $1$ & $[0,0,1,0]$ & $\frac{q^3}{[2]_q[3]_q}$ & $\frac{1}{[2]_q[3]_q}$ & $q^0$ & $q^3$ & $q^{-3}$ \\
\hline 
$(0,0,1,2)$ & $\begin{ytableau}1 & 2 & 3 \\4\end{ytableau}$ & $0$ & $[0,0,1,0]$ & $\frac{q}{[3]_q}$ & $\frac{1}{[3]_q}$ & $q^1$ & $q^1$ & $q^0$ \\ 
\hline
$(0,0,1,2)$ & $\begin{ytableau}1 & 2 \\3 & 4\end{ytableau}$ & $0$ & $[1,0,0,0]$ & $\frac{1}{[2]_q}$ & $\frac{1}{[2]_q}$ & $q^1$ & $q^0$ & $q^1$ \\
\hline
\end{tabular}
\end{figure}
\ytableausetup{boxsize = 1.5em}
\endgroup

In Conjecture~\ref{nonzeroCoeffConj}, we conjecture that, for a \threeone -free poset $P$, if $\strongSYT_P(\lambda) = \emptyset$, then $c_\lambda^P = 0$. Theorem~\ref{nonzeroCoeffThm} is a converse of this conjecture in the case when $P$ is a natural unit interval order, showing that if $c_\lambda^P > 0$, then $\strongSYT_P(\lambda) \neq \emptyset$. 
For natural unit interval orders, we may then restate Conjecture~\ref{nonzeroCoeffConj} as the following.
\begin{conjecture}
\label{conj strong implies Hik}
Let $\mathbf{m} \in \mathbb{E}_n$, let $P = P_{\mathbf{m}}$, and let $\lambda \vdash n$. Then $\strongSYT_P(\lambda)$ is non-empty if and only if $\posSYT(\mathbf{m}, \lambda)$ is non-empty. 
\end{conjecture}

As $\posSYT(\bm, \lambda) \subseteq \strongSYT_{P_{\bm}}(\lambda)$, we have that the polynomial
\begin{equation}
\sum_{T\, \in\, \strongSYT_{P_{\bm}}(\lambda)} q^{\inv_{P_{\bm}}(T)} - \sum_{T\, \in\, \posSYT(\bm, \lambda)} q^{\inv_{P_\bm}(T)}
\end{equation}
has nonnegative coefficients. Hence, relaxing Conjecture~\ref{undercountqConj}, we conjecture that $\posSYT(\bm, \lambda)$ gives a lower bound for $c_\lambda^{P_\bm}(q)$.
\begin{conjecture}
\label{conj hik tab undercount}
Let $\bm \in \mathbb{E}_n$ and $\lambda \vdash n$. Then the polynomial
\begin{equation}
c_\lambda^{P_\bm}(q) - \sum_{T\, \in\, \posSYT(\bm, \lambda)} q^{\inv_{P_\bm}(T)}
\end{equation}
has nonnegative coefficients.
\end{conjecture}
Assuming Conjecture~\ref{conj hik tab undercount}, the inequality
\begin{equation}
\label{eq hik eval ineq}
\sum_{T\, \in\, \posSYT(\bm, \lambda)} \tau^{\inv_{P_\bm}(T)} \leq c_\lambda^{P_\bm}(\tau)
\end{equation}
holds for all real numbers $\tau \geq 0$.
As $\sum_{T \in \posSYT(\bm, \lambda)} \tau^{\inv_{P_\bm}(T)}h_T(\tau) = \frac{c_\lambda^{P_\bm}(\tau)}{[\lambda]_\tau!},$ it is natural to investigate the relationship between $h_T(q)$ and $\frac{1}{[\lambda]_q!}$ for $T \in \posSYT(\bm, \lambda)$. Optimistically, one might hope that $h_T(\tau) \geq \frac{1}{[\lambda]_\tau!}$ for all real values $\tau \geq 0$, as this would imply \eqref{eq hik eval ineq}.
\ytableausetup{boxsize = 1em}
Unfortunately, this does not hold in general. For $\bm = (0,0,1,1,2,4)$, $\lambda = (3,2,1)$, and 
\begin{equation}
T = \begin{ytableau} 1 & 2 & 3 \\ 4 & 5 \\ 6 \end{ytableau} \in \posSYT(\bm, \lambda),
\end{equation}
we have 
\begin{equation}
h_T(\tau) [\lambda]_\tau! = \frac{1}{\tau + 1} < 1
\end{equation}
for all real numbers $\tau > 0$.

\section{Greedy Partition Coefficients and Tableaux Concatenation}
\label{dominanceSection}
In this section we use the greedy partition and concatenation of $P$-tableaux to obtain $e$-positivity results for $X_{\inc(P)}(\mathbf{x})$ and $X_{\inc(P)}(\mathbf{x}, q)$ in some special cases. In particular, Theorem~\ref{dominantPositivity} follows as a corollary of Theorem~\ref{concatTheorem}. Additionally, we obtain refinements of Theorem~\ref{HwangMaxChainThm} and Theorem~\ref{3freeKeyTabHwang}.

For partitions $\mu = (\mu_1, \mu_2,...)$ and $\nu = (\nu_1, \nu_2,...)$, write $\mu + \nu$ for the partition $(\mu_1 + \nu_1, \mu_2 + \nu_2,...)$. 
For compositions $\alpha = (\alpha_1,\alpha_2,...,\alpha_l)$ and $\beta = (\beta_1, \beta_2,...,\beta_m)$, 
let $\alpha \cdot \beta = (\alpha_1, \alpha_2,...,\alpha_l, \beta_1, \beta_2,..., \beta_m)$.
For $P$-arrays $A \in E_P(\col(\alpha))$ and $B \in E_P(\col(\beta))$, let the \emph{concatenation} $A\odot B \in E_P(\col(\alpha \cdot \beta))$ 
be the $P$-array with columns $A_1,A_2,...,A_l,B_1,B_2,...,B_m$, where $A_i$ and $B_i$ denote the columns of $A$ and $B$. Figure~\ref{fig array concat} shows the concatenation of arrays of shape $\col(3,2)$ and $\col(3,1)$. 

\begin{figure}[h]
\caption{Concatenation of tableaux}
\label{fig array concat}
\begin{align}
S = \begin{ytableau} 1 & 2 \\ 3 & 4 \\ 5 \end{ytableau} & & T = \begin{ytableau} 6 & 7 \\ 8 \\ 9 \end{ytableau} & & S\odot T = \begin{ytableau} 1 & 2 & 6 & 7 \\3 &4 & 8 \\ 5 & \none & 9 \end{ytableau}\,.
\end{align}
\end{figure}

Let $P$ be a \threeone -free poset, and let $\lambda^{gr} = \lambda^{gr}(P)$ be the greedy partition of $P$ as defined in Theorem~\ref{greedyPartThm}. 
For a positive integer $r$, let $\lambda^{gr,r} = (\lambda_1^{gr}, \lambda_2^{gr},...,\lambda_r^{gr})$ be the initial $r$ entries of $\lambda^{gr}$.
Recall the definition of $\keyIYT_P(\lambda)$ from Equation~\eqref{eq keyIT def}.
The main result of this section is the following.
\begin{theorem}
\label{concatTheorem}
Let $\mu$ be a partition such that $\mu_1 = r$ and $|\mu| = |\lambda^{gr, r}|$. If subsets of powerful $P$-tableaux $\keyIYT_P(\mu)$ and $\keyIYT_P(\nu)$ satisfying \eqref{eq keyIT def} are known, then for any partition $\nu$,
\begin{equation}
\keyIYT_P(\mu + \nu) = \{S\odot T \in \IYT_P(\mu + \nu) \mid S \in \keyIYT_P(\mu), \, T\in \keyIYT_P(\nu)\}.
\end{equation}
Furthermore, if $\keyIYT_P(\mu) = \strongIYT_P(\mu)$ and $\keyIYT_P(\nu) = \strongIYT_P(\nu)$, then $\keyIYT_P(\mu+\nu) = \strongIYT_P(\mu+\nu)$.
\end{theorem}

To prove Theorem \ref{concatTheorem}, we recall some facts about products of monomial symmetric functions. Lemma~\ref{mStructLemma} follows easily from Definition~\ref{mStructCoeffsDef}.
\begin{defin}
\label{mStructCoeffsDef}
For partitions $\mu, \nu, \eta$ such that $|\eta| = |\mu| + |\nu|$, let $g_{\mu, \nu}^{\eta}$ be the structure coefficients for monomial symmetric functions such that
\begin{equation}
m_{\mu}(\mathbf{x}) m_{\nu}(\mathbf{x}) = \sum_{\eta} g_{\mu, \nu}^{\eta} m_{\eta}(\mathbf{x}).
\end{equation}
\end{defin}

\begin{lemma}
\label{mStructLemma}
For partitions $\mu, \nu, \eta$ such that $|\eta| = |\mu| + |\nu|$, then
\begin{enumerate}
\item $g_{\mu,\nu}^{\mu + \nu} = 1$,
\item if $g_{\mu, \nu}^{\eta} \neq 0$, then $\mu + \nu \trianglerighteq \eta$,
\item if $g_{\mu, \nu}^{\eta} \neq 0$, then for each $i$, $\eta_i' \geq \max\{\mu_i', \nu_i'\},$ and
\item if $\mu_1 = r, \mu_r' \geq \nu_1'$, and $g_{\mu,\nu}^\eta \neq 0$, then either $\eta = \mu + \nu$, or $\eta_i' > \mu_i'$ for some $i \leq r$.
\end{enumerate}
\end{lemma}

\begin{proof}[Proof of Theorem~\ref{concatTheorem}]
Let $\mu$ be a partition such that $\mu_1 = r$ and $|\mu| = |\lambda^{gr,r}|$, and let $\nu$ be any partition.
We claim that 
\begin{equation}
\partEval(m_\mu^P(\mathbf{u}) m_{\nu}^P(\mathbf{u})) = 
\partEval(m_{\mu + \nu}^P(\mathbf{u}))
\end{equation}
Let $\eta \vdash |\mu + \nu|$ be a partition such that $g_{\mu, \nu}^\eta \neq 0$.
If $\eta_i' > \mu_i'$ for some $1 \leq i \leq r$, then $\partEval(m_\eta^P(\mathbf{u})) = 0$ by Lemma~\ref{mStructLemma}(3), Corollary~\ref{partialDominantCor}, and the fact that $|\mu| = |\lambda^{gr,r}|$.
Suppose $\mu_r' \geq \nu_1'$.
By Lemma~\ref{mStructLemma}(4), either $\eta = \mu + \nu$ or $\partEval(m_\mu^P(\bu)m_\nu^P(\bu)) = 0.$ Thus, as $g_{\mu,\nu}^{\mu + \nu} = 1$, we have  $\partEval(m_\mu^P(\bu)m_\nu^P(\bu)) = \partEval(m_{\mu + \nu}^P(\bu))$. 
Suppose $\mu_r' < \nu_1'$.
Then $(\mu + \nu)_i' > \mu_i'$ for some $i \leq r$. Therefore, by Lemma~\ref{mStructLemma}(2) and the fact that $(\mu + \nu)_j' \geq \mu_j'$ for all $1 \leq j \leq r$,
\begin{equation}
\eta_1' + \eta_2' + \dots + \eta_r' \geq (\mu + \nu)_1' + (\mu + \nu)_2' + \dots + (\mu + \nu)_r' > \mu_1' + \mu_2' + \dots + \mu_r'.
\end{equation}
Then by Corollary~\ref{partialDominantCor}, $\partEval(m_\eta^P(\mathbf{u})) = 0$. Therefore, $\partEval(m_{\mu}^P(\bu)m_\nu^P(\bu)) = \partEval(m_{\mu + \nu}^P(\bu)) = 0.$

As $\partEval(m_{\mu}^P(\bu)m_{\nu}^P(\bu)) = \partEval(m_{\mu + \nu}^P(\bu))$, we have
\begin{equation}
\partEval(m_{\mu + \nu}^P(\mathbf{u})) = \partEval\left(\sum_{S \in \keyIYT_P(\mu)} \sum_{T \in \keyIYT_P(\nu)} \mathbf{u}_{\colword(S\odot T)} \right).
\end{equation}
Let $S \in \keyIYT_P(\mu)$, and let $T \in \keyIYT_P(\nu)$.
Suppose the concatenation $R = S\odot T$ is an injective $P$-array. If $R$ has a right-unbalanced ladder $L$ between columns $r$ and $r+1$, then performing a ladder swap on $L$ gives an injective $P$-array $R'$ of shape $\col(\mu_1',..., \mu_{r-1}', \mu_r'+1, \nu_1' -1,\nu_2',...,\nu_{\nu_1}')$. But then, as $|\mu| = |\lambda^{gr,r}|$, $R'$ contradicts Corollary~\ref{partialDominantCor}, and $R$ has no right-unbalanced ladders between columns $r$ and $r+1$. This implies $\mu_r' \geq \nu_1'$, so $R$ is a $P$-tableau of shape $\mu + \nu$. Furthermore, if $S$ and $T$ are strong $P$-tableaux, then $R$ is a strong $P$-tableau.
\end{proof}

Theorem~\ref{dominantPositivity} is an application of Theorem~\ref{concatTheorem}. We recall the statement of Theorem~\ref{dominantPositivity}.

\begin{theorem}[Theorem~\ref{dominantPositivity}]
\label{thm dom pos proof}
Let $P$ be a natural unit interval order on $[n]$, and let $\lambda^{gr}(P) = (\lambda_1^{gr}, \lambda_2^{gr},...,\lambda_l^{gr})$ be the \emph{greedy partition} of $P$ as defined in Theorem~\ref{greedyPartThm}.
Let $S$ be a subset of $[n]$ that contains no adjacent entries, and let $k_i$ be a positive integer for each $i \in S$. 
If $\lambda$ is a partition such that the conjugate partition $\lambda'$ 
satisfies $\lambda'_i = \lambda_i^{gr} - k_i$ if $i \in S$, $\lambda'_i = \lambda^{gr}_i + k_{i-1}$ if $i-1 \in S$, and $\lambda'_i = \lambda_{i}^{gr}$ otherwise, 
then
\begin{equation}
c_{\lambda}^P(q) = \sum_{T\, \in\, \strongSYT_P(\lambda)} q^{\inv_P(T)}.
\end{equation}
\end{theorem} 

\begin{proof}
Let $S \subset [n]$ have no adjacent elements. Let $(k_i)_{i \in S}$ be positive integers such that the composition $\mu \vDash n$ with parts
\begin{equation}
\mu_i = \begin{cases}
            \lambda_i^{gr} - k_i &\text{if } i \in S, \\
            \lambda_i^{gr} + k_{i-1} & \text{if }i - 1 \in S, \\
            \lambda_i^{gr} & \text{otherwise,}
\end{cases}
\end{equation}
is a partition.
Observe that $\mu$ can be written as the concatenation of length 1 partitions $(\lambda_i^{gr})$ and length 2 partitions $(\lambda_i^{gr} - k_i, \lambda_{i+1}^{gr} + k_i)$.
Write $\mu = \nu^1 \cdot \nu^2 \cdots \nu^m$ for this concatenation, and for each $j \in [m]$, write $\mu^{j} = \nu^1 \cdot \nu^2 \cdots \nu^j$. 
Let $\lambda = \mu'$, $\lambda^j = (\mu^j)'$, and $r_j = \lambda^j_1$.
Observe that for each $j \in [m]$, $|\lambda^{j}| = |\lambda^{gr, r_j}|$. 
As $\lambda^1$ is either a one or two column shape, $\keyIYT_P(\lambda^1) = \strongIYT_P(\lambda^1)$ by Lemma~\ref{twoColDecLemma}. Suppose that $\keyIYT_P(\lambda^j) = \strongIYT_P(\lambda^j)$ for some $j \in [m]$. By Lemma~\ref{twoColDecLemma}, $\keyIYT_P((\nu^{j+1})') = \strongIYT_P((\nu^{j+1})')$. Then as $|\lambda^j| = |\lambda^{gr, r_j}|$, $\keyIYT_P(\lambda^{j+1}) = \strongIYT_P(\lambda^{j+1})$ by Theorem~\ref{concatTheorem}. By induction, $\keyIYT_P(\lambda) = \strongIYT_P(\lambda)$ as desired.
\end{proof}

\begin{remark}
A key ingredient to the proof of Theorem~\ref{thm dom pos proof} is the concatenation of two-column key $P$-tableaux. Using Theorem~\ref{concatTheorem}, key $P$-tableaux could be determined for various other shapes by concatenating known key $P$-tableaux. In particular, the condition that $S$ has no adjacent entries in Theorem~\ref{thm dom pos proof} could be relaxed if a description of key $P$-tableaux for partitions with more than two columns were known.
\end{remark}

\begin{exmp}
Suppose for some poset $P$, $\lambda^{gr} = (18, 12, 11, 11, 5, 2).$ Taking $S = \{1,4,6\}$ and $(k_1,k_4,k_6) = (2,3,1)$. Then for $\lambda = (7, 6^7, 3^3, 2^3, 1^2) = (16, 14, 11,8, 8 ,1,1)'$, we have $c_\lambda^P(q) = \sum_{T\, \in\, \strongSYT_P(\lambda)} q^{\inv_P(T)}$.

\ytableausetup{boxsize = 0.5em}
\begin{equation*}
\raisebox{-1.2cm}{$\lambda^{gr'} =\ \ $} 
\begin{ytableau} 
\ & \ & \ & \ & \ & \  \\
\ & \ & \ & \ & \ & *(red) \ \\
\ & \ & \ & \ & \ \\
\ & \ & \ & \ & \ \\
\ & \ & \ & \ & \ \\
\ & \ & \ & \ \\
\ & \ & \ & \ \\
\ & \ & \ & \ \\
\ & \ & \ & *(red) \ \\
\ & \ & \ & *(red) \ \\
\ & \ & \ & *(red) \ \\
\ & \ \\
\ \\
\ \\
\ \\
\ \\
*(red) \ \\
*(red) \
\end{ytableau}
\ \ \ \ \ \ \ 
\raisebox{-1.2cm}{$\lambda = \ \ $}
\begin{ytableau} 
\ & \ & \ & \ & \ & \ & *(red) \ \\
\ & \ & \ & \ & \ \\
\ & \ & \ & \ & \ \\
\ & \ & \ & \ & \ \\
\ & \ & \ & \ & \ \\
\ & \ & \ & \  & *(red) \ \\
\ & \ & \ & \  & *(red) \ \\
\ & \ & \ & \  & *(red) \ \\
\ & \ & \  \\
\ & \ & \  \\
\ & \ & \  \\
\ & \ \\
\ & *(red) \ \\
\ & *(red) \  \\
\ \\
\ \\
\end{ytableau}
\end{equation*}
\ytableausetup{boxsize = 2em}
\end{exmp}

\begin{exmp}
Let $P = P_\bm$ for $\bm = (0,0,1,1,3),$ as shown in Figure~\ref{fig 5 elem poset}. Then $\lambda^{gr}(P) = (3,1,1)$. We apply Theorem~\ref{thm dom pos proof} with $S = \emptyset$ to obtain $c_{(3,1,1)}^P(q) = \sum_{T \in \strongSYT_P(3,1,1)} q^{\inv_P(T)} = q^2 + q^3$. With $S = \{1\}$ and $k_1 = 1$, we obtain $c_{(3,2)}^P(q) = \sum_{T \in \strongSYT_P(3,2)} q^{\inv_P(T)} = q^2 + q^3.$ Figure~\ref{fig greedy strong tabs} shows all standard strong $P$-tableau of shape $(3,1,1)$ and $(3,2)$. The collection of all standard powerful $P$-tableaux is shown in Figure~\ref{fig 5 elem pow tabs} in Appendix~\ref{Appendix}.
\end{exmp}

\begin{figure}[h]
\caption{Strong $P_\bm$-tableaux of shape $(3,1,1)$ and $(3,2)$ for $\bm = (0,0,1,1,3)$}
\label{fig greedy strong tabs}
\ytableausetup{boxsize = 1.2em}
\setlength{\tabcolsep}{20pt}
\renewcommand{\arraystretch}{1.5}
\vspace{.1cm}
\begin{tabular}{cccc}
$\begin{ytableau}
1 & 2 & 4 \\ 
3 \\
5
\end{ytableau}$ & 
$\begin{ytableau}
1 & 4 & 2 \\ 
3 \\
5
\end{ytableau}$ & 
$\begin{ytableau} 1 & 2 & 3 \\ 4 & 5 \end{ytableau}$ & 
$\begin{ytableau} 2 & 1 & 3\\ 5 & 4 \end{ytableau}$ \\
$q^2$ & $q^3$ & $q^2$ & $q^3$ 
\end{tabular}
\ytableausetup{boxsize = 1.5em}
\end{figure}


For a partition $\lambda = (\lambda_1, \lambda_2,...,\lambda_l)$, write $\lambda^- = (\lambda_1-1, \lambda_2 - 1,..., \lambda_l-1)$ for the partition obtained by removing the first column of $\lambda$. For a $P$-tableau $T \in \T_P(\lambda)$, write $T^- \in \T_P(\lambda^-)$ for the $P$-tableau obtained by removing the first column of $T$. The following theorem gives an interpretation of \cite[Theorem~4.37]{Hwang}, which we recall in Theorem~\ref{HwangMaxChainThm}, in terms of $P$-tableaux. The following corollary interprets Corollary~\ref{3freeKeyTabHwang} in terms of $P$-tableaux.

\begin{theorem}
\label{maxChainThm}
Let $P$ be a \threeone -free poset with maximal chain length $r$.
For a partition $\lambda$ with $\ell(\lambda) = r$,
\begin{equation}
\keySSYT_P(\lambda) = \{ T \in \T_P(\lambda) \mid T^- \in \keySSYT_P(\lambda^-)\}.
\end{equation}
Furthermore, 
\begin{itemize}
\item if $\keySSYT_P(\lambda^-) \subseteq \decSSYT_P(\lambda^-)$, then $\keySSYT_P(\lambda) \subseteq \decSSYT_P(\lambda)$, and
\item if $\keySSYT_P(\lambda^-) = \strongSSYT_P(\lambda^-)$, then $\keySSYT_P(\lambda) = \strongSSYT_P(\lambda)$.
\end{itemize}
\end{theorem}
\begin{proof}
Let $P$ be a \threeone -free poset with maximal chain of length $r$ and let $\lambda$ be a partition such that $\ell(\lambda) = r$.
By Theorem~\ref{HwangMaxChainThm},
\begin{equation}
m_{\lambda}^P(\mathbf{u}) = e_r^P(\mathbf{u})m_{\lambda^-}^P(\mathbf{u}).
\end{equation}
We can then write
\begin{equation}
m_{\lambda}^P(\mathbf{u}) \equiv \sum_{S \in\, \T_P(1^r)} \sum_{T \in\, \keySSYT_P(\lambda^-)} \mathbf{u}_{\colword(S)}\mathbf{u}_{\colword(T)} \mod I_P.
\end{equation}
Let $S \in \T_P(1^r)$ and $T \in \keySSYT_P(\lambda^-)$. As $P$ has maximal chain length $r$, $S\odot T$ has no right-unbalanced ladders between columns $1$ and $2$. Thus, by Lemma~\ref{lem strong concatenation}, $S\odot T$ is a powerful $P$-tableau. Therefore,
\begin{equation}
m_{\lambda}^P(\mathbf{u}) = \sum_{U \in\, \keySSYT_P(\lambda)} \mathbf{u}_{\colword(U)} \mod I_P, 
\end{equation}
where $\keySSYT_P(\lambda) = \{S\odot T \in \T_P(\lambda) : S \in \T_P(1^r),\, T \in \keySSYT_P(\lambda^-)\}.$
\end{proof}


\begin{corollary}
\label{twoChainCor}
Let $P$ be a $\mathbf{3}$-free poset. Then for any partition $\lambda$,
\begin{equation}
m_{\lambda}^P(\mathbf{u}) \equiv \sum_{T \in\, \decSSYT_P(\lambda)} \mathbf{u}_{\colword(T)} \mod I_P.
\end{equation}
\end{corollary}
\begin{proof}
Let $P$ be a $\mathbf{3}$-free poset. Observe that $\T_P(\lambda) = \emptyset$ if $\ell(\lambda) > 2$.
Therefore, it suffices to consider $\lambda = (\lambda_1, \lambda_2)$.
As a consequence of Lemma~\ref{hookDecomp}, if $\lambda_2 = 0$, then
\begin{equation}
m_\lambda^P(\mathbf{u}) \equiv \sum_{T\, \in \, \decSSYT_P(\lambda)} \mathbf{u}_{\colword(T)} \mod I_P.
\end{equation} 
Suppose that $\keySSYT_P(\lambda) = \decSSYT_P(\lambda)$ whenever $\lambda_2 < k$. Let $\lambda_2 = k$, then $\lambda^-_2 = k-1$ and $\keySSYT_P(\lambda^-) = \decSSYT_P(\lambda^-)$ by the inductive hypothesis. Then by Theorem~\ref{maxChainThm}, $\keySSYT_P(\lambda) = \{T \in \decSSYT_P(\lambda) \mid T^- \in \decSSYT_P(\lambda^-)\}$. As $T^-$ is a powerful $P$-tableau whenever $T$ is a powerful $P$-tableau, $\keySSYT_P(\lambda) = \decSSYT_P(\lambda)$, as desired.
\end{proof}

\section{Positivity for $\lambda = (k, 2^a)$.}
\label{sectionk2}
Fix $P$ to be a natural unit interval order.
In this section, we determine a set $\keySSYT_P(\lambda)$ for $\lambda$ of the form $(k, 2^a)$. Hence, we obtain a $\bu$-positive expansion for $m_{(k,2^a)}^P(\bu),$ as outlined in Section~\ref{sec prev mlamP results}. We use this $\bu$-positive expansion to prove Theorem~\ref{thm k2positivity}. 
For a positive integer $k$, identify one-row powerful $P$-tableaux $\decSSYT_P(k)$ with powersum words of length $k.$ In particular, observe that $\decSSYT_P(2)$ is the set of words $\ba = a_1a_2$ such that $a_1 \dote_P a_2$, meaning $a = b$ or $a$ is incomparable to $b$. Recall that for $\bv, \bw \in P^k$, we write $\bv \equiv \bw \mod I_P$ to mean $\bu_{\bv} \equiv \bu_\bw \mod I_P$.
\begin{defin}
\label{def slots}
Let $\ba = a_1a_2 \in \decSSYT_P(2)$ and $\bb = b_1b_2 \dots b_k \in \decSSYT_P(k)$ be powersum words. We say $\ba$ \emph{slots into} $\bb$ if one of the following hold.
\begin{enumerate}
\item $\ba\bb = a_1a_2b_1b_2\dots b_k$ is a powersum word,
\item $a_2 <_P b_j$ for all $j \in [k]$, $a_1 \dote_P b_1$, and $a_1b_2b_3 \dots b_k$ is a powersum word,
\item there exists $j$ such that $b_1 <_P b_2 <_P \dots <_P b_j <_P a_2$,\ \ $b_j <_P a_1$, and $a_2 \not>_P b_{j+1}$, or
\item there exists $j$ such that $b_1 <_P b_2 <_P \dots <_P b_j <_P a_2$,\ \  $b_j \dote_P a_1$, and $a_1 \not>_P b_{j+1}$.
\end{enumerate}
If $\ba$ slots into $\bb$, define $\slot(\ba, \bb)$ by
\begin{equation}
\slot(\ba,\bb) = \begin{cases}
	\ba\bb = a_1a_2b_1b_2 \dots b_k & \text{if } (\ba, \bb) \text{ satisfies condition (1)},\\
	a_2 b_1 a_1 b_2 b_3 \dots b_k & \text{if } (\ba, \bb) \text{ satisfies condition (2)},\\
	b_1b_2\dots b_j a_1 a_2 b_{j+1} b_{j+2} \dots b_k & \text{if } (\ba, \bb) \text{ satisfies condition (3)},\\
	b_1b_2 \dots b_j a_2 a_1 b_{j+1} b_{j+2} \dots b_k & \text{if } (\ba, \bb) \text{ satisfies condition (4)}.
\end{cases}
\end{equation}
We say $\ba$ \emph{strictly slots into} $\bb$ if $(\ba,\bb)$ satisfies condition (2), (3), or (4).
\end{defin}

\begin{defin}
\label{def KP}
Let $k \in \mathbb{Z}_{\geq 0}$, let $T \in \decSSYT_P(k, 2^a),$ and let $A$ be the powerful array such that $\tab(A) = T$. We define \emph{non-slotting $P$-tableaux} as follows.
If $k \leq 2$, $T$ is a non-slotting $P$-tableau. Otherwise, let $A_i$ be the unique row of length $k$ in $A$. Then $T$ is a non-slotting $P$-tableau if row $A_{i+1}$ does not strictly slot into row $A_i$.
Write $\K_P(k,2^a)$ for the set of non-slotting $P$-tableaux of shape $(k, 2^a)$, and write $\SK_P(k,2^a)$ for the set of standard non-slotting $P$-tableaux.
\end{defin}

Recall the definition of the ideal $I_P$ from Definition~\ref{def IP}.
\begin{theorem}
\label{thm k2a key tabs}
Let $P$ be a natural unit interval order, and let $k > 2$. Then 
\begin{equation}
m_\lambda^P(\bu) \equiv \sum_{T\, \in \, \K_P(k2^a)} \bu_{\colword(T)} \mod I_P.
\end{equation}
Hence, $c_{k2^a}^P(q) = \sum_{T \in \SK_P(k2^a)} q^{\inv_P(T)}.$
\end{theorem}

\begin{exmp} Consider $\bm = (0,0,1,1,1,3,4,6) \in \mathbb{E}_8$ and $\lambda = (4,2,2)$. Then $c_{\lambda}^{P_\bm}(q) = q^7 + 2q^6 + q^5$. Figure~\ref{fig slot tabs} shows all standard powerful $P_{\bm}$-tableaux $T$ of shape $\lambda$, colored to show their decomposition into powersum words and marked with $q^{\inv_{P_\bm}(T)}$. A tableau $T$ is marked 'slotting' if $T \not\in \SK_{P_\bm}(\lambda)$ and 'non-slotting' if $T \in \SK_{P_\bm}(\lambda)$.
\ytableausetup{boxsize=1.2em}
\setlength{\tabcolsep}{20pt}
\renewcommand{\arraystretch}{1.5}
\begin{figure}[h]
\caption{Slotting and non-slotting powerful $P_{\bm}$-tableaux of shape $(4,2,2)$ for $\bm = (0,0,1,1,1,3,4,6)$}
\label{fig slot tabs}
\vspace{.5cm}
\begin{tabular}{ccc}
$\begin{ytableau} *(red!20) 1 &*(red!20) 3 &*(red!20) 2 &*(red!20) 5 \\
                  *(green!20) 4 & *(green!20) 6 \\
                  *(yellow!20) 7 & *(yellow!20) 8 
\end{ytableau}$ 
&
$\begin{ytableau} *(red!20)1 & *(red!20)3 & *(red!20)5 & *(red!20)2 \\
                  *(green!20)4 & *(green!20)6 \\
                  *(yellow!20)7 & *(yellow!20)8 
\end{ytableau}$
&
$\begin{ytableau} *(red!20)1 & *(red!20)2 & *(red!20)3 & *(red!20)5 \\
                  *(green!20) 4 & *(green!20)6 \\
                  *(yellow!20) 7 & *(yellow!20) 8 
\end{ytableau}$ 
\\
$q^6$ & $q^7$ & $q^5$ \\
slotting & slotting & non-slotting \\
& & \\
$\begin{ytableau} *(red!20) 1 & *(red!20)2 & *(red!20)5 & *(red!20)3 \\
                  *(green!20) 4 & *(green!20) 6 \\
                  *(yellow!20) 7 & *(yellow!20) 8 
\end{ytableau}$ 
&
$\begin{ytableau} *(red!20)2 & *(red!20)1 & *(green!20) 3 & *(green!20) 5 \\
                  *(green!20) 6 & *(green!20) 4 \\
                  *(yellow!20) 8 & *(yellow!20) 7 
\end{ytableau}$ 
&
$\begin{ytableau} *(red!20)2 & *(red!20)1 & *(green!20)5 & *(green!20)3 \\
                  *(green!20)6 & *(green!20)4 \\
                  *(yellow!20) 8 & *(yellow!20) 7 
\end{ytableau}$ 
\\
$q^6$ & $q^6$ & $q^7$ \\
non-slotting & non-slotting & non-slotting
\end{tabular}
\end{figure}
\ytableausetup{boxsize = 1.5em}
\end{exmp}

Recall that for words $\bw, \bv \in P^k$, we write $\bw \equiv \bv \mod I_P$ to mean $\bu_\bw \equiv \bu_\bv \mod I_P$. We prove Theorem~\ref{thm k2a key tabs} below, with many of the technical details postponed to later in the section.

\begin{proof}[Proof of Theorem~\ref{thm k2a key tabs}]
We want to show that $m_{k2^a}^P(\bu) = \sum_{T \in \K_P(k,2^a)} \bu_{\colword(T)}$. Recall the symmetric function identites
\begin{equation}
m_{2^a}(\bx)m_k(\bx) = m_{k2^a}(\bx) + m_{(k+2)2^{a-1}}(\bx)
\end{equation}
for $k \neq 2$, and for $k = 2$,
\begin{equation}
m_{2^a}(\bx)m_2(\bx) = (a+1) m_{2^{a+1}}(\bx) + m_{42^{a-1}}(\bx).
\end{equation}
Hence,
\begin{equation}
\label{eq m2aP mkP}
m_{2^a}^P(\bx)m_k(\bu) = m_{k2^a}^P(\bu) + m_{(k+2)2^{a-1}}^P(\bu)
\end{equation}
for $k \neq 2$, and for $k = 2$,
\begin{equation}
\label{eq m2aP m2P}
m_{2^a}^P(\bu)m_2^P(\bu) = (a+1) m_{2^{a+1}}(\bu) + m_{42^{a-1}}^P(\bu).
\end{equation}
From Lemma~\ref{twoColDecLemma} and Lemma~\ref{hookDecomp}, we have that $m_\lambda^P(\bu) \equiv \sum_{T \in \decSSYT_P(\lambda)}  \bu_{\colword(T)}$ modulo $I_P$ when $\lambda$ is either a two-column or hook shape.
Hence,
\begin{equation}
m_{2^a}^P(\bu) m_k^P(\bu) \equiv \sum_{T \, \in \, \decSSYT_P(2^a)} \sum_{\bz\, \in\, \decSSYT_P(k)} \bu_{\colword(T)} \bu_{\bz} \mod I_P.
\end{equation}
In Definition~\ref{def FP phi}, we define sets $\F_P(2^a, k) \subseteq \decSSYT_P(2^a) \times \decSSYT_P(k)$ and define a map 
\begin{equation}
\varphi_{a,k}: \decSSYT_P(2^a) \times \decSSYT_P(k) \to \decSSYT_P(2^a) \times \decSSYT_P(k) \sqcup \decSSYT_P(2^{a-1}) \times \decSSYT_P(k+2).
\end{equation}
In Lemma~\ref{lem phi inj}, Lemma~\ref{lem phi im FP}, and Lemma~\ref{lem phi sur}, we will show that $\varphi_{a,k}$ is an injection from $\decSSYT_P(2^a) \times \decSSYT_P(k)$ to $\F_P(2^a,k) \sqcup \F_P(2^{a-1}, k+2)$. From Lemma~\ref{lem ej IP}, Lemma~\ref{lem Ej inj}, and the definition of $\varphi_{a,k}$, if $(S, \by) = \varphi_{a,k}(T, \bz)$ for $T \in \decSSYT_P(2^a)$ and $\bz \in \decSSYT_P(k)$, we have
\begin{equation}
\colword(S)\by \equiv \colword(T)\bz \mod I_P.
\end{equation}
Hence, applying $\varphi_{a,k}$ to each $(T, \bz) \in \decSSYT_P(2^a) \times \decSSYT_P(k)$, we have
\begin{align}
m_{2^a}^P(\bu) m_k^P(\bu) &= \sum_{T \in \keySSYT_P(2^a)} \sum_{\bz \in \keySSYT_P(k)} \bu_{\colword(T)}\bu_{\bz} \\
\label{eq FP exp}
&\equiv \sum_{(S, \by) \in \F_P(2^a, k)} \bu_{\colword(S)} \bu_{\by} + \sum_{(S,\by) \in \F_P(2^{a-1}, k+2)} \bu_{\colword(S)} \bu_{\by} \mod I_P. 
\end{align}

In Definition~\ref{def mult}, we will define a map $\mult_{a,k}: \F_P(2^a, k) \to \T_P(k,2^a)$ such that 
\begin{equation}
\colword(T) \bz \equiv \colword(\mult_{a,k}(T, \bz)) \mod I_P 
\end{equation}
for all $(S, \bz) \in \F_P(2^a, k).$ In Lemma~\ref{lem mult im}, we show that the image of $\mult_{a,k}$ is $\K_P(k, 2^a)$, and in Lemma~\ref{lem mult sur}, we show that for $T \in \K_P(k, 2^a)$, we have
\begin{equation}
|\mult_{a,k}^{-1}(T)| = \begin{cases}
							a + 1 & \text{if } k = 2, \\
							1     & \text{otherwise.}
						\end{cases}
\end{equation}

Thus, applying $\mult_{a,k}$ and $\mult_{a-1, k+2}$ to each term in \eqref{eq FP exp}, we have
\begin{equation}
\label{eq m2a mk KP}
m_{2^a}^P(\bu) m_k^P(\bu) \equiv \sum_{T\, \in\, \K_P(k, 2^a)} \bu_{\colword(T)} + \sum_{R\, \in\, \K_P(k+2,2^{a-1})} \bu_{\colword(R)} \mod I_P
\end{equation}
for $k \neq 2$, and for $k = 2$,
\begin{equation}
\label{eq m2a m2 KP}
m_{2^a}^P(\bu) m_2^P(\bu) \equiv (a+1)\sum_{T\, \in\, \K_P(2^{a+1})} \bu_{\colword(T)} + \sum_{R\, \in\, \K_P(4,2^{a-1})} \bu_{\colword(R)} \mod I_P,
\end{equation}
matching \eqref{eq m2aP mkP} and \eqref{eq m2aP m2P}. In Lemma~\ref{lem KP induction}, we will show that 
\begin{equation}
m_{k, 2^a}^P(\bu) \equiv \sum_{T\, \in \, \K_P(k, 2^a)} \bu_{\colword(T)} \mod I_P
\end{equation}
follows from \eqref{eq m2a mk KP} and \eqref{eq m2a m2 KP}.
\end{proof}

\begin{lemma}
Suppose $\ba = a_1a_2 \in \decSSYT_P(2)$ slots into $\bb = b_1b_2 \dots b_k \in \decSSYT_P(k)$. Then $\slot(\ba, \bb) \in \decSSYT_P(k+2),$ and $\ba\bb \equiv \slot(\ba, \bb) \mod I_P$.
\end{lemma}
\begin{proof}
We prove the claim case by case. If $\ba, \bb$ satisfy condition (1), then $\slot(\ba, \bb) = \ba\bb$ and $\slot(\ba,\bb)$ is a powersum word. Suppose $\ba,\bb$ satisfy condition (2). As $a_2 <_P b_1$, $a_1 \dote_P a_2$, and $a_1 \dote_P b_1$, we have $a_1 a_2 b_1 \equiv a_2 b_1 a_1 \mod I_P$ and $\ba\bb \equiv a_2b_1a_1b_2b_3 \dots b_k \mod I_P$. As $a_1b_2 b_3 \dots b_k$ is a powersum word, $\slot(\ba, \bb) = a_2b_1a_1b_2b_3 \dots b_k$ is a powersum word. If $\ba, \bb$ satisfy condition (3), then $b_i <_P a_1$ and $b_i <_P a_2$ for all $i \leq j$. Hence, $\ba\bb \equiv b_1 b_2 \dots b_j a_1 a_2 b_{j+1} b_{j+2} \dots b_k = \slot(\ba,\bb) \mod I_P$. As $a_2 \not>_P b_{j+1}$, $\slot(\ba,\bb)$ has no $P$-descents. As $\bb$ is a powersum word, there is some $t > j$ such that $b_j \dote_P b_t$. Thus, $b_t \not>_P a_2$ and $\slot(\ba,\bb)$ is a powersum word. Suppose $\ba,\bb$ satisfy condition (4). As $b_1 <_P b_2 <_P \dots <_P b_{j-1} <_P b_j <_P a_2$, $a_1 \dote_P a_2$, and $P$ is \threeone -free, we have $b_1 <_P b_2 <_P \dots <_P b_{j-1} <_P a_1$. Hence, as $a_1 \dote_P b_j$, $\ba\bb \equiv b_1 b_2 \dots b_{j-1} a_1 a_2 b_j b_{j+1} \dots b_k \equiv b_1 b_2 \dots b_{j-1} b_j a_2 a_1 b_{j+1} \dots b_k = \slot(\ba, \bb) \mod I_P$. As $\bb$ is a powersum word, there is some minimal $t > j$ such that $b_j \dote_P b_t$. As $t$ is minimal among such indices and $a_2 \not>_P b_{j+1}$, we have $a_2 \not>_P b_t$. As $P$ is \twotwo -free, $b_j <_P a_1$, $b_j \dote_P a_2$, $a_1 \dote_P a_2$, $b_j \dote_P b_t$, and $a_2 \not>_P b_t$, we have $a_2 \dote_P b_t$ and $\slot(\ba,\bb)$ is a powersum word.
\end{proof}

\begin{lemma}
\label{lem slot inj}
Let $\bw = w_1w_2\dots w_{k+2} \in \decSSYT_P(k + 2)$. Then there is a unique pair $(\ba, \bb) \in \decSSYT_P(2) \times \decSSYT_P(k)$ such that $\bw = \slot(\ba,\bb)$.
\end{lemma}
\begin{proof}
Let $r$ be the smallest index such that $w_r \dote_P w_{r+1}$. If $r = 1$, then $\bw = \slot(\ba, \bb)$ where $\ba = w_1w_2$ and $\bb = w_3w_4\dots w_{k+2}$. Conversely, if $\bw = \slot(\ba,\bb)$ and $r = 1$, then $\ba,\bb$ satisfy condition (1). Suppose $r > 1$ and $w_{r-1} <_P w_{r+1}$. Observe that if $\bw = \slot(\ba,\bb)$, then $\ba,\bb$ satisfy condition (3). Let $\ba = w_r w_{r+1}$ and $\bb = w_1 w_2 \dots w_{r-1} w_{r+2} \dots w_{k+2}$. Then $\bw = \slot(\ba,\bb)$.

Suppose $r > 1$, $w_{r-1} \dote_P w_{r+1},$ and $w_{r-1} <_P w_j$ for $r+2 \leq j \leq k+2$. If $r > 2$, then as $P$ is \threeone -free, $w_{r-2} <_P w_{r-1} <_P w_r$, $w_{r-1} \dote_P w_{r+1}$, and $w_{r} \dote_P w_{r+1}$, we have $w_{r-2} <_P w_j$ for $r-1 \leq j \leq k+2$, which contradicts the fact that $\bw$ is a powersum word. Hence, we have that $r = 2$. Furthermore, if $\bw = \slot(\ba,\bb)$, then $\ba, \bb$ satisfy condition (3), and we must have $\ba = w_3w_1$ and $\bb = w_2w_4w_5\dots w_{k+2}$. As $\bw$ is a powersum word, there is some $t > 3$ such that $w_3 \dote_P w_t$. As $P$ is \threeone -free, $w_1 <_P w_2$, $w_1 \dote_P w_3,$ $w_2 \dote_P w_3$, and $w_1 <_P w_t$, we have $w_2 \dote_P w_t$. Similarly, as $w_3 \not>_P w_4$, we have $w_2 \not>_P w_4$. Therefore $\bb = w_2w_4w_5 \dots w_{k+2}$ is a powersum word and $\bw = \slot(\ba, \bb)$. 

Suppose $r > 1$, $w_{r-1} \dote_P w_{r+1},$ and $w_{r-1} \not<_P w_j$ for some $r+2 \leq j \leq k+2$. If $\bw = \slot(\ba,\bb)$ then $\ba,\bb$ must satisfy condition (4), and therefore $\ba = w_{r+1}w_r$ and $\bb = w_1w_2 \dots w_{r-1}w_{r+2}w_{r+3}\dots w_{k+2}$. As $P$ is \threeone -free, $w_{r-1} <_P w_r$, $w_{r-1} \dote_P w_{r+1}$, $w_r \dote_P w_{r+1}$, and $w_{r+1} \not >_P w_{r+2}$, we have $w_{r-1} \not>_P w_{r+2}$. As $w_{r-1} \not<_P w_j$ for some $r+2 \leq j \leq k+2$, this means $\bb$ is a powersum word, and $\bw = \slot(\ba, \bb)$.
\end{proof}

\begin{defin}
Let $a > 0$, and let $T \in \decSSYT_P(2^a)$. Write $v_1 <_P v_2 <_P \dots <_P v_a$ and $w_1 <_P w_2 <_P \dots <_P w_a$ for the columns of $T$. For $i \in [a]$, let $s \leq i \leq t$ be such that $\{v_s, w_s, v_{s+1}, w_{s+1},\dots, v_t, w_t\}$ is a ladder in $T$. 
Define the $i^{th}$ \emph{ejection} $\ej_i(T) \in \decSSYT_P(2^{a-1})$
to be the $P$-tableau with columns
\begin{equation}
v_1 <_P \dots <_P v_{s-1} <_P w_s <_P w_{s+1} <_P \dots <_P w_{i-1} <_P v_{i+1} <_P v_{i+2} <_P \dots <_P v_a
\end{equation}
\begin{equation}
w_1 <_P \dots <_P w_{s-1} <_P v_s <_P v_{s+1} <_P \dots <_P v_{i-1} <_P w_{i+1} <_P w_{i+2} <_P \dots <_P w_a
\end{equation}
if $v_s \dote_P w_s \dote_P v_{s+1} \dote_P \dots \dote_P v_t \dote_P w_t$ and columns
\begin{equation}
v_1 <_P \dots <_P v_{i-1} <_P w_{i+1} <_P w_{i+2} <_P \dots <_P w_t <_P v_{t+1} <_P v_{t+2} <_P \dots <_P v_a
\end{equation}
\begin{equation}
w_1 <_P \dots <_P w_{i-1} <_P v_{i+1} <_P v_{i+2} <_P \dots <_P v_t <_P w_{t+1} <_P w_{t+2} <_P \dots <_P w_a
\end{equation}
if $w_s \dote_P v_s \dote_P w_{s+1} \dote_P v_{s+1} \dote_P \dots \dote_P w_t \dote_P v_t$.
As $P$ is \twotwo -free, $\ej_i(T)$ is well-defined.
\end{defin}

\begin{lemma}
\label{lem ej IP}
If $T \in \decSSYT_P(2^a)$, then $\colword(T) \equiv \colword(\ej_i(T)) T_{i,1}T_{i,2} \mod I_P$.
\end{lemma}
\begin{proof}
Let $v_1 <_P v_2 <_P \dots <_P v_a$ and $w_1 <_P w_2 <_P \dots <_P w_a$ be the columns of $T$. 
Then $\colword(T) = v_a v_{a-1} \dots v_1 w_a w_{a-1} \dots w_1$.
Let $S$ be the $P$-tableau obtained from $T$ by removing $w_i$. As $w_i$ commutes with each $w_j$ modulo $I_P$, we have $\colword(T) \equiv \colword(S)w_i \mod I_P$. Observe that $S$ has a unique unbalanced ladder $L$ containing $v_i$. 
If $v_s \dote_P w_s \dote_P \dots \dote_P v_t \dote_P w_t$, then $L = \{v_s, w_s, v_{s+1},\dots, w_{i-1}, v_i\}$. If instead we have $w_s \dote_P v_s \dote_P w_{s+1} \dote_P \dots \dote_P w_t \dote_P v_t$, then $L = \{v_i, w_{i+1}, v_{i+1}, \dots, w_t, v_t\}$. Observe that we obtain $\ej_i(T)$ from $S$ by performing a ladder swap on $L$ and removing $v_i$. We then have
\begin{equation}
\colword(T) \equiv \colword(S)w_i \equiv \colword(\ej_i(T))v_iw_i \mod I_P
\end{equation}
as desired.
\end{proof}

\begin{lemma}
\label{lem Ej inj}
Let $P$ be a natural unit interval order. The map 
\begin{equation}
\Ej_i: \decSSYT_P(2^a) \to \decSSYT_P(2^{a-1}) \times \decSSYT_P(2)
\end{equation}
given by $\Ej_i(T) = (\ej_i(T), T_{i,1}T_{i,2})$ is an injection. Furthermore, if $S = \ej_i(T)$, then for all $j < i$, $S_{j,2} <_P T_{i,1}$ and $S_{j,2} <_P T_{i,2}$, and for each $k \geq i$, $S_{k,1} >_P T_{i,1}$ and $S_{k, 2} >_P T_{i,2}$.
\end{lemma}
\begin{proof}
Observe that $\ej_i(T)$ can be obtained from $T$ by removing $T_{i,2}$, performing a ladder swap on the unbalanced ladder containing $T_{i,1}$, and removing the entry $T_{i,1}$. Each step in this process is reversible, so $\Ej_i$ is an injection. The fact that $S_{1,2} <_P \dots <_P S_{i-1,2} <_P T_{i,1} <_P S_{i,2} <_P \dots <_P S_{a-1, 2}$ and $S_{1,2} <_P \dots <_P S_{i-1,2} <_P T_{i,2} <_P S_{i,2} <_P \dots <_P S_{a-1,2}$ follows from the definition of $\ej_i$ and the fact that $P$ is \twotwo -free.  
\end{proof}

\begin{defin}
\label{def FP phi}
Write $\F_P(2^a, k) \subseteq \decSSYT_P(2^a) \times \decSSYT_P(k)$ for the set of pairs $(T, \bz)$ such that there is no row $i$ such that $T_{i,1}T_{i,2}$ slots into $\bz$. 
Let $(T, \bz) \in \decSSYT_P(2^a) \times \decSSYT_P(k) \setminus \F_P(2^a, k)$. If $i$ is the minimal index such that $T_{i,2}\bz$ is a powersum word, let $i_{T, \bz} = i$. Otherwise, let $i_{T, \bz}$ be the smallest index such that $T_{i_{T, \bz}, 1}T_{i_{T, \bz}, 2}$ slots into $\bz$.
Define the map 
\begin{equation}
\varphi_{a,k}: \decSSYT_P(2^a) \times \decSSYT_P(k) \to \decSSYT_P(2^a) \times \decSSYT_P(k) \sqcup \decSSYT_P(2^{a-1}) \times \decSSYT_P(k+2)
\end{equation}
by
\begin{equation}
\varphi_{a,k}(T, \bz) = \begin{cases}
(T, \bz) & \text{if } (T, \bz) \in \F_P(2^a, k), \\
(\ej_{i_{T, \bz}}(T), \slot(T_{i_{T, \bz},1}T_{i_{T, \bz},2}, \bz)) & \text{otherwise.}
\end{cases}
\end{equation}
\end{defin}

\begin{lemma}
\label{lem phi inj}
The map $\varphi_{a,k}$ is an injection.
\end{lemma}
\begin{proof}
Observe that $\varphi_{a,k}$ is the identity map when restricted to $F_P(2^a, k)$, so it suffices to show that the restriction of $\varphi_{a,k}$ to $\decSSYT_P(2^a) \times \decSSYT_P(k) \setminus F_P(2^a, k)$ is an injection.

Suppose $(S, \by) \in \decSSYT_P(2^{a-1}) \times \decSSYT_P(k+2)$ is in the image of $\varphi_{a,k}$. By Lemma~\ref{lem slot inj}, there are unique $q_1 q_2 \in \decSSYT_P(2)$ and $\bz \in \decSSYT_P(k)$ such that $\by = \slot(q_1q_2, \bz)$. As $(S, \by)$ is in the image of $\varphi_{a,k}$, there is some $T \in \decSSYT_P(2^a)$ and $i \in [a]$ such that $(S, q_1q_2) = \Ej_i(T)$. By Lemma~\ref{lem Ej inj}, $i$ is the unique index such that $S_{i-1,2} <_P q_1 <_P S_{i,2}$. Therefore, $(T, \bz)$ is the unique pair such that $\varphi_{a,k}(T, \bz) = (S, \by)$, so $\varphi_{a,k}$ is injective.
\end{proof}

\begin{lemma}
\label{lem phi im FP}
If $T \in \decSSYT_P(2^a)$ and $\bz \in \decSSYT_P(k)$, then 
\begin{equation}
\varphi_{a,k}(T,\bz) \in \F_P(2^a, k) \sqcup \F_P(2^{a-1}, k+2).
\end{equation}
\end{lemma}
\begin{proof}
If $T \in \F_P(2^a, k)$, then the claim follows immediately from the definition of $\varphi_{a,k}$.
Let $T \in \decSSYT_P(2^a) \times \decSSYT_P(k) \setminus F_P(2^a, k)$, and let $v_1 <_P v_2 <_P \dots <_P v_a$ and $w_1 <_P w_2 <_P \dots <_P w_a$ be the columns of $T$. 
Let $i = i_{T, \bz}$ and let $S = \ej_{i}(T)$.
Let $\by = y_1 y_2 \dots y_{k+2} = \slot(v_iw_i,\bz)$.
Suppose $v_iw_i\bz$ is a powersum word. Then $\slot(v_iw_i, \bz) = v_iw_i \bz$. 
If $w_{i-1} <_P v_i$, then $S_{i-1,1}S_{i-1,2} = v_{i-1}w_{i-1}$, $w_{i-1} <_P y_j$ for all $j \in [k+2]$, and $v_{i-1} <_P y_1$, so $S_{j,1}S_{j,2}$ does not slot into $\by$ for any $j < i$. If $w_{i-1} \dote_P v_i$, then $S_{i-1,1}S_{i-1,2} = w_{i-1}v_{i-1}$. As $P$ is \twotwo -free, $v_{i-1} <_P v_i$, $w_{i-1}$ is incomparable to $v_{i-1}$ and $v_i$, and $w_{i-1} <_P z_j$ for all $j \in [k]$, we have $v_{i-1} <_P y_j$ for all $j \in [k+2]$. Thus $S_{j,1}S_{j,2}$ does not slot into $\by$ for any $j < i$.
If $v_i <_P w_{i+1}$, $S_{i, 1}S_{i,2} = v_{i+1} w_{i+1}$, and $S_{i,1}S_{i,2}$ does not slot into $\by$. Hence, $S_{j,1}S_{j,2}$ does not slot into $\by$ for any $j \geq i$. If $v_i \dote_P w_{i+1}$, then $S_{i, 1}S_{i,2} = w_{i+1}v_{i+1}$, and $S_{j,1}S_{j,2}$ does not slot into $\by$ for any $j \geq i$. Thus, $\varphi_{a,k}(T) \in \F_P(2^{a-1}, k+2)$.

Suppose $v_iw_i$ slots into $\bz$ by satisfying condition $(2)$ of Definition~\ref{def slots}. We then have $\by = w_i z_1 v_i z_2z_3 \dots z_k$. As $w_{i-1} <_P w_i <_P z_1$, $v_i$ is incomparable to $w_i$ and $z_1$, and $P$ is \threeone -free, we have $w_{i-1} <_P v_i$ and $S_{j,1}S_{j,2} = v_jw_j$ for $j < i$. Then for each $j < i$, $S_{j,2} <_P y_t$ for all $t \in [k+2]$ and $S_{j,1} <_P y_1$. Thus, $S_{j,1}S_{j,2}$ does not slot into $\by$ for $j < i$.
By the definition of $i_{T, \bz}$, $w_{i+1}$ is comparable to $z_1$. As $v_i$ is incomparable to $w_i$ and $z_1$, $w_i <_P z_1$, and $P$ is \threeone -free, we have $z_1 <_P w_{i+1}$ and $v_i <_P w_{i+1}$. Hence, $S_{j,1}S_{j,2} = v_{j+1}w_{j+1}$ for $j \geq i$. Then for each $j \geq i$, we have $y_1 <_P y_2 <_P S_{j,2}$, $y_3 <_P S_{j,2}$, and $y_3 <_P S_{j,1}$, so $S_{j,1}S_{j,2}$ does not slot into $\by$ for $j \geq i$.

Suppose $v_iw_i$ slots into $\bz$ by satisfying conditions $(3)$ or $(4)$ of Definition~\ref{def slots}.  
Then $\by = z_1 z_2 \dots z_r q_1 q_2 z_{r+1} z_{r+2} \dots z_k$, where $\{q_1, q_2\} = \{v_i, w_i\}$ and $z_1<_P z_2 <_P \dots <_P z_r$. 
As there is no index $j$ such that $v_jw_j\bz$ is a powersum word, $w_j <_P z_t$ for all $j < i$ and $t \in [k]$, and $w_j >_P z_1$ for all $j > i$. 
As $v_iw_i$ satisfies conditions $(3)$ or $(4)$, we have $z_1 <_P w_i$. 
As $P$ is \threeone -free and $w_{i-1} <_P z_1 <_P w_i$, we have $w_{i-1} <_P v_i$. 
Hence, $S_{j,1}S_{j,2} = v_jw_j$ for all $j < i$ and, by the definition of $i$, $S_{j,2} <_P y_t$ for all $t \in [k+2]$. 
As $i$ is the minimal index such that $v_iw_i$ slots into $\bz$, if $v_{i-1} \dote_P z_1$, then $v_{i-1} <_P z_t$ for $t = 2,3,\dots,k$. 
As $P$ is \threeone -free and $w_{i-1} <_P z_1 <_P w_i$, if $v_{i-1} \dote_P z_1$, then $v_{i-1} <_P w_i$. Hence, $v_{i-1} <_P y_t$ for all $t=2,3,\dots,k+2$, and $S_{j,1}S_{j,2}$ does not slot into $\by$ for any $j < i$. From Lemma~\ref{lem Ej inj}, we have that $q_1 <_P S_{j,2}$ and $q_2 <_P S_{j,2}$ for all $j \geq i$. As at most one $\{q_1,q_2\}$ will be incomparable to any $S_{j,1}$ for $j \geq i$, we have that $S_{j,1}S_{j,2}$ does not slot into $\bz$ for any $j \geq i$. Thus, $(S, \by) \in F_P(2^{a-1},k+2)$.
\end{proof}

\begin{lemma}
\label{lem phi sur}
If $(S, \by) \in \F_P(2^{a-1}, k+2)$, then there exist $T \in \decSSYT_P(2^a)$ and $\bz \in \decSSYT_P(k)$ such that $(S, \by) = \varphi_{a,k}(T, \bz)$.
\end{lemma}
\begin{proof}
By Lemma~\ref{lem slot inj}, there is a unique pair $\ba \in \keySSYT_P(2)$ and $\bb\in \keySSYT_P(k)$ such that $\by = \slot(\ba, \bb)$. As $(S, \by) \in \F_P(2^{a-1}, k+2),$ there is some $i \in [a]$ such that $S_{j,2} <_P y_t$ for all $j < i$ and $t \in [k+2]$ and $y_1 <_P S_{j',2}$ for all $j' \geq i$. 

Suppose $(\ba,\bb)$ satisfies condition $(1)$ or $(3)$ of Definition~\ref{def slots}. Then for some $r \geq 1$, $\by = b_1b_2\dots b_{r-1} a_1 a_2 b_r b_{r+1} \dots b_k$. As $S_{i,1}S_{i,2}$ does not slot into $\by$, $a_1 <_P S_{i,2}$. Furthermore, either $a_1 <_P S_{i,1}$ and $a_2 <_P S_{i,2}$ or $a_{1} \dote_P S_{i,1}$ and $a_2 <_P S_{i,1}$. In either case, as $P$ is \twotwo -free, $a_2 <_P S_{i,2}$.
Suppose $S_{i-1, 1} \dote_P a_1$. Then, as $P$ is \threeone -free, $S_{i-1,1} \dote_P S_{i-1, 2}$, and $S_{i-1,2} <_P b_1 <_P b_2 <_P \dots <_P b_{r-1} <_P a_1$, we have $r = 1$. As $S_{i-1,1}S_{i-1,2}$ does not slot into $\by$, we have $S_{i-1,1} <_P a_2$. Let $S' \in \keySSYT_P(2^{a-1}1)$ be the $P$-tableau obtained by inserting $a_1$ into the second column of $S$ and performing a ladder swap on the right-unbalanced ladder containing $a_1$. From the previous arguments, we have that $S'_{1,2}<_P S'_{2,2} <_P \dots <_P S'_{i-1,2} <_P a_2 <_P S'_{i,2} <_P S'_{i+1,2} <_P \dots <_P S'_{a-1,2}$. Let $T \in \keySSYT_P(2^a)$ be the $P$-tableau obtained by inserting $a_2$ into the second column of $S'$. Then $(S, \by) = \varphi_{a,k}(T, \bb)$.

Suppose $(\ba, \bb)$ satisfies condition $(2)$ of Definition~\ref{def slots}.
Then $\by = a_2 b_1 a_1 b_2 b_3 \dots b_k$. As $S_{i-1,1}S_{i-1,2}$ does not slot into $\by$, we have $S_{i-1,2} <_P a_1$ and, if $S_{i-1, 1} \dote_P a_2$, then $S_{i-1,1} <_P a_1$. As $S_{i,1} S_{i,2}$ does not slot into $\by$, we have $a_2 <_P b_1 <_P S_{i,2}$. As $P$ is \threeone -free, $S_{i,1} \dote_P S_{i,2}$, and $a_1 \dote_P a_2$, we have $a_2 <_P S_{i-1,2}$ and $a_1 <_P S_{i,2}$. Hence, we can take $S' \in \keySSYT_P(2^{a-1}1)$ to be the $P$-tableau obtained by inserting $a_1$ into the second column of $S$ and performing a ladder swap on the ladder containing $a_1$. We then take $T \in \keySSYT_P(2^a)$ to be the $P$-tableau obtained by inserting $a_2$ into the second column of $S'$. Then $(S, \by) = \varphi_{a,k}(T, \bb)$.

Suppose $(\ba,\bb)$ satisfies condition $(4)$ of Definition~\ref{def slots}. Then for some $r > 1$, $\by = b_1 b_2 \dots b_{r-1} a_2a_1 b_r b_{r+1} \dots b_k$. As $S_{i-1,1}S_{i-1,2}$ does not slot into $\by$, we have $S_{i-1,2} <_P b_{r-1} <_P a_2$. As $S_{i-1,1} \dote_P S_{i-1,2}$ and $P$ is \threeone -free, $S_{i-1,1} <_P a_2$. As $S_{i,1}S_{i,2}$ does not slot into $\by$, we have $a_2 <_P S_{i,2}$. As $P$ is \threeone -free, $b_{r-1} <_P a_2 <_P S_{i,2}$, and $a_1 \dote_P b_{r-1}$, we have $a_1 <_P S_{i,2}$. Hence, we can take $S' \in \decSSYT_P(2^{a-1},1)$ to be the $P$-tableau obtained by inserting $a_1$ into the second column of $S$ and performing a ladder swap on the ladder containing $a_1$. We then have $(S, \by) = \varphi_{a,k}(T, \bb)$, where $T \in \decSSYT_P(2^a)$ is the $P$-tableau obtained by inserting $a_2$ into the second column of $S'$.
\end{proof}

\begin{defin}
\label{def mult}
Let $(S, \by) \in \F_P(2^a, k)$. Let $i \in [a]$ be the unique index such that $S_{j,2} <_P y_1$ for all $j \leq i$ and $y_1 <_P S_{j, 2}$ for all $j > i$. Observe that as $S_{i-1,1}S_{i-1,2}$ does not slot into $\by$, $S_{i-1,2} <_P y_2$, and if $S_{i-1,1} \dote_P y_1$, then $S_{i-1, 1} <_P y_2$. As $S_{i,1}S_{i,2}$ does not slot into $\by$, if $y_1 \dote_P S_{i,1}$, then $y_2 <_P S_{i,1}$, and if $y_1 <_P S_{i,1}$, then $y_2 <_P S_{i,2}$. Define $\mult_{a,k}(S, \by) \in \T_P(k,2^a)$ to be the $P$-tableau obtained by inserting $y_1$ into the second column of $S$, performing a ladder swap on the ladder containing $y_1$, inserting $y_2$ into the second column, and adding each $y_3,y_4,\dots, y_k$ to its own column. 
\end{defin}

\begin{lemma}
\label{lem mult im}
Let $P$ be a natural unit interval order, and let $(S, \by) \in \F_P(2^a, k)$. Then $\mult_{a,k}(S, \by) \in \K_P(k,2^a)$.
\end{lemma}
\begin{proof}
Let $v_1 <_P v_2 <_P \dots <_P v_{a-1}$ and $w_1 <_P w_2 <_P \dots <_P w_{a-1}$ be the columns of $S$. As $v_{i-1}w_{i-1}$ does not slot into $\by$, if $v_{i-1} \dote_P y_1$, then $v_{i-1} <_P y_t$ for all $t > 1$. Hence, $T = \mult_{a,k}(S, \by)$ is a powerful $P$-tableau. Suppose $T_{i+1,1}T_{i+1,2}$ strictly slots into $\by$. Then $y_1 <_P T_{i+1, 2}$, which means $T_{i+1,1}T_{i+1,2} = v_iw_i$. But this contradicts the fact that $(S, \by) \in \F_P(2^a, k)$. Thus, $T \in \K_P(k,2^a)$, as desired.
\end{proof}

\begin{lemma}
\label{lem mult sur}
Let $T \in \K_P(k, 2^a)$. Then
\begin{equation}
|\mult^{-1}_{a,k}(T)| = \begin{cases} a+1 & \text{if } k = 2,\\
1 & \text{otherwise.}
\end{cases}
\end{equation}
\end{lemma}
\begin{proof}
Let $A$ be the powerful array corresponding to $T$. If $k \neq 2$, there is a unique row $A_i$ of length $k$. If $k = 2$, we can take any of the $a+1$ rows to be $A_i$. We then have $T  = \mult_{a,k}(S, \by)$, where $\by = A_i$, and $S$ is the $P$-tableau obtained taking the first two columns of $T$, removing $T_{i,2}$, performing a ladder swap on the ladder containing $T_{i,1}$, and removing $T_{i,1}$.
\end{proof}

\begin{lemma}
\label{lem KP induction}
Let $k, a \in \mathbb{Z}_{\geq 0}$.
Suppose for $k \neq 2$, 
\begin{equation}
m_{2^a}^P(\bu) m_k^P(\bu) \equiv \sum_{T \, \in \, \K_P(k, 2^a)} \bu_{\colword(T)} + \sum_{R \, \in \, \K_P(k+2, 2^{a-1})}\bu_{\colword(R)} \mod I_P,
\end{equation}
and for $k = 2$,
\begin{equation}
m_{2^a}^P(\bu) m_2^P(\bu) \equiv (a+1) \sum_{T\, \in \, \K_P(2^{a+1})} \bu_{\colword(T)} + \sum_{R \, \in \, \K_P(4, 2^{a-1})}\bu_{\colword(R)} \mod I_P.
\end{equation}
Then $m_{k, 2^a}^P(\bu) \equiv \sum_{T \in \K_P(k, 2^a)} \bu_{\colword(T)} \mod I_P$.
\end{lemma}
\begin{proof}
We prove the claim by induction on $k$.
As $\K_P(k,2^a) = \decSSYT_P(k, 2^a)$ and $m_{k,2^a}^P(\bu) = \sum_{T \in \decSSYT_P(k, 2^a)} \bu_{\colword(T)}$ when $k$ is 1 or 2, $m_{k, 2^a}^P(\bu) = \sum_{T \in \K_P(k, 2^a)} \bu_{\colword(T)}$ and the claim holds for $k \in \{1,2\}$.
Recall that $m_{2^{a+1}}(\bx) m_2(\bx) = (a+2) m_{2^{a+2}}(\bx) + m_{4, 2^{a}}(\bx)$,
so 
\begin{equation}
m_{4, 2^a}^P(\bu) = m_{2^{a+1}}^P(\bu) m_2 - (a+2) m_{2^{a+2}}^P(\bu).
\end{equation}
By the assumption of the claim, and the fact that $m_{2^{a+2}}^P(\bu) \equiv \sum_{T \in \K_P(2^{a+2})} \bu_{\colword(T)}$, we have 
\begin{equation}
m_{4, 2^a}^P(\bu) \equiv \sum_{T \, \in \, \K_P(4, 2^a)} \bu_{\colword(T)} \mod I_P.
\end{equation}

Suppose $k \not\in \{1,2,4\}$ and that $m_{k-2, 2^{a+1}}^P(\bu) \equiv \sum_{T \, \in \, \K_P(k-2,2^{a+1})} \bu_{\colword(T)} \mod I_P$.
Recall that $m_{2^{a+1}}(\bx) m_{k-2}(\bx) = m_{k-2, 2^{a+1}}(\bx) + m_{k, 2^a}(\bx)$. Hence,
\begin{align}
m_{k, 2^a}^P(\bu) &= m_{2^{a+1}}^P(\bu) m_{k-2}^P(\bu) - m_{k-2, 2^{a+1}}^P(\bu) \\
				 &\equiv \left(\sum_{T \, \in \, \K_P(k-2, 2^{a+1})} \bu_{\colword(T)} + \sum_{R \, \in \, \K_P(k, 2^{a})}\bu_{\colword(R)}\right) \\
				 & \hspace{2in} - \sum_{T \, \in \, \K_P(k-2, 2^{a+1})} \bu_{\colword(T)} \\
&\equiv \sum_{R \, \in \, \K_P(k, 2^{a})}\bu_{\colword(R)} \mod I_P.
\end{align} 
Therefore, $m_{k, 2^a}^P(\bu) \equiv \sum_{R \in \K_P(k,2^a)} \bu_{\colword(R)}$ modulo $I_P$ for all $k, a \in \mathbb{Z}_{\geq 0}$, as desired.
\end{proof}
\section{Positivity of the Path Graph}
\label{pathSection}
Recall from Section~2.7 that $P_n$ is the natural unit interval order on $[n]$ such that $\inc(P_n)$ is a path with $n$ vertices.
In this section, we give a proof of Theorem \ref{dumbbellqPositivityTheorem}, showing that $\keySYT_{P_n}(\lambda) = \decSYT_{P_n}(\lambda)$.
To begin, we observe some properties of powersum permutations and powerful arrays for $P_n$. Write $\mathcal{P}(S)$ for the set of powersum permutations for $P_n$ of a set $S \subseteq [n]$, and write $\injDecArray_{P_n}(\alpha)$ for the set of \emph{bijective powerful arrays} of shape $\alpha$.

\begin{lemma}\label{pathOrientationLemma}
Let $\mathbf{w} = w_1w_2...w_k$ be a powersum permutation of $S \subseteq [n]$ for $P_n$. 
Then $\{w_1,w_2,...,w_k\} = \{a, a+1,a+2,...,a+k-1\}$ for some $a \in [n]$, and  
\begin{equation}
\mathbf{w} = (a)\ (a+1)\, ...\, (a+i-2)\ (a+k-1)\ (a+k-2)\ ...\ (a+i-1)
\end{equation}
for some $1 \leq i \leq k$.
\end{lemma}
\begin{proof}
Let $\mathbf{w} = w_1 w_2... w_k$ be a powersum permutation of a set $S$ for $P_n$, and let $w_i$ be the maximum entry of $\mathbf{w}$. As $\mathbf{w}$ has no $P_n$-descents and no nontrivial RL $P_n$-minima, $w_i w_{i+1} ... w_k = w_i\, (w_{i}-1)\, ...\,(w_{i} - k + i)$. Likewise, $w_{i-1}$ must be incomparable to one of $\{w_i, w_{i+1},...,w_k\}$, so $w_{i-1} = w_k - 1 = w_{i} - k + i -1$. Similarly, we have $w_1w_2...w_{i-1} = w_1\, (w_1 + 1)\, (w_1 +2)\,...\,(w_1 +i -2)$. Then, letting $a = w_1$, 
\begin{equation}
\mathbf{w} =  (a)\ (a+1)\, ...\, (a+i-2)\ (a+k-1)\ (a+k-2)\ ...\ (a+i-1),
\end{equation}
as desired.
%
\end{proof}
\begin{corollary}
\label{pathqIntCor}
Let $S = \{a, a+1,...,a+k-1\}$ be a subset of $[n]$. Then
\begin{equation}
\sum_{\mathbf{w} \in \mathcal{P}(S)} q^{\inv_{P_n}(\mathbf{w})} = [k]_q = 1 + q + q^2 + ... + q^{k-1}.
\end{equation}
\end{corollary}
\begin{proof}
By Lemma \ref{pathOrientationLemma}, a powersum permutation $\mathbf{w}$ for $P_n$ of a set $S$ is uniquely determined by the position of its maximum entry. Note that $\inv_{P_n}(\mathbf{w})$ is the number of entries $w_i$ such that $w_i -1$ appears to the right of $w_i$. If $\mathbf{w} = w_1w_2...w_k$ is a powersum permutation such that $w_j = \max(S) = a+k-1$, then $w_s - 1$ appears to the left of $w_s$ for $1 < s < j$ and $w_s - 1$ appears to the right of $w_s$ for $j \leq s < k$. Thus, $\inv_{P_n}(\mathbf{w}) = k - j$. As $j$ ranges from $1$ to $k$, 
\begin{equation}
\sum_{\mathbf{w} \in \mathcal{P}(S)} q^{\inv_{P_n}(\mathbf{w})} = 1 + q + q^2 + ... + q^{k-1} = [k]_q.
\end{equation}  
\end{proof}

\begin{lemma}\label{pathContentLemma}
Let $\alpha = (\alpha_1,...,\alpha_l) \vDash n$. If $A \in \injDecArray_{P_n}(\alpha)$, then the content of row $i$ of $A$ is $\{z_i + 1, z_i+2,..., z_i+\alpha_i\}$ where $z_i = \sum_{j=1}^{i-1} \alpha_j$.
\end{lemma}
\begin{proof}
Let $A \in \injDecArray_{P_n}(\alpha)$.
By Lemma \ref{pathOrientationLemma}, the entries of row $i$ of $A$ will be a set of the form $\{a_i + 1, a_i+2, a_i+3,...,a_i+\alpha_i\}$ for some $0 \leq a_i < n$.
Suppose for the purpose of contradiction that for some $i$, $a_i \neq z_i$. 
Without loss of generality, let $i$ be the minimal such row index. 
Then $z_i + 1$ must appear in some row $r > i$.
Suppose $A_{r,t} = z_i+1.$ If $t \leq \alpha_i$, then $A_{r, t} < A_{i,t}$ in the natural order on $[n]$, which contradicts the condition that columns of powerful arrays are increasing in $P$. If $t > \alpha_i$, then $A_{i, \alpha_i} < A_{r,t},$ which contradicts the condition that $A_{i, \alpha_i} <_P A_{s,u}$ for all $s > i$ and $u > \alpha_i$.
Therefore, if $A \in \injDecArray_{P_n}(\alpha)$, then the content of row $i$ is $\{z_i+1,z_i + 2,..., z_i + \alpha_i\}$ as desired.
\end{proof}

\begin{defin} Let $\alpha = (\alpha_1, \alpha_2,..., \alpha_l) \vDash n$ and $A \in \injDecArray_{P_n}(\alpha)$. Let $\mathbf{w}^{(1)},\mathbf{w}^{(2)},..., \mathbf{w}^{(l)}$ be the rows of $A$. Define the \emph{peak vector} $\peak(A) \in [\alpha_1] \times [\alpha_2] \times ... \times [\alpha_l]$ by $\peak(A) = (r_1, r_2,..., r_l)$, where $r_i$ is the index such that $\mathbf{w}^{(i)}_{r_i} = \max(\mathbf{w}^{(i)})$. Write $\Peak(\alpha)$ for the image of $\injDecArray_{P_n}(\alpha)$ under $\peak$.
\end{defin}

\begin{lemma}
\label{pathPeakLemma}
Let $\alpha = (\alpha_1, \alpha_2,...,\alpha_l) \vDash n$. Then the peak vector map
\begin{equation}
\peak: \injDecArray_{P_n}(\alpha) \to \Peak(\alpha)
\end{equation}
is a bijection. Furthermore, $\Peak(\alpha)$ consists of vectors $(r_1, r_2,...,r_l)$ such that, for all $1 < i \leq l$,
\begin{enumerate}
\item if $r_i = 1$, then $r_{i-1} \neq \min\{\alpha_i, \alpha_{i-1}\},$ and
\item if $r_i \neq 1$, then $r_{i-1} \neq 1$.
\end{enumerate}
\end{lemma}
\begin{proof}
As $\Peak(\alpha)$ is defined to be the image of $\injDecArray_{P_n}(\alpha)$ under the peak vector map, $\peak$ is surjective.
By Lemma \ref{pathContentLemma}, the entries of each row of $A \in \injDecArray_{P_n}(\alpha)$ are determined by $\alpha$, and by Lemma \ref{pathOrientationLemma}, each row is determined by the position of its maximum entry. Therefore, $\peak$ is an injection.
Let $(r_1, r_2,...,r_l) \in [\alpha_1] \times [\alpha_2] \times ... \times [\alpha_l]$, and let $\mathbf{w}^{(1)}, \mathbf{w}^{(2)},..., \mathbf{w}^{(l)}$ be powersum permutations of length $\alpha_1, \alpha_2,..., \alpha_l$ with content satisfying Lemma \ref{pathContentLemma} and maximum elements in positions $r_1,r_2,..., r_l$. Let $A$ be the array of shape $\row(\alpha)$ such that entry $A_{s,t} = \mathbf{w}_t^{(s)}$.

Fix $1 < i \leq l$. 
By Lemma~\ref{pathContentLemma}, for $s > i-1$ and $t > \alpha_{i-1}$, one knows $A_{i-1, \alpha_{i-1}} \not<_{P_n} A_{s,t}$ only if $A_{i-1, \alpha_{i-1}} = \max(\mathbf{w}^{(i-1)})$ and $A_{s,t} = A_{i-1, \alpha_{i-1}} + 1 = \min(\mathbf{w}^{(i)})$.
If $A_{i-1, \alpha_{i-1}} = \max(\mathbf{w}^{(i-1)})$, then $r_{i-1} = \alpha_{i-1}$, and if $r_i = 1,$ then $\min(\mathbf{w}^{(i)}) = \mathbf{w}^{(i)}_{\alpha_i}$. Thus, if $r_{i-1} = \alpha_{i-1}$, $r_i = 1$, and $\alpha_i > \alpha_{i-1}$, then $A_{i-1, \alpha_{i-1}} \not<_{P_n} A_{i, \alpha_i}$.
Conversely, if $r_{i-1} \neq \alpha_{i-1}$, $r_i \neq 1,$ or $\alpha_i \leq \alpha_{i-1}$, then $A_{i-1, \alpha_{i-1}} <_{P_n} A_{s, t}$ for all $s > i-1$ and $t > \alpha_{i-1}$.  

By Lemma~\ref{pathContentLemma}, if $A_{i-1, j} \not<_{P_n} A_{i, j}$ for some $j$, then $r_{i-1} = j$ and $A_{i,j} = \min(\mathbf{w}^{(i)})$. As $\min(\mathbf{w}^{(i)}) = A_{i,1}$ whenever $r_i \neq 1$, if $r_{i-1} = 1$ and $r_i \neq 1$, then $A_{i-1,1} \not<_{P_n} A_{i, 1}$. If $r_i = 1$, then $\min(\mathbf{w}^{(i)}) = A_{i, \alpha_i}$, so $A_{i-1, \alpha_i} \not<_{P_n} A_{i, \alpha_i}$ if $r_{i-1} = \alpha_i$. Therefore, if $r_i = 1$ and $r_{i-1} \neq \alpha_i$ or if $r_i \neq 1$ and $r_{i-1} \neq 1$, then $A_{i-1,j} <_{P_n} A_{i, j}$ for all $j$. Hence, $A$ is a powerful array if $r_{i-1} \neq \min\{\alpha_{i-1}, \alpha_i\}$ when $r_i = 1$ and $r_{i-1} \neq 1$ when $r_i \neq 1$ for all $1 < i \leq l$.
\end{proof}

By Lemma~\ref{pathPeakLemma}, we can identify elements of $\injDecArray_{P_n}(\alpha)$ with their peak vectors. In particular, for a peak vector $\mathbf{r} = (r_1,r_2,...,r_l) \in \Peak(\alpha)$, we write $\inv_{P_n}(\mathbf{r})$ for $\inv_{P_n}(A)$, where $A \in \injDecArray_{P_n}(\alpha)$ such that $\peak(A) = \mathbf{r}$.

\begin{lemma}
\label{pathPeakInvLemma}
Let $\alpha \vDash n$, $A \in \injDecArray_{P_n}(\alpha)$, and $\peak(A) = (r_1, r_2,...,r_l)$. Then 
\begin{equation}
\inv_{P_n}(A) = \sum_{i=1}^{l} (\alpha_i - r_i + b_i),
\end{equation}
where 
\begin{equation}
b_i = \begin{cases} 0 & \text{if }i < l,\, r_{i+1} = 1, \text{ and } r_i < \alpha_{i+1}, \\
0 & \text{if i = l}, \\
1 & \text{otherwise.}
\end{cases}
\end{equation}
\end{lemma}
\begin{proof}
Recall from \eqref{eqn tab inv} that $\inv_{P_n}(A)$ is the number of elements $p \in P_n$ such that $p + 1$ appears in a column to the left of $p$ in $A$. From Lemma \ref{pathContentLemma}, if $p \in P_n$ appears in row $i$ of $A$, then $p+1$ will appear in either row $i$ or row $i+1$.
Let $\mathbf{w}^{(1)}, \mathbf{w}^{(2)},..., \mathbf{w}^{(l)}$ be the rows of $A$. 
From the proof of Corollary \ref{pathqIntCor}, $\inv_{P_n}(\mathbf{w}^{(i)}) = \alpha_i - r_i$. 
Observe that $A$ has an inversion between rows $i$ and $i+1$ if $\mathbf{w}_{r_i}^{(i)}$ appears to the right of $\mathbf{w}_{r_i}^{(i)}+1$. 
If $r_{i+1} \neq 1$, then $\mathbf{w}_{r_i}^{(i)}+1 = \mathbf{w}^{(i+1)}_1$, which means $\mathbf{w}_{r_i}^{(i)}+1$ appears to the left of $\mathbf{w}_{r_i}^{(i)}$ for all values of $r_i$. 
If $r_{i+1} = 1$, then $\mathbf{w}_{r_i}^{(i)}+1 = \mathbf{w}_{\alpha_{i+1}}^{(i+1)}$, which means $\mathbf{w}_{r_i}^{(i)}+1$ appears to the left of $\mathbf{w}_{r_i}^{(i)}$ if $r_i > \alpha_{i+1}$. Therefore, $b_i$ is the number of inversions between rows $i$ and $i+1$, and $\inv_{P_n}(A) = \sum_{i=1}^l (\alpha_i - r_i + b_i)$.
\end{proof}

\begin{lemma}
\label{pathInjDecArrayqInv}
Let $\alpha = (\alpha_1, \alpha_2,..., \alpha_l) \vDash n$. Then 
\begin{equation}
\sum_{A\, \in\, \injDecArray_{P_n}(\alpha)} q^{\inv_{P_n}(A)} = q^{l-1}[\alpha_l]_q\prod_{i=1}^{l-1} [\alpha_i - 1]_q.
\end{equation}
\end{lemma}
\begin{proof}
Let $p_\alpha(q) = \sum_{A \in \injDecArray_{P_n}(\alpha)} q^{\inv_{P_n}(A)}$. We will prove the recurrence
\begin{equation}
p_{\alpha}(q) = \begin{cases} q[\alpha_1-1]_q p_{\hat{\alpha}}(q) &\text{if } \ell(\alpha) > 1, \\
[\alpha_1]_q &\text{if } \ell(\alpha) = 1,
\end{cases}
\end{equation}
where $\hat{\alpha} = (\alpha_2,\alpha_3,...,\alpha_l)$.
Observe that if $\ell(\alpha) = 1$, then $\injDecArray_{P_n}(\alpha) = \mathcal{P}(P_n)$. By Corollary \ref{pathqIntCor}, we have that $p_\alpha(q) = [\alpha_1]_q$.

Suppose $\ell(\alpha) > 1$, and consider the forgetful map $\varphi: \Peak(\alpha) \to \Peak(\hat{\alpha})$ that sends $(r_1,r_2,...,r_l)$ to $(r_2,r_3,...,r_l)$. Observe that if $\Peak(\alpha)$ is non-empty, then $\varphi$ is a surjection. To show that $p_\alpha(q) = q[\alpha_1 - 1]_qp_{\hat{\alpha}}(q)$, it suffices to show that if $\peak(\alpha)$ is non-empty, then 
\begin{equation}
\{\inv_{P_n}(\mathbf{r}) \mid \mathbf{r} \in \varphi^{-1}(\hat{\mathbf{r}})\} = \{1 + \inv_{P_{n'}}(\hat{\mathbf{r}}), 2 + \inv_{P_{n'}}(\hat{\mathbf{r}}),..., \alpha_1 - 1 + \inv_{P_{n'}}(\hat{\mathbf{r}})\}
\end{equation}
for all $\hat{\mathbf{r}} \in \Peak(\hat{\alpha})$, where $n' = n - \alpha_1$.
From Lemma \ref{pathPeakInvLemma}, if $\mathbf{r} = (r_1,r_2,...,r_l) \in \Peak(\alpha)$, then
\begin{equation}
\inv_{P_n}(\mathbf{r}) = \sum_{i=1}^{l} \alpha_i - r_i + b_i = \alpha_1 - r_1 + b_1 + \inv_{P_{n'}}(\varphi(\mathbf{r})),
\end{equation}
where
\begin{equation}
b_i = \begin{cases} 0 & \text{if }i < l,\, r_{i+1} = 1, \text{ and } r_i < \alpha_{i+1}, \\
1 & \text{otherwise.}
\end{cases}
\end{equation}
Thus, it suffices to show that for $\hat{\mathbf{r}} = (r_2,r_3,...,r_l) \in \Peak(\hat{\alpha})$,
\begin{equation}
\{\alpha_1 - r_1 + b_1 \mid (r_1,r_2,...,r_l) \in \varphi^{-1}(\hat{\mathbf{r}})\} = \{1,2,...,\alpha_1-1\}.
\end{equation}
If $r_2 \neq 1$, then $b_1 = 1$ for all values of $r_1$. As $r_1$ ranges from $2$ to $\alpha_1$,
\begin{align}
\{\alpha_1 - r_1 + b_1 \mid \mathbf{r} \in \varphi^{-1}(\hat{\mathbf{r}})\} = \{\alpha_1 - r_1 + 1 \mid r_1 = 2,3,...,\alpha_1\} = \{1,2,...,\alpha_1 - 1\}. 
\end{align}
If $r_2 = 1$ and $\alpha_1 \leq \alpha_2,$ then $r_1 = 1,2,...,\alpha_1-1$ and $b_1 = 0$ for all values of $r_1$. Thus,
\begin{align}
\{\alpha_1 - r_1 + b_1 \mid \mathbf{r} \in \varphi^{-1}(\hat{\mathbf{r}})\} = \{\alpha_1 - r_1 \mid r_1 = 1,2,...,\alpha_1-1\} = \{1,2,...,\alpha_1 - 1\}. 
\end{align}
If $r_2 = 1$ and $\alpha_1 > \alpha_2$, then $r_1 \in [\alpha_1] \setminus \{\alpha_2\}$, $b_1 = 0$ if $r_1 < \alpha_2$, and $b_1 = 1$ if $r_1 > \alpha_2$. Thus,
\begin{align}
\{\alpha_1 &- r_1 + b_1 \mid \mathbf{r} \in \varphi^{-1}(\hat{\mathbf{r}})\} \\
&= \{\alpha_1 - r_1 \mid r_1 = 1,2,...,\alpha_2 - 1\} \cup \{\alpha_1 - r_1 + 1 \mid r_1 = \alpha_2 + 1,\alpha_2 + 2,...,\alpha_1\} \\
&= \{\alpha_1 - \alpha_2 + 1, \alpha_1 -\alpha_2 + 2, ..., \alpha_1 - 1\} \cup \{1,2,...,\alpha_1 - \alpha_2\} \\
&= \{1,2,...,\alpha_1-1\}.
\end{align}
Therefore $p_\alpha(q) = q[\alpha_1-1]_qp_{\hat{\alpha}}(q)$ as desired.
\end{proof}

\begin{theorem}
\label{BPA path theorem}
With the notation above, we have 
\begin{equation}
X_{\inc(P_n)}(\mathbf{x}, q) = \sum_{\alpha \vDash n} \sum_{A \,\in\, \injDecArray_{P_n}(\alpha)} q^{\inv_{P_n}(A)} e_{\text{\rm sort}(\alpha)}(\mathbf{x}).
\end{equation}
\end{theorem}
\begin{proof}
From Theorem \ref{pathSWqFormulaThm},
\begin{equation}
X_{\inc(P_n)}(\mathbf{x}, q) = \sum_{\alpha \vDash n} q^{\ell(\alpha) - 1} [\alpha_{\ell(\alpha)}]_q \prod_{i=1}^{\ell(\alpha) - 1} [\alpha_i - 1]_q e_{\text{\rm sort}(\alpha)}(\mathbf{x}).
\end{equation}
Then by Lemma~\ref{pathInjDecArrayqInv}, 
\begin{equation}
X_{\inc(P_n)}(\mathbf{x}, q) = \sum_{\alpha \vDash n} \sum_{A \,\in\, \injDecArray_{P_n}(\alpha)} q^{\inv_{P_n}(A)} e_{\text{\rm sort}(\alpha)}(\mathbf{x}).
\end{equation}
\end{proof}

\begin{theorem}[Theorem~\ref{dumbbellqPositivityTheorem}]
Let $P_n$ be the natural unit interval order on $[n]$ such that $\inc(P_n)$ is a path. Then for all partitions $\lambda \vdash n$,
\begin{equation}
c_\lambda^{P_n}(q) = \sum_{T\, \in\, \decSYT_{P_n}(\lambda)} q^{\inv_{P_n}(T)}.
\end{equation}
\end{theorem}
\begin{proof}
Recall from Definition~\ref{def pow tabx} and Lemma~\ref{lem tab injective} that $\decSYT_{P_n}(\lambda)$ is the image of $\sqcup_{\sort(\alpha) = \lambda} \injDecArray_{P_n}(\alpha)$ under the bijection $\tab$. Thus, by Theorem~\ref{BPA path theorem}, 
\begin{equation}
X_{\inc(P_n)}(\mathbf{x}, q) = \sum_{\alpha\, \vDash\, n} \sum_{A\, \in\, \injDecArray_{P_n}(\alpha)} q^{\inv_{P_n}(A)} e_{\sort(\alpha)}(\mathbf{x}) = \sum_{\lambda\, \vdash\, n} \left(\sum_{T\, \in\, \decSYT_{P_n}(\lambda)} q^{\inv_{P_n}(T)} \right) e_{\lambda}(\mathbf{x}),
\end{equation} 
as desired.
\end{proof}
A natural goal is to extend Theorem \ref{dumbbellqPositivityTheorem} to $K$-chains.
Towards this goal, computational experiments suggest that $\keySYT_P(\lambda) = \decSYT_P(\lambda)$ for all $\lambda \vdash |P|$ when $\inc(P)$ is two complete graphs connected by a path. In the language of Definition \ref{def K chain}, this gives the following conjecture.
\begin{conjecture} 
\label{conj barbell powerful}
Let $a, b \geq 2$ be positive integers. Then for $\gamma = (a,2,2,...,2, b)$, 
\begin{equation}
X_{K_\gamma}(\mathbf{x}, q) = \sum_{\lambda \vdash |K_\gamma|} \sum_{T \in\, \decSYT_P(\lambda)} q^{\inv_{K_\gamma}(T)} e_\lambda(\mathbf{x}).
\end{equation}
\end{conjecture}
\noindent
See Figure~\ref{fig dumbbell} for the $K$-chain associated to $\gamma = (4,2,2,2,3)$.
For general $K$-chains, powerful $P$-tableaux do not serve as a combinatorial interpretation of $c_\lambda^P(q)$.

\begin{figure}[h]
\caption{The $K$-chain $K_{(4,2,2,2,3)}$}
\label{fig dumbbell}
\[\begin{tikzcd}
	1 & 2 & 3 & 4 & 5 & 6 & 7 & 8 & 9
	\arrow[no head, from=1-1, to=1-2]
	\arrow[curve={height=-8pt}, no head, from=1-1, to=1-3]
	\arrow[no head, from=1-2, to=1-3]
	\arrow[curve={height=-8pt}, no head, from=1-2, to=1-4]
	\arrow[no head, from=1-3, to=1-4]
	\arrow[curve={height=20pt}, no head, from=1-4, to=1-1]
	\arrow[no head, from=1-4, to=1-5]
	\arrow[no head, from=1-5, to=1-6]
	\arrow[no head, from=1-6, to=1-7]
	\arrow[no head, from=1-7, to=1-8]
	\arrow[curve={height=-8pt}, no head, from=1-7, to=1-9]
	\arrow[no head, from=1-8, to=1-9]
\end{tikzcd}\]
\end{figure}
\section*{Acknowledgements}
The author would like to thank his advisor Sara Billey for carefully reading this manuscript and offering many helpful suggestions. The author would also like to thank Jonah Blasiak, Holden Eriksson, Jim Haglund, Nathan Lesnevich, Ricky Liu, Martha Precup, John Shareshian, and Foster Tom for helpful conversations.

In a previous version of this paper, we conjectured that powerful $P$-tableaux give an upper bound to $c_\lambda^P(q)$. The author is very grateful to an anonymous reviewer for finding a counterexample to this conjecture.

\appendix

\section{}
\label{Appendix}
\ytableausetup{boxsize= 1em}

In this appendix, we apply the results of this paper to the poset from Example~\ref{ex poset}. Additionally, we give an example of a natural unit interval order $P$ for which the polynomial $\sum_{T \in \decSYT_P(\lambda)} q^{\inv_P(T)} - c_\lambda^P(q)$ does not have non-negative coefficients, giving a counterexample to a conjecture from a previous version of this paper.

\subsection{Examples and Data}
\label{append data}
Let $P$ be the poset from Example~\ref{ex poset}, shown again in Figure~\ref{fig 5 elem poset}. In the language of reverse Hessenberg functions, we have $P = P_{(0, 0, 1, 1, 3)}.$ The chromatic symmetric function $X_{\inc(P)}(\mathbf{x}, q)$ is given by
\begin{align}
X_{\inc(P)}(\mathbf{x}, q) &=
\left(q^{5} + 4 q^{4} + 7 q^{3} + 7 q^{2} + 4 q + 1\right) s_{1,1,1,1,1}(\bx) \\
&+ \left(2 q^{4} + 6 q^{3} + 6 q^{2} + 2 q\right) s_{2,1,1,1}(\bx) + \left(2 q^{3} + 2 q^{2}\right) s_{2,2,1}(\bx) + \left(q^{3} + q^{2}\right) s_{3,1,1}(\bx) \\
&=
\left(q^{3} + q^{2}\right) e_{3,1,1}(\mathbf{x}) + \left(q^{3} + q^{2}\right) e_{3,2}(\mathbf{x}) + \left(2 q^{4} + 3 q^{3} + 3 q^{2} + 2 q\right) e_{4,1}(\mathbf{x})\\
& + \left(q^{5} + 2 q^{4} + 2 q^{3} + 2 q^{2} + 2 q + 1\right) e_{5}(\mathbf{x}).
\end{align}
For each partition $\lambda \vdash 5$, the coefficient $c_{\lambda}^P(q)$ of $e_\lambda(\bx)$ is bounded above by the set of powerful standard $P$-tableaux, as defined in Definition~\ref{def pow tabx}.
Figure~\ref{fig 5 elem pow tabs} shows all powerful standard $P$-tableaux, sorted by shape and inversion number, labeled with 'Yes' or 'No' to indicate whether the $P$-tableau is a strong $P$-tableau. Additionally, each $P$-tableau $T$ is labeled with $q^{\inv_P(T)}$. 
\begin{figure}[h]
\caption{A $5$-element poset}
\label{fig 5 elem poset}
\[P = \begin{tikzcd}
	& 5 \\
	2 & 3 & 4 \\
	& 1
	\arrow[from=2-1, to=1-2]
	\arrow[from=2-2, to=1-2]
	\arrow[from=3-2, to=2-2]
	\arrow[from=3-2, to=2-3]
\end{tikzcd}\]
\end{figure}
\begin{figure}
\caption{Powerful $P$-tableaux. Each $P$-tableau $T$ is labeled with 'Yes' if $T$ is a strong $P$-tableau and 'No' if $T$ is not a strong $P$-tableau. Additionally, each $P$-tableau $T$ is labeled with $q^{\inv_P(T)}.$}
\label{fig 5 elem pow tabs}
\begin{center}
\begin{tabular}{p{3cm} p{3cm} p{3cm} p{3cm}}
\\
$\begin{ytableau}
1 & 2 & 3 & 4 & 5 \\ 
\end{ytableau}$&
$\begin{ytableau}
1 & 2 & 3 & 5 & 4 \\ 
\end{ytableau}$&
$\begin{ytableau}
1 & 3 & 2 & 4 & 5 \\ 
\end{ytableau}$&
$\begin{ytableau}
1 & 2 & 5 & 4 & 3 \\ 
\end{ytableau}$\\
Yes, $q^0$ &\textcolor{red}{No}, $q^1$ &\textcolor{red}{No}, $q^1$ &\textcolor{red}{No}, $q^2$ \\[12pt]
$\begin{ytableau}
1 & 3 & 2 & 5 & 4 \\ 
\end{ytableau}$&
$\begin{ytableau}
1 & 3 & 5 & 4 & 2 \\ 
\end{ytableau}$&
$\begin{ytableau}
1 & 5 & 4 & 2 & 3 \\ 
\end{ytableau}$&
$\begin{ytableau}
1 & 5 & 4 & 3 & 2 \\ 
\end{ytableau}$\\
\textcolor{red}{No}, $q^2$ &\textcolor{red}{No}, $q^3$ &\textcolor{red}{No}, $q^3$ &\textcolor{red}{No}, $q^4$ \\[12pt]
$\begin{ytableau}
3 & 5 & 4 & 2 & 1 \\ 
\end{ytableau}$&
$\begin{ytableau}
5 & 4 & 3 & 2 & 1 \\ 
\end{ytableau}$&
$\begin{ytableau}
1 & 2 & 4 & 5 \\ 
3 \\ 
\end{ytableau}$&
$\begin{ytableau}
1 & 2 & 3 & 4 \\ 
5 \\ 
\end{ytableau}$\\
\textcolor{red}{No}, $q^4$ &Yes, $q^5$ &Yes, $q^1$ &Yes, $q^1$ \\[12pt]
$\begin{ytableau}
1 & 2 & 5 & 4 \\ 
3 \\ 
\end{ytableau}$&
$\begin{ytableau}
1 & 2 & 4 & 3 \\ 
5 \\ 
\end{ytableau}$&
$\begin{ytableau}
1 & 3 & 2 & 4 \\ 
5 \\ 
\end{ytableau}$&
$\begin{ytableau}
1 & 5 & 4 & 2 \\ 
3 \\ 
\end{ytableau}$\\
\textcolor{red}{No}, $q^2$ &Yes, $q^2$ &\textcolor{red}{No}, $q^2$ &\textcolor{red}{No}, $q^3$ \\[12pt]
$\begin{ytableau}
1 & 3 & 4 & 2 \\ 
5 \\ 
\end{ytableau}$&
$\begin{ytableau}
1 & 4 & 2 & 3 \\ 
5 \\ 
\end{ytableau}$&
$\begin{ytableau}
1 & 4 & 3 & 2 \\ 
5 \\ 
\end{ytableau}$&
$\begin{ytableau}
3 & 4 & 2 & 1 \\ 
5 \\ 
\end{ytableau}$\\
\textcolor{red}{No}, $q^3$ &Yes, $q^3$ &Yes, $q^4$ &Yes, $q^4$ \\[12pt]
$\begin{ytableau}
1 & 2 & 3 \\ 
4 & 5 \\ 
\end{ytableau}$&
$\begin{ytableau}
2 & 1 & 3 \\ 
5 & 4 \\ 
\end{ytableau}$&
$\begin{ytableau}
1 & 3 & 2 \\ 
4 & 5 \\ 
\end{ytableau}$&
$\begin{ytableau}
1 & 2 & 4 \\ 
3 \\ 
5 \\ 
\end{ytableau}$\\
Yes, $q^2$ &Yes, $q^3$ & \textcolor{red}{No}, $q^3$ &Yes, $q^2$ \\[12pt]
$\begin{ytableau}
1 & 4 & 2 \\ 
3 \\ 
5 \\ 
\end{ytableau}$\\
Yes, $q^3$ &
\end{tabular}
\end{center}
\end{figure}
Observe that, for the hook-shape partitions $\lambda = (3,1,1), (4,1)$, and $(5)$, we have $c_\lambda^P(q) = \sum_{T \in \decSYT_P(\lambda)} q^{\inv_P(T)}$, which is consistent with Lemma~\ref{hookDecomp}.
Note that the greedy partition $\lambda^{gr}(P)$ of $P$, as defined in Lemma~\ref{greedyPartThm}, is $(3,1,1)$. Hence, Theorem~\ref{dominantPositivity} shows that $c_{\lambda}^P(q) = \sum_{T \in \strongSYT_P(\lambda)}q^{\inv_P(T)}$ for $\lambda = (3,1,1)$ and $\lambda = (3,2)$. As $T = \small{\begin{ytableau} 1 & 3 & 2 \\ 4 & 5 \end{ytableau}}$ is a powerful $P$-tableau of shape $(3, 2)$ that is not a strong $P$-tableau, $c_{3,2}^P(q) \neq \sum_{T \in \decSSYT_P(3,2)} q^{\inv_P(T)}$. Furthermore, $T$ is the only powerful $P$-tableau of shape $(3,2)$ which does not satisfy the definition of $K_P(n-2,2)$ given in Definition~\ref{def KP}.

\subsection{Counterexample to Powerful $P$-Tableaux Upper Bound}
\label{append counterexamp}

In a previous version of this paper, we conjectured that powerful $P$-tableaux serve as an upper bound for the polynomial $e$-coefficients $c_\lambda^P(q)$ when $P$ is a natural unit interval order. A counterexample to this conjecture was pointed out by an anonymous reviewer. We restate the false conjecture and give details of the counterexample.

\begin{falseconjecture}
Let $P$ be a natural unit interval order on $n$ elements, and let $\lambda \vdash n$. Then the polynomial
\begin{equation}
\sum_{T\, \in\, \decSYT_P(\lambda)} q^{\inv_P(T)} - c_\lambda^P(q)
\end{equation}
has non-negative coefficients.
\end{falseconjecture} 

The smallest counterexample to this conjecture corresponds to the reverse Hessenberg function $\bm = (0, 0, 1, 1, 2, 3, 4, 6) \in \mathbb{E}_8$ and $\lambda = (4,4)$. Figure~\ref{fig counterexample graph} shows $\inc(P_\bm)$. The polynomial $e$-coefficient $c_{4,4}^{P_\bm}(q)$ is $2q^8 + 6q^7 + 8q^6 + 8q^5 + 6q^4 + 2q^3$, whereas $\sum_{T \in \decSYT_{P_\bm}(4,4)} q^{\inv_{P_\bm}(T)} = 2q^8 + 7q^7 + 9q^6 + 7q^5 + 6q^4 + 2q^3$. Hence, the coefficient of $q^5$ in $\sum_{T \in \decSYT_{P_\bm}(4,4)} q^{\inv_{P_\bm}(T)} - c_{4,4}^{P_\bm}(q)$ is negative. 

\begin{figure}[h!]
\caption{$\inc(P_{(0,0,1,1,2,3,4,6)})$}
\label{fig counterexample graph}
\[\begin{tikzcd}[ampersand replacement=\&]
	1 \& 2 \& 3 \& 4 \& 5 \& 6 \& 7 \& 8
	\arrow[no head, from=1-1, to=1-2]
	\arrow[no head, from=1-2, to=1-3]
	\arrow[curve={height=-12pt}, no head, from=1-2, to=1-4]
	\arrow[no head, from=1-3, to=1-4]
	\arrow[curve={height=-12pt}, no head, from=1-3, to=1-5]
	\arrow[no head, from=1-4, to=1-5]
	\arrow[curve={height=-12pt}, no head, from=1-4, to=1-6]
	\arrow[no head, from=1-5, to=1-6]
	\arrow[curve={height=-12pt}, no head, from=1-5, to=1-7]
	\arrow[no head, from=1-6, to=1-7]
	\arrow[no head, from=1-7, to=1-8]
\end{tikzcd}\]
\end{figure}

\clearpage
\bibliographystyle{plain}
\bibliography{mycitations}

\end{document}